\documentclass[11pt,b5paper,notitlepage]{article}
\usepackage[b5paper, margin={0.5in,0.65in}]{geometry}
\usepackage{amsmath,amscd,amssymb,amsthm,mathrsfs,amsfonts,layout,indentfirst,graphicx,caption,mathabx, stmaryrd,appendix,calc,imakeidx,upgreek} 
\usepackage[dvipsnames]{xcolor}
\usepackage{palatino}  
\usepackage{slashed} 
\usepackage{mathrsfs} 
\usepackage{extarrows} 
\usepackage{enumitem} 
\usepackage{verbatim}   

\usepackage{fancyhdr} 

\usepackage{relsize} 

\usepackage{simpler-wick}

\usepackage{tcolorbox}
\tcbuselibrary{theorems}

\usepackage{pdfrender}

\usepackage{tikz}
\usetikzlibrary{fadings}
\usetikzlibrary{patterns}
\usetikzlibrary{shadows.blur}
\usetikzlibrary{shapes}

\usepackage{tikz-cd}
\usepackage[nottoc]{tocbibind}   

\makeindex

\usepackage{lipsum}  
\let\OLDthebibliography\thebibliography
\renewcommand\thebibliography[1]{
	\OLDthebibliography{#1}
	\setlength{\parskip}{0pt}
	\setlength{\itemsep}{2pt} 
}

\allowdisplaybreaks  
\usepackage{latexsym}
\usepackage{chngcntr}
\usepackage[colorlinks,linkcolor=blue,anchorcolor=blue, linktocpage,
]{hyperref}
\hypersetup{ urlcolor=cyan,
	citecolor=[rgb]{0,0.5,0}}

\counterwithin{figure}{section}

\theoremstyle{definition}
\newtheorem{df}{Definition}[section]
\newtheorem{eg}[df]{Example}

\newtheorem{rem}[df]{Remark}
\newtheorem{obs}[df]{Observation}
\newtheorem{ass}[df]{Assumption}
\newtheorem{cv}[df]{Convention}

\newtheorem{diss}[df]{Disscusion}
\theoremstyle{plain}
\newtheorem{thm}[df]{Theorem}

\newtheorem{pp}[df]{Proposition}
\newtheorem{co}[df]{Corollary}
\newtheorem{lm}[df]{Lemma}

\usepackage{calligra}
\DeclareMathOperator{\shom}{\mathscr{H}\text{\kern -3pt {\calligra\large om}}\,}
\DeclareMathOperator{\sext}{\mathscr{E}\text{\kern -3pt {\calligra\large xt}}\,}
\DeclareMathOperator{\Rel}{\mathscr{R}\text{\kern -3pt {\calligra\large el}~}\,}
\DeclareMathOperator{\sann}{\mathscr{A}\text{\kern -3pt {\calligra\large nn}}\,}
\DeclareMathOperator{\send}{\mathscr{E}\text{\kern -3pt {\calligra\large nd}}\,}
\DeclareMathOperator{\stor}{\mathscr{T}\text{\kern -3pt {\calligra\large or}}\,}

\usepackage{aurical}
\DeclareMathOperator{\VVir}{\text{\Fontlukas V}\text{\kern -0pt {\Fontlukas\large ir}}\,}

\newcommand{\fk}{\mathfrak}
\newcommand{\mc}{\mathcal}
\newcommand{\wtd}{\widetilde}

\newcommand{\ovl}{\overline}

\newcommand{\tr}{\mathrm{t}} 

\newcommand{\End}{\mathrm{End}} 
\newcommand{\id}{\mathbf{1}}
\newcommand{\Hom}{\mathrm{Hom}}

\newcommand{\Res}{\mathrm{Res}}

\newcommand{\uni}{\mathrm{u}}

\newcommand{\Span}{\mathrm{Span}}

\newcommand{\bk}[1]{\langle {#1}\rangle}

\newcommand{\bigbk}[1]{\big\langle {#1}\big\rangle}

\newcommand{\scr}{\mathscr}

\newcommand{\im}{\mathbf{i}}
\newcommand{\Co}{\complement}

\newcommand{\sgm}{\varsigma}

\newcommand{\mbf}{\mathbf}

\newcommand{\blt}{\bullet}
\newcommand{\Vbb}{\mathbb V}
\newcommand{\Ubb}{\mathbb U}
\newcommand{\Xbb}{\mathbb X}

\newcommand{\Abb}{\mathbb A}
\newcommand{\Wbb}{\mathbb W}
\newcommand{\Mbb}{\mathbb M}
\newcommand{\Gbb}{\mathbb G}
\newcommand{\Cbb}{\mathbb C}
\newcommand{\Nbb}{\mathbb N}
\newcommand{\Zbb}{\mathbb Z}
\newcommand{\Pbb}{\mathbb P}
\newcommand{\Rbb}{\mathbb R}
\newcommand{\Ebb}{\mathbb E}
\newcommand{\Dbb}{\mathbb D}

\newcommand{\Ybb}{\mathbb Y}
\newcommand{\cbf}{\mathbf c}

\newcommand{\Ker}{\mathrm{Ker}}

\newcommand{\Sp}{\mathrm{Sp}}

\newcommand{\vbf}{\mathbf v}

\newcommand{\wbf}{\mathbf w}

\newcommand{\Sbb}{{\mathbb S}}

\newcommand{\Imag}{\mathrm{Im}}
\newcommand{\cl}{\mathrm{cl}}

\newcommand{\Mod}{\mathrm{Mod}}

\newcommand{\Aut}{\mathrm{Aut}}

\newcommand{\Ibf}{\mathbf I}
\newcommand{\Std}{\widetilde{\mc S}}

\usepackage{tipa} 

\usepackage{tipx}

\numberwithin{equation}{section}

\title{Geometric Positivity of the Fusion Products of Unitary Vertex Operator Algebra Modules}
\author{{\sc Bin Gui}
}
\date{}
\begin{document}\sloppy 
	\pagenumbering{arabic}
	\setcounter{section}{-1}


	\maketitle

\begin{abstract}
A unitary and strongly rational vertex operator algebra (VOA) $\Vbb$ is called strongly unitary if all irreducible $\Vbb$-modules are unitarizable. A strongly unitary VOA $\Vbb$ is called completely unitary if for each unitary $\Vbb$-modules $\Wbb_1,\Wbb_2$ the canonical non-degenerate Hermitian form on the fusion product $\Wbb_1\boxtimes\Wbb_2$ is positive. It is known that if $\Vbb$ is completely unitary, then the  modular category $\Mod^\uni(\Vbb)$ of unitary $\Vbb$-modules is unitary \cite{Gui19b}, and all simple VOA extensions of $\Vbb$ are automatically unitary and moreover completely unitary \cite{Gui22,CGGH23}.

In this paper, we give a geometric characterization of the positivity of the Hermitian product on $\Wbb_1\boxtimes\Wbb_2$, which helps us prove that the positivity is always true when $\Wbb_1\boxtimes\Wbb_2$ is an irreducible and unitarizable $\Vbb$-module. We give several applications: (1) We show that if $\Vbb$ is a unitary (strongly rational) holomorphic VOA with a finite cyclic unitary automorphism group $G$, and if $\Vbb^G$ is strongly unitary, then $\Vbb^G$ is completely unitary. This result applies to the cyclic permutation orbifolds of unitary holomophic VOAs. (2) We show that if $\Vbb$ is unitary and strongly rational, and if $\Ubb$ is a simple current extension which is unitarizable as a $\Vbb$-module, then $\Ubb$ is a unitary VOA.
\end{abstract}

\tableofcontents



\section{Introduction}

Let $\Vbb$ be a vertex operator algebra (VOA). Suppose that $\Vbb$ is strongly rational, i.e. CFT-type, self-dual, $C_2$-cofinite, and rational. Then the category of $\Vbb$-modules $\Mod(\Vbb)$ is a modular tensor category \cite{Hua08}. Suppose that $\Vbb$ is also unitary \cite{DL14,CKLW18}. Then it is conjectured that all irreducible $\Vbb$-modules are unitarizable, and that the modular tensor category $\Mod^\uni(\Vbb)$ of unitary $\Vbb$-modules is unitary, which means among other things that the associativity isomorphisms and the braiding isomorphisms are unitary. As pointed out in \cite{Kir98}, understanding the unitarity of $\Mod^\uni(\Vbb)$ is closely related to the problem of proving the (projective) unitarity of the connections on conformal block bundles over the moduli spaces of compact Riemann surfaces. The study of this problem has attracted algebraic geometers due to its connection with Hodge theory \cite{Ram09,Loo09,Bel12,Loo21}. (See also \cite{BDH17} for an approach to the unitarity of conformal blocks connections from the operator algebraic viewpoint.)

It should be emphasized that the unitary structure on $\Mod^\uni(\Vbb)$ or on the conformal block bundles should not be an arbitrary one. In other words, one should not be content with proving the existence of a unitary structure. One should also show that this unitary structure is canonical, and that its definition is general and is not based on examples.

In fact, when all $\Vbb$-modules are unitarizable (i.e. when $\Vbb$ is \textbf{strongly unitary}), a canonical non-degenerate Hermitian structure on $\Mod^\uni(\Vbb)$ was introduced by the author in \cite{Gui19b}. This is the correct Hermitian structure in the following sense: suppose that one can show that this Hermitian structure is positive (i.e. unitary, in this case we call $\Vbb$ \textbf{completely unitary}), then there is a 1-1 correspondence between simple unitary VOA extensions of $\Vbb$ and normalized haploid commutative $C^*$-Frobenius algebras with trivial twists in the unitary modular tensor category $\Mod^\uni(\Vbb)$ \cite{Gui22}. (As one would expect, this correspondence could be generalized to non-local unitary extensions of $\Vbb$ appearing e.g. in \emph{unitary} boundary conformal field theory, and to \emph{unitary}  full conformal field theory.) Based on this fact, Carpi-Gaudio-Giorgetti-Hillier proved in \cite[Thm. 4.7]{CGGH23} the remarkable result that \emph{every simple VOA extension of $\Vbb$ has a unique unitary VOA structure extending that of $\Vbb$}. For instance, all unitary affine VOAs are completely unitary \cite{Gui19b,Gui19c,Ten19c}. So their simple VOA extensions are unitary.

To summarize, the correct Hermitian structure or even unitary structure on $\Mod^\uni(\Vbb)$ should be the one that is compatible with the unitary structures of VOA extensions (and more generally, non-local extensions) of $\Vbb$.

Let us be more explicit about this canonical Hermitian structure. Let $\Vbb$ be strongly unitary. For each unitary $\Vbb$-modules $\Wbb_1,\Wbb_2$, since the fusion product $\Wbb_1\boxtimes\Wbb_2$ \cite{HL95a,HL95b,HL95c,Hua95} is unitarizable, we can choose non-degenerate Hermitian products on the multiplicity spaces (which are dual spaces of intertwining operators of $\Vbb$) in the irreducible decomposition of $\Wbb_1\boxtimes\Wbb_2$. If these Hermitian products are positive, then $\Wbb_1\boxtimes\Wbb_2$ is a unitary $\Vbb$-module, and we say that the fusion of $\Wbb_1$ and $\Wbb_2$ is \textbf{algebraically positive} (cf. Def. \ref{lb84}). This gives a fusion bifunctor $\boxtimes$ of $\Mod^\uni(\Vbb)$ (but not just of the category $\Mod(\Vbb)$ of $\Vbb$-modules). With this bifunctor, the modular tensor category $\Mod^\uni(\Vbb)$ is unitary \cite{Gui19b}. If the fusion products of all unitary $\Vbb$-modules are algebraically positive, then we say that $\Vbb$ is \textbf{completely unitary}.

Many unitary and strongly rational VOAs have already been proved to be completely unitary. Unitary holomorphic VOAs (e.g. moonshine VOA) are clearly completely unitary since they have only one irreducible. The following are completely unitary due to \cite{Gui19b,Gui19c,Ten19c,Gui22,Gui20,CGGH23}:
\begin{gather}\label{eq99}
\begin{gathered}
\text{Even lattice VOAs, unitary affine VOAs}\\
\text{Type $A,D,E$ discrete series $W$-algebras}\\
\text{(including unitary ``minimal model" Virasoro VOAs)}\\
\text{Parafermion VOAs with positive integer levels}\\
\text{Their tensor products, VOA extensions, and unitary strongly-rational coset VOAs}
\end{gathered}
\end{gather}
In this list, only even lattice VOAs can be proved completely unitary using purely VOA methods, since the braiding and fusion relations for their intertwining operators are simple (cf. \cite{DL12} or \cite[Sec. A]{Gui19c}). The proofs for the other VOAs in this list rely directly or indirectly on methods from operator algebras and especially Jones' subfactor theory \cite{Jon83}.

This paper has three goals. The first one is to study complete unitarity for an important class of strongly unitary VOAs not in the above list: orbifold VOAs. Our ultimate goal is to prove that if $\Ubb$ is in the above list, if $G$ is a finite unitary automorphism group of $\Ubb$, then $\Ubb^G$ is completely unitary. In this paper, we make the first step toward this goal by proving that $\Ubb^G$ is completely unitary if $\Ubb$ is unitary and holomorphic, $G$ is cyclic, and $\Ubb^G$ is strongly unitary. In this case, although all irreducible modules of $\Ubb^G$ are simple currents, the proof of complete unitarity is not easy and requires a completely new way of thinking: we need a geometric understanding of algebraic positivity and the canonical Hermitian structure on $\Mod^\uni(\Vbb)$.

This leads to the second goal of our paper: we define a notion of \textbf{geometric positivity} for the fusion product $\Wbb_1\boxtimes\Wbb_2$ (cf. Def. \ref{lb39}) and show that it is equivalent to the algebraic positivity mentioned above  (cf. Thm. \ref{lb50} and Cor. \ref{lb72}). Recall that the algebraic positivity denotes the positivity of certain Hermitian products on the multiplicity spaces in the irreducible decomposition of $\Wbb_1\boxtimes\Wbb_2$. These Hermitian products are defined using certain fusion matrices for intertwining operators (called transport matrices in \cite{Gui19b}), see \eqref{eq69}. Advantages of this definition are that it can be easily adapted to the tensor category framework, and that it can be related to Connes fusion product for bimodules of von Neumann algebras \cite{Con94,Sau83}. (Indeed, for many of the VOAs in the above list, the proof of complete unitarity follows the line laid out by Wassermann \cite{Was98} by comparing VOA fusion with Connes fusion.) An obvious drawback, however, is that this definition does not even heuristically give a geometric explanation (in the spirit of Segal CFT \cite{Seg04}) of why that Hermitian form is expected to be positive.

To be sure, the fusion relations for intertwining operators, also called operator product expansions in physics literature, have a clear geometric picture in terms of sewing and factorization of Riemann spheres \cite{HL95a,HL95b,HL95c,Hua95}. But this geometric picture is \emph{incompatible with the complex-conjugate structures of Riemann surfaces}. For instance, the iterate of intertwining operators $\mc Y(\mc Y(w_1,z_1-z_2)w_2,z_2)$ (where $0<|z_1-z_2|<|z_2|$) corresponds to the sewing of a $3$-pointed sphere $\fk P_1=(\Pbb^1;0,z_1-z_2,\infty;\zeta,\zeta-z_1+z_2,1/\zeta)$ (where $0,z_1-z_2,\infty$ are marked points, the last three terms are local coordinates, $\zeta$ is the standard coordinate of $\Cbb$) with another one $\fk P_2=(\Pbb^1;0,z_2,\infty;\zeta,\zeta-z_2,1/\zeta)$ along the points $\infty\in\fk P_1$ and $z_2\in\fk P_2$ using their chosen local coordinates. (See Def. \ref{lb8} for details.) Neither $\fk P_1$ nor $\fk P_2$ is symmetric with respect to the reflection $z\in\Pbb^1\mapsto 1/\ovl z$. 

Instead of using the usual fusion relations for intertwining operators, we consider the following construction. Let $\zeta$ be the standard coordinate of $\Cbb$. Choose unitary $\Vbb$-modules $\Wbb_1,\Wbb_2$. Let $\Wbb_i'$ be the contragredient of $\Wbb_i$. Then there is an antiunitary map  $\Co:\Wbb_i\rightarrow\Wbb_i'$, $\Co w_i=\bk{\cdot|w_i}$. Define linear map 
\begin{gather*}
\upphi_q:\Wbb_1'\otimes\Wbb_1\otimes\Wbb_2\otimes\Wbb_2'\rightarrow\Cbb[[q]]\\
\Co m\otimes w\otimes \wtd w\otimes\Co\wtd m\mapsto \wick{\big\langle Y(2^{L_0}e^{-\im L_1}q^{L_0} \c2-,\im)\Co m\big|\Co w\big\rangle\cdot\big\langle Y(2^{L_0}e^{-\im L_1} \c2-,\im)\wtd w\big|\wtd m  \big\rangle  }
\end{gather*}
where the contraction is over a homogeneous orthonormal basis of $\Vbb$. This series of $q$ actually converges absolutely on $|q|<1$. When $0<q<1$, this expression is the conformal block (i.e. chiral correlation function) associated to $\fk P_1\#_q\fk P_2$, the sewing of
\begin{align*}
\fk P_1=\fk P_2=\Big(\Pbb^1;0,\infty,\im;\zeta,1/\zeta,\varpi=\frac{\im(\zeta-\im)}{\zeta+\im}\Big)
\end{align*}
along $\im\in \fk P_1$ and $\im\in\fk P_2$ by removing small discs and gluing small annuli around these two points by identifying $x\in\fk P_1$ and $y\in\fk P_2$ via the relation $\varpi(x)\cdot\varpi(y)=q$. (See again Def. \ref{lb8} for details of the sewing construction.) $2^{L_0}e^{-\im L_1}$ is the change-of-coordinate operator from the standard one $\zeta-\im$ to $\varpi$. $\fk P_1$ and $\fk P_2$ are symmetric with respect to the conjugate structure $z\mapsto 1/\ovl z$. They are examples of \textbf{positive trinions}, cf. Def. \ref{lb85}. Then the linear functional
\begin{gather*}
\upomega_{\Wbb_2,\fk P_2}:\Wbb_2\otimes\Wbb_2'\otimes \Vbb\mapsto \Cbb\\
\wtd w\otimes\Co\wtd m\otimes v\mapsto \bigbk{Y(2^{L_0}e^{-\im L_1}v,\im)\wtd w|\wtd m}
\end{gather*}
is a conformal block associated to $\fk P_2$, called \textbf{basic conformal block}, cf. Thm. \ref{lb28}-(a).

The deepest part of this paper is the theorem that if $\Vbb$ is a unitary and strongly-rational VOA, if $\Wbb_1,\Wbb_2$ are unitary $\Vbb$-modules, and if $\Wbb_1\boxtimes\Wbb_2$ is a unitarizable $\Vbb$-module, then the algebraic positivity is equivalent to the geometric positivity of the fusion product $\Wbb_1\boxtimes\Wbb_2$  (cf. Thm. \ref{lb50} and Cor. \ref{lb72}). By saying that $\Wbb_1\boxtimes\Wbb_2$ is \textbf{geometrically positive}, we mean that for some $0<q<1$
\begin{gather}\label{eq93}
\begin{gathered}
\sum_{k,l=1}^n\upphi_q\big(\Co w_l\otimes  w_k\otimes  \wtd w_k \otimes \Co\wtd w_l\big)\geq0\\
(\forall n\in\Zbb_+, w_1,\dots,w_n\in\Wbb_1,\wtd w_1,\dots,\wtd w_n\in\Wbb_2) 
\end{gathered}
\end{gather}
See Def. \ref{lb39}. Indeed, it can be shown that if \eqref{eq93} holds for some $0<q<1$, then it holds for all $0<q<1$. See Lem. \ref{lb40}.

Now assume that $\Wbb_1,\Wbb_2$ are simple currents. Then one can show that the geometric positivity (and hence the algebraic positivity) of $\Wbb_1,\Wbb_2$ is always true (cf. Thm. \ref{lb60}). The idea is simple: suppose that $\Wbb_1\boxtimes\Wbb_2$ is not algebraically positive. Then it must be algebraically (and hence geometrically) negative because $\Wbb_1\boxtimes\Wbb_2$ is an irreducible $\Vbb$-module. So for each non-zero $w\in\Wbb_1,\wtd w\in\Wbb_2$ and each $0<q<1$ we have
\begin{align*}
\upphi_q(\Co w\otimes w\otimes\wtd w\otimes\Co \wtd w)\leq 0.
\end{align*}
Let $q\rightarrow0$. Then the LHS above converges to $\bk{w|w}\cdot \bk{\wtd w|\wtd w}>0$ by the definition of $\upphi_q$, which is impossible. This argument has a clear geometric picture: when $q\rightarrow0$, the sewn sphere $\fk P_1\#_q\fk P_2$ ``converges" to the one-point union of $\fk P_1$ and $\fk P_2$ along $\im\in\fk P_1$ and $\im\in\fk P_2$. And the geometric positivity is easy to prove on this ``degenerate Riemann surface" (called \emph{nodal curve} by algebraic geometers). This picture of ``going to the boundary of moduli space" is well-known to algebraic geometers.

This result has some immediate consequences. Let $\Vbb=\Ubb^G$ where $\Ubb$ is a unitary (strongly-rational and) holomorphic VOA and $G$ is a finite cyclic unitary automorphism group of $\Ubb$. Assume that all $\Vbb$-modules are strongly unitary. Then, since all irreducible $\Vbb$-modules are simple currents (by \cite{vEMS20}), we conclude that $\Vbb$ is completely unitary. Therefore, by \cite{CGGH23}, all simple VOA extensions of $\Vbb$ are unitary (cf. Thm. \ref{lb61}). Indeed, one can prove the unitarity of simple current extensions without assuming that the larger VOA is holomorphic: Let $\Vbb$ be a unitary strongly-rational VOA, and let $\Ubb$ be a simple current extension of $\Vbb$. Then the irreducible unitary $\Vbb$-submodules of $\Ubb$ are closed under $\boxtimes$ and taking contragredient (up to isomorphisms), and hence their direct sums form a fusion category $\mc C$ which is unitary by the geometric and algebraic positivity of fusion of objects in $\mc C$. By \cite{CGGH23}, every Haploid rigid algebra in $\mc C$ is unitary. So $\Ubb$ is a unitary VOA extension of $\Vbb$. See Thm. \ref{lb80}.

As mentioned earlier, the theorem that algebraic positivity is equivalent to geometric positivity constitutes the deepest part of this paper. The difficulty of its proof lies in the fact that the local coordinates chosen for the sphere when dealing with algebraic positivity (i.e. the standard coordinate $\zeta$ of $\Cbb$, its translation $\zeta-z$, and $1/\zeta$) are very different from those when dealing with geometric positivity (e.g. $\varpi=\frac{\im(\zeta-\im)}{\zeta+\im}$).   It turns out that the language of intertwining operators as developed in \cite{FHL93} is not very convenient for proving this theorem. We need a more geometric language, namely, the language of \textbf{conformal blocks} for VOAs as developed in \cite{Zhu94,FB04}. 

Intertwining operators are special cases of conformal blocks. An intertwining operator at $z\in\Cbb^\times$ is a conformal block associated to the $3$-pointed sphere $(\Pbb^1;0,z,\infty;\zeta,\zeta-z,1/\zeta)$ with local coordinates $\zeta,\zeta-z,1/\zeta$. On the other hand, any conformal block associated to a $3$-pointed sphere with arbitrary local coordiantes can be expressed in terms of intertwining operators and Huang's change-of-coordinate operators \cite{Hua97}. (See Subsec. \ref{lb86} for details.) Thus, at least for $3$-pointed spheres, intertwining operators and conformal blocks are equivalent. The main difference between these two languages is the difference of pictures: Intertwining operators are thought first of all as \emph{operators}. Thus, as generalizations of vertex operators, they can be added and multiplied to form a ``generalized algebra". In this picture, people often ask what are the correct \emph{algebraic relations} between an intertwining operator and an intertwining operator, a vertex operator, or (exponentials of) some Virasoro operators. On the other hand, conformal blocks are thought of as morphisms in the cobordism category of chiral conformal field theory in the sense of Segal \cite{Seg04}. In this picture, one often thinks about the extent to which the underlying geometry of Riemann surfaces determines the conformal blocks. Thus, algebraic relations in the previous picture is replaced by uniqueness properties in this picture. Such uniqueness arguments are used extensively in this paper and especially in the proof of the equivalence of the two positivities (see e.g. Prop. \ref{lb32} or Lem. \ref{lb56}), and is much simpler than dealing with the algebraic relations of intertwining operators for the purpose of this paper.

This brings us to the third goal of this paper: we want to introduce the mathematical theory of conformal blocks to those working on vertex operator algebras, in particular to those studying problems related to the tensor categories of VOA modules. (This is also part of the goal of the lecture notes \cite{Gui23b}, and we refer the readers to \cite{Gui23b} for a more detailed explanation of the motivations of conformal blocks from the perspective of Segal CFT.) To date, conformal blocks still do not seem to be very popular among vertex algebraists. Possible reasons are that in the literature conformal blocks are usually written in the difficult language of algebraic geometry, and that interest in conformal blocks has long been closely linked to problems rooted in the discipline of algebraic geometry. In my opinion, it is certainly not true that in order to understand conformal blocks one must first comprehend the theory of schemes, stacks, and $D$-modules. As we will see in this paper and in \cite{Gui23b}, those with some basic knowledge in complex differential manifolds and holomorphic vector bundles can understand many of the key ideas in conformal blocks. Especially, in Sec. 2, we introduce the notion of conformal blocks by comparing it with familiar concepts such as intertwining operators and Li's vacuum-like vectors \cite{Li94}; products and iterates of intertwining operators are typical examples of sewing conformal blocks, as recalled in Sec. 3. We hope that this paper will convince more people that conformal blocks are a powerful tool in the study of VOAs and their representation theory, and that conformal blocks are not something to be feared.

We close this introduction by discussing some possible future directions. As mentioned earlier, it is proved in this paper (Thm. \ref{lb61}) that $\Vbb^G$ is completely unitary if $\Vbb$ is unitary and holomorphic, $G$ is cyclic, and all $\Vbb^G$-modules are unitarizable. It is expected that this theorem can be generalized to an arbitrary finite unitary automorphism group $G$ by generalizing Thm. \ref{lb60} to twisted simple currents, i.e. by proving that if $\Wbb_1,\Wbb_2$ are $G$-twisted unitary $\Vbb$-modules and if $\Wbb_1\boxtimes\Wbb_2$ is an irreducible unitarizable $G$-twisted $\Vbb$-module, then the fusion $\Wbb_1\boxtimes\Wbb_2$ is algebraically positive. To prove such a result one must first develop suitable results about conformal blocks for twisted modules.

Furthermore, we wish to weaken the assumption that $\Vbb$ is holomorphic, and prove that for every $\Vbb$ in the list \eqref{eq99} and for every unitary finite automorphism group $G$, if $\Vbb^G$ is $C_2$-cofinite (and hence completely rational by \cite{McR21})\footnote{It is expected that $\Vbb^G$ is automatically strongly rational. This is a theorem when $G$ is solvable \cite{Miy15,CM16}.}, then $\Vbb^G$ is also completely unitary. Moreover, we wish to prove as in \cite{Gui20} that $\Mod^\uni(\Vbb^G)$ is unitarily equivalent to the braided tensor category of semisimple modules of $\mc A_{\Vbb^G}$ (where $\mc A_{\Vbb^G}$ is the conformal net associated to $\Vbb^G$ via the construction of Carpi-Kawahigashi-Longo-Weiner \cite{CKLW18} or Tener \cite{Ten19a}) via the $*$-functor of ``strong-integrability" by Carpi-Weiner-Xu \cite{CWX} or the one by Tener \cite{Ten19b}. Of course, this is a difficult task. We hope that the methods and the results in \cite{Gui19c,Ten19c,Gui20,CGGH23,CT23} are helpful for the study of this problem.

Thus, the present paper can be viewed as the first of our project to study the complete unitarity of orbifold VOAs and to compare their representation categories with those of conformal nets.

\subsection*{Acknowledgment}

I would like to thank Angela Gibney for inspiring me to think about the geometric interpretation (in terms of conformal blocks) of the Hermitian/inner products on the spaces of intertwining operators defined in \cite{Gui19b}. I want to thank Florencia Orosz Hunziker for inviting me to give a talk at the \textit{Rocky Mountain Representation Theory Seminar} in the spring of 2021. I came up with the definition of geometric positivity (Def. \ref{lb39}) when preparing for this talk.  I also want to express my gratitude to Robert McRae and Yilong Wang for answering my questions about orbifold VOAs and their representation categories. I thank Ching Hung Lam for answering my questions about unitary twisted modules of lattice VOAs.

The author was supported by BMSTC and ACZSP (Grant no. Z221100002722017).

\subsection*{Notations}
$\Nbb=\{0,1,2,\dots\}$, $\Zbb_+=\{1,2,3,\dots\}$. $\Cbb=\Cbb\setminus\{0\}$. $\Pbb^1=\Cbb\cup\{\infty\}$.

Unless otherwise stated, all neighborhoods are open. The closure of a subset $A$ in a topological space is denoted by $A^\cl$. \index{A@$A^\cl=\{\text{closure of }A\}$}

If $0\leq r<R\leq+\infty$ then \index{D@$\Dbb_r=\{z\in\Cbb:|z|<r\}$, $\Dbb_r^\times=\Dbb_r\setminus\{0\}$} \index{D@$\ovl\Dbb_r=\{z\in\Cbb:|z|\leq z\}$} \index{Ar@$\Abb_r=\{z\in\Cbb:r<|z|<R\}$}
\begin{gather}\label{eq14}
\begin{gathered}
\Dbb_r=\{z\in\Cbb:|z|<r\}\qquad \ovl\Dbb_r=\Dbb_r^{\cl}=\{z\in\Cbb:|z|\leq r\}\\
\Dbb_r^\times=\Dbb_r\setminus\{0\}\qquad \Abb_{r,R}=\{z\in\Cbb:r<|z|<R\}
\end{gathered}
\end{gather}

All vector spaces are over $\Cbb$. If $W$ is a vector space and $z$ is a formal variable, we let
\begin{gather*}
\Wbb[z^{\pm1}]=\Wbb[z,z^{-1}]=\Big\{\sum_{n\in\Zbb\cap[-N,N]}w_nz^n:N\in\Zbb_+,w_n\in W~(\text{for all } n)\Big\}\\
W[[z]]=\Big\{\sum_{n\in\Nbb}w_nz^n:w_n\in W~(\text{for all } n) \Big\}\\
W((z))=\Big\{\sum_{n\in m+\Nbb}w_nz^n:m\in\Zbb, w_n\in W~(\text{for all } n)\Big\}
\end{gather*}

Riemann surfaces are not assumed to be connected. If $C$ is a Riemann surface, then $\scr O(C)$ denotes the space of holomorphic functions $f:C\rightarrow\Cbb$. We let $\scr O_C$ \index{OC@$\scr O(C),\scr O_C$} denote the trivial rank $1$ holomorphic vector bundle on $C$. If $\scr E$ is a holomorphic vector bundle on $C$, then we write
\begin{align}
\scr E(C)\equiv H^0(C,\scr E):=\{\text{The space of global holomorphic sections of $\scr E$ on $C$}\}.
\end{align}
According to this notation, $\scr O_C(C)=\scr O(C)$. More generally, if $\{x_1,x_2,\dots\}$ is a discrete subset of $C$, we define \index{HC@$H^0(C,\scr E),H^0(C,\scr E(\star x_\blt))$}
\begin{align}\label{eq4}
\begin{aligned}
&H^0(C,\scr E(\star x_\blt))=H^0(C,\scr E(\star x_1+\cdots+\star x_N))\\
=&\{\text{The space of global meromorphic sections of $\scr E$ whose poles are inside $\{x_1,x_2,\dots\}$}\}.
\end{aligned}
\end{align}
If $W$ is a vector space, then $W\otimes_\Cbb\scr O_C$ (or simply $W\otimes\scr O_C$) denotes the trivial vector bundle whose fibers are equivalent to $W$. If $\scr F$ is another holomorphic vector bundle, then
\begin{align*}
\scr E\otimes_{\scr O_C}\scr F\equiv\scr E\otimes\scr F
\end{align*}
denotes their tensor product. (Namely, the fibers (resp. transition functions) of $\scr E\otimes\scr F$ are equivalent to the tensor product over $\Cbb$ of those of $\scr E$ and $\scr F$.)

Unless otherwise stated, a vector bundle on a Riemann surface $C$ means a holomorphic vector bundle.

We let $\omega_C$ denote the holomorphic cotangent bundle. \index{zz@$\omega_C$} So if $U\subset C$ is open and $\eta\in\scr O(U)$ is a holomorphic coordinate, then elements of $\omega_C(U)$ are of the form $fd\eta$ where $f\in\scr O(U)$. If $\mu$ is another coordinate, then $fd\eta=f\partial_\mu\eta\cdot d\eta$. According to the above notations,
\begin{align}\label{eq11}
\begin{aligned}
H^0(C,\scr E\otimes\omega_C(\star x_\blt))=\{&\text{The space of global meromorphic $1$-forms on $C$}\\
&\text{with coefficients in $\scr E$ and poles in $\{x_1,x_2,\dots\}$}\}
\end{aligned}
\end{align}
whose elements are locally of the form $\sgm d\eta$ where $\eta$ is a holomorphic coordinate and $\sgm$ is a meromorphic section of $\scr E$.

\section{VOA modules and change of coordinates}\label{lb15}

\subsection{VOAs and their modules}

Throughout this article, a vertex operator algebra (VOA) $\Vbb=\bigoplus_{n\in\Nbb}\Vbb(n)$ is assumed to be $\Nbb$-graded  with conformal vector $\cbf$ and vacuum vector $\id$ satisfying $\dim\Vbb(n)<+\infty$ for all $n$. If $v\in\Vbb$, we write the vertex operator as \index{Y@$Y(v)_n$}
\begin{align*}
Y(v,z)=\sum_{n\in\Zbb} Y(v)_n z^{-n-1}
\end{align*}
$L_n:=Y(\cbf)_{n+1}$ denotes the Virasoro operators. Recall that $\Vbb$ is called of \textbf{CFT-type} if $\Vbb(0)=\Cbb\id$. Recall that $\Vbb$ is called \textbf{$C_2$-cofinite} if $\Span_\Cbb\{Y(u)_{-2}v:u,v \}$ has finite codimension in $\Vbb$. 

A $\Vbb$-module $\Wbb$ in our article always means an \textbf{ordinary $\Vbb$-module}. As above, we write the vertex operators as
\begin{align*}
Y_\Wbb(v,z)=\sum_{n\in\Zbb}Y_\Wbb(v)_n z^{-n-1}
\end{align*} 
and abbreviate $Y_\Wbb$ to $Y$ when the context is clear. We also write $Y_\Wbb(\cbf)_{n+1}$ as $L_n$. By saying that $\Wbb$ is an ordinary $\Vbb$-module, we mean that $\Wbb$ is a weak $\Vbb$-module in the sense of \cite[Sec. 2]{DLM97}, and that $L_0$ is diagonal on $\Wbb$ such that the $L_0$-eigenspace decomposition of $\Wbb$ takes the form \index{Wn@$\Wbb_{(s)},\Wbb(n)$}
\begin{align}
\Wbb=\bigoplus_{s\in\Nbb+E}\Wbb_{(s)}\qquad(\dim\Wbb_{(s)}<+\infty) \label{eq1}
\end{align}
where $E$ is a \emph{finite} subset of $\Cbb$ such that any two elements of $E$ do not differ by an integer.
\begin{df}
We say that $\Wbb$ is \textbf{$L_0$-simple} if $E$ consists of a single element. For instance, if $\Wbb$ is \textbf{simple} (also called \textbf{irreducible}, which means that $0$ and $\Wbb$ are the only subspaces of $\Wbb$ invariant under the action of $\Vbb$), then $\Wbb$ is $L_0$-simple.
\end{df}

\begin{df}
We say that a $\Vbb$-module $\Wbb$ is \textbf{finitely generated} if it is generated by finitely many elements $w_1,\dots,w_n$ (i.e., the smallest subspace invariant under $Y(v)_n$ (for all $v\in\Vbb,n\in\Zbb$) which contains $w_1,\dots,w_n$ is $\Wbb$). We say that $\Wbb$ is \textbf{semisimple} if it is a \emph{finite} direct sum of irreducible $\Vbb$-modules. 
\end{df}

\begin{rem}\label{lb83}
Let $\Wbb$ be a $\Vbb$-module. If $\Wbb$ is irreducible, then $\End_\Vbb(\Wbb)=1$ by a Schur's lemma type argument (i.e., if $A\in\End_\Vbb(\Wbb)$, then for some $s$ such that $\Wbb_{(s)}\neq 0$, $A|_{\Wbb_{(s)}}$ must have an eigenvalue $\lambda$. Then $\Ker(A-\lambda)$ must be $\Wbb$). Conversely, if $\End_\Vbb(\Wbb)=1$ and if $\Wbb$ is completely reducible, then  $\Wbb$ is clearly irreducible.
\end{rem}

\begin{rem}\label{lb10}
Semisimple $\Vbb$-modules are clearly finitely-generated since irreducible $\Vbb$-modules are so. Also, by \cite[Cor 3.16]{Hua09}, if $\Vbb$ is $C_2$-cofinite then any $\Vbb$-module (as defined in this article) is finitely-generated.
\end{rem}

\begin{rem}
Each $\Vbb$-module $\Wbb$ is a direct sum of $L_0$-simple modules. (Indeed, if we assume \eqref{eq1}, then $\Wbb$ has decomposition $\Wbb=\bigoplus_{\lambda\in E}\Wbb_\lambda$ where each $\Wbb_\lambda=\bigoplus_{n\in\Nbb}\Wbb_{(n+\lambda)}$ is a submodule.) Therefore, one can often assume for simplicity that $\Wbb$ is $L_0$-simple.
\end{rem}

The above $E$ is uniquely determined if we assume moreover that
\begin{align*}
\dim \Wbb_{(\lambda)}>0\qquad(\forall\lambda\in E)
\end{align*}
Define a diagonal operator $\wtd L_0$ \index{L0@$\wtd L_0$} (the \textbf{normalized conformal Hamiltonian}) as
\begin{align*}
\wtd L_0\big|_{\Wbb_{(n+\lambda)}}=n\qquad(\forall \lambda\in E)
\end{align*}
Then $\wtd L_0-L_0$ commutes with every $Y_\Wbb(v)_n$. Set $\Wbb(n)=\Wbb_{(n+\lambda)}$ \index{Wn@$\Wbb_s,\Wbb(n)$} if $\lambda\in E$. Then $\wtd L_0$ gives an $\Nbb$-grading
\begin{align*}
\Wbb=\bigoplus_{n\in\Nbb}\Wbb(n).
\end{align*}

\begin{rem}
By the above convention, $\wtd L_0$ and $L_0$ agree on $\Vbb$. Namely,
\begin{align*}
\Vbb(n)=\Vbb_{(n)}.
\end{align*}
More generally, if $\Wbb$ is $L_0$-simple, then $\wtd L_0-L_0$ is a constant.
\end{rem}

It is then clear that $\wtd L_0$ gives $\Wbb$ an admissible $\Vbb$-module structure, which means that for each $v\in\Vbb,n\in\Zbb$ we have
\begin{align*}
[\wtd L_0,Y(v)_n]=Y(L_0 v)_n-(n+1)Y(v)_n.
\end{align*}
because a similar relation holds for $L_0$ instead of $\wtd L_0$, and because $\wtd L_0-L_0$ commutes with $Y(v)_n$. Since clearly
\begin{align*}
\dim\Wbb(n)<+\infty\qquad(\forall n\in\Nbb),
\end{align*}
$\Wbb$ is better called a \textbf{finitely admissible $\Vbb$-module}, as in \cite{Gui23a}.

Recall that the \textbf{contragredient module $\Wbb'$} \index{W@$\Wbb'$} of a $\Vbb$-module $\Wbb$ is described as follows. As a vector space,
\begin{align*}
\Wbb'=\bigoplus_{n\in\Nbb}\Wbb(n)^*=\bigoplus_{s\in\Cbb}\Wbb_{(s)}^*
\end{align*}
where $\Wbb(n)^*$ and $\Wbb_{(s)}^*$ are respectively the dual spaces of $\Wbb(n)$ and $\Wbb_{(s)}$. For each $v\in \Vbb$, set \index{Y@$Y'(v,z)=Y(e^{zL_1}(-z^{-2})^{L_0}v,z^{-1})$, $Y'(v)_n$}
\begin{align}
Y_\Wbb'(v,z)=Y_\Wbb(\mc U(\vartheta_z)v,z^{-1})=Y_{\Wbb}(e^{zL_1}(-z^{-2})^{L_0}v,z^{-1}).\label{eq10}
\end{align}
(See Exp. \ref{lb3} for the reason of setting $\mc U(\vartheta_z)=e^{zL_1}(-z^{-2})^{L_0}$.) Then for each $w\in\Wbb,w'\in\Wbb',v\in\Vbb$, we have (in $\Cbb[z^{\pm1}]$)
\begin{align}
\bk{w,Y_{\Wbb'}(v,z)w'}=\bk{Y_{\Wbb}'(v,z)w,w'} \label{eq6}
\end{align}
where $\bk{\cdot,\cdot}$ stands for the standard pairing between $\Wbb$ and $\Wbb'$. Clearly $\Wbb''=\Wbb$. Setting
\begin{align*}
Y_\Wbb'(v)_n=\Res_{z=0}Y'(v,z)z^ndz
\end{align*}
then \eqref{eq6} means that $Y_{\Wbb'}(v)_n$ is the transpose of $Y'_\Wbb(v)_n$:
\begin{align}\label{eq13}
\bk{w,Y_{\Wbb'}(v)_nw'}=\bk{Y_{\Wbb}'(v)_nw,w'}.
\end{align}

\begin{rem}
Clearly, a $\Vbb$-module $\Wbb$ is irreducible if and only if its contragredient $\Wbb'$ is irreducible. Indeed, if $\Wbb$ has a proper submodule $\Mbb$, then $\{w'\in\Wbb':\bk{w',\Mbb}=0\}$ is a proper submodule.
\end{rem}
\begin{df}
We call $\Vbb$ \textbf{self-dual} and write $\Vbb\simeq\Vbb'$ if $\Vbb$ is isomorphic to $\Vbb'$ as a $\Vbb$-module. We call $\Vbb$ a \textbf{simple VOA} if it is irreducible as a $\Vbb$-module. 
\end{df}

\begin{rem}
If $\Vbb$ is self-dual, then $\dim\End_\Vbb(\Vbb)\leq\dim\Vbb(0)$ by Exp. \ref{lb74}. Thus $\Vbb$ is simple if $\Vbb$ is of CFT-type.
\end{rem}

The algebraic completion $\ovl\Wbb$ \index{W@$\ovl \Wbb=(\Wbb')^*$, the algebraic completion} is defined to be
\begin{align*}
\ovl \Wbb=(\Wbb')^*=\coprod_{n\in\Nbb} \Wbb(n)=\coprod_{s\in\Cbb}\Wbb_{(s)}.
\end{align*}
The action of $Y_\Wbb(v)_n$ on $\Wbb$ extends naturally to that of $\ovl\Wbb$.

\subsection{Change of coordinates}\label{lb86}

The advantage of $\wtd L_0$ over $L_0$ is that the change of coordinate operator defined by $\wtd L_0$ is a group homomorphism. The group we consider here is $\Gbb$, the subset of
\begin{align*}
\scr O_{\Cbb,0}=\Big\{f(z)=\sum_{n\in\Nbb}a_nz^n\in\Cbb[[z]]:\sum_n |a_n|r^n<+\infty\text{ for some }r>0\Big\}
\end{align*}
defined by \index{G@$\Gbb$}
\begin{align*}
\Gbb=\{\rho\in\scr O_{\Cbb,0}:\rho(0)=0,\rho'(0)\neq0\}=\Big\{\sum_n a_nz^n\in\scr O_{\Cbb,0}: a_0=0,a_1\neq 0 \Big\}.
\end{align*}
The group multiplication of $\rho_1,\rho_2\in\Gbb$ is their composition $\rho_1\circ\rho_2$.

The change of coordinate operators were introduced by Huang in \cite{Hua97}. They are defined as follows. For each $\rho\in\Gbb$, we can find $c_1,c_2,\dots\in\Cbb$  such that
\begin{align*}
\rho(z)=\rho'(0)\cdot \exp\Big(\sum_{n>0}c_nz^{n+1}\partial_z \Big)z=\rho'(0)\cdot\sum_{k\in\Nbb} \frac 1{k!}\Big(\sum_{n>0}c_nz^{n+1}\partial_z \Big)^k(z)
\end{align*}
Then we set \index{U@$\mc U(\rho),\mc U_0(\rho)$}
\begin{subequations}\label{eq94}
\begin{gather}
\mc U(\rho)=\rho'(0)^{\wtd L_0}\exp\Big(\sum_{n>0}c_n L_n\Big)\\
\mc U_0(\rho)=\rho'(0)^{L_0}\exp\Big(\sum_{n>0}c_n L_n\Big)
\end{gather}
\end{subequations}
Note that $\mc U_0(\rho)$ depends on the choice of the \textbf{argument} $\arg \rho'(0)$, but $\mc U(\rho)$ does not. Also, if $\Wbb$ is $L_0$-simple, $\mc U_0(\rho)$ equals a non-zero constant times $\mc U(\rho)$. And $\mc U(\rho)$ agrees with $\mc U_0(\rho)$ on $\Vbb$.

\begin{rem}
Although the definitions of $\mc U(\rho)$ and $\mc U_0(\rho)$ involve infinite sums, there is actually no convergence issue. Consider \index{W@$\Wbb^{\leq n}$}
\begin{align*}
\Wbb^{\leq n}=\bigoplus_{k\leq n}\Wbb(k)
\end{align*}
which is finite dimensional. Note that
\begin{align}
[\wtd L_0,L_n]=[L_0,L_n]=-nL_n. \label{eq2}
\end{align}
If $n>0$, then each $L_n$ lowers the $\wtd L_0$-weights by at least $1$, and so $L_n|_{\Wbb^{\leq n}}$ is a nilpotent operator. (The fact that $\dim\Wbb^{\leq n}<+\infty$ is not really needed here, but it certainly makes one feel safe because one can do finite-dimensional linear algebra.) So for each $w\in\Wbb^{\leq n}$, $\mc U(\rho)w$ is a finite sum. The same can be said about $\mc U_0(\rho)w$ since, when $\Wbb$ is $L_0$-semisimple, we have $\mc U_0(\rho)w=\lambda\mc U(\rho)w$ for some $\lambda\neq0$.
\end{rem}

\begin{eg}\label{lb2}
Let $\tau\in\Cbb$. Since $(z^2\partial_z)^n(z)=n!z^{n+1}$, one has
\begin{align*}
\frac{z}{1-\tau z}=\exp(\tau z^2\partial_z)z.
\end{align*}
Therefore, if $\rho(z)=z/(1-\tau z)$, then $\mc U(\rho)=\mc U_0(\rho)=e^{\tau L_1}$. 

In general, if $\rho\in\Gbb$ is a M\"obius transformation, then one can find $\lambda\neq 0$ and $\tau$ such that $\rho(z)=\frac{\lambda z}{1-\tau z}$. Then the expression $\mc U_0(\rho)=\lambda^{L_0}e^{\tau L_1}$ includes only $L_0, L_1$ among all Virasoro operators.
\end{eg}

The following two fundamental facts are due to \cite{Hua97}. (See also \cite[Chapter 6]{FB04})  They were originally stated in the special case that $\Wbb=\Vbb$, but the proofs also apply to the general case. (See for instance \cite[Sec. 10]{Gui23b}.)

\begin{thm}[{\cite[Sec. 4.2]{Hua97}}]  \label{lb1}
$\mc U$ is a representation of $\Gbb$ on $\Wbb$. Namely, $\mc U(\id)=\id$ (here $\id(z)=z$), and $\mc U(\alpha\circ\beta)=\mc U(\alpha)\circ\mc U(\beta)$ if $\alpha,\beta\in\Gbb$. In particular, $\mc U(\alpha)$ has inverse $\mc U(\alpha^{-1})$.
\end{thm}

The main idea of the proof that $\mc U$ preserves multiplication is simple: If $\alpha'(0)=\beta'(0)=1$, then one applies Campbell-Hausdorff Theorem to the subgroup \index{G+@$\Gbb_+$}
\begin{align*}
\Gbb_+=\{\rho\in\Rbb:\rho'(0)=1\}
\end{align*}
(whose Lie algebra is $\Span_\Cbb\{l_n=z^{n+1}\partial_z,n>0\}$). To prove the general case, one appeals to
\begin{align}
\exp\Big(\sum_{n\geq 1}c_nL_n \Big)\lambda^{\wtd L_0}=\lambda^{\wtd L_0}\exp\Big(\sum_{n\geq 1}c_n\lambda^n L_n \Big)\label{eq80}
\end{align}
which follows from
\begin{align}
L_n\lambda^{\wtd L_0}=\lambda^{\wtd L_0+n}L_n,
\end{align}
and hence follows from \eqref{eq2}.

Let $\alpha\in\Gbb$. Suppose that $U\subset\Cbb$ is a neighborhood of $0$ on which the power series $\alpha(t)$ converges absolutely, and $\alpha'$ vanish nowhere on $U$. Then for each $z\in U$, we can define an element of $\Gbb$:
\begin{align}
\varrho(\alpha|\id)_z(t)=\alpha(z+t)-\alpha(z)  \label{eq3}
\end{align}
If we write $\varrho(\alpha|\id)_z(t)=\sum_n a_n(z)t^n$ then each $a_n$ is holomorphic on $U$. Therefore, for each $w\in\Wbb^{\leq n}$, the expression $\mc U(\varrho(\alpha|\id)_z)w$ is a holomorphic $\Wbb^{\leq n}$-valued function of $z$.

The meaning of the symbol $\varrho(\alpha|\id)$ will be revealed in the next section.

\begin{thm}[\cite{Hua97}] \label{lb5}
Let $\Wbb$ be a $\Vbb$-module. Let $\alpha\in\Gbb$. Then for each $w\in\Wbb,w'\in\Wbb',v\in\Vbb$, the following relation holds at the level of $\Cbb((z))$:
\begin{align*}
\bk{\mc U(\alpha)Y_\Wbb(v,z)w,w'}=\left\langle Y_\Wbb\big(\mc U(\varrho(\alpha|\id)_z)v,\alpha(z)\big)\mc U(\alpha)w,w' \right\rangle
\end{align*}
\end{thm}

Instead of using this formula directly, we will only use its consequence, Prop. \ref{lb6}.

Finally, we show that $\mc U_0$ is a group representation of ``the universal cover of $\Gbb$".

\begin{co}\label{lb29}
We have $\mc U_0(\id)=\id$ if the argument of $\id'(z)=1$ is $0$. Let $\alpha,\beta\in\Gbb$. (Note that $(\alpha\circ\beta)'(0)=\alpha'(0)\beta'(0)$.) Then $\mc U_0(\alpha\circ\beta)=\mc U_0(\alpha)\mc U_0(\beta)$ if the arguments are chosen such that
\begin{align*}
\arg\big((\alpha\circ\beta)'(0)\big)=\arg\alpha'(0)+\arg\beta'(0).
\end{align*}
\end{co}

\begin{proof}
It suffices to assume that $\Wbb$ is $L_0$-simple. Then on $\Wbb$ we have $L_0=\wtd L_0+\lambda$ for some $\lambda\in\Cbb$. Then we have $((\alpha\circ\beta)'(0))^\lambda=\alpha'(0)^\lambda\beta'(0)^\lambda$, and hence, by Thm. \ref{lb1},
\begin{align*}
\mc U_0(\alpha\circ\beta)=((\alpha\circ\beta)'(0))^\lambda\mc U(\alpha\circ\beta)=\alpha'(0)^\lambda\beta'(0)^\lambda\cdot\mc U(\alpha)\mc U(\beta)=\mc U_0(\alpha)\mc U_0(\beta).
\end{align*} 
\end{proof}

\section{Introduction to conformal blocks}

In this section, we give a brief introduction to conformal blocks and their major properties. We discuss conformal blocks on arbitrary compact Riemann surfaces in general, but our main interest is in the genus-0 case. (So, for the sake of simplicity, the readers may assume that all the compact Riemann surfaces mentioned in this article are the sphere $\Pbb^1$ and their disjoint unions.)  In this case, conformal blocks can be explicitly written down in terms of VOA intertwining operators (as developed in \cite{FHL93}).

We refer the readers to \cite{FB04} for a detailed account of the theory of VOA conformal blocks. \cite{FB04} uses the language of algebraic geometry. For a complex analytic approach, see \cite{Gui23a} or \cite{Gui23b} (which contains more motivational explanations). Our presentation in this section is elementary:  readers with some basic knowledge of (complex) differential manifolds and vector bundles can read this section without difficulty. No knowledge of algebraic geometry or advanced complex geometry is required.

\subsection{VOA bundles}

Let $C$ be a Riemann surface, and let $\Vbb$ be a VOA. Our first step is to define a holomorphic vector bundle $\scr V_C^{\leq n}$ (for each $n\in\Nbb$) associated to $C$ and $\Vbb$ in terms of its transition functions. Cf. \cite{FB04}. We follow the approach of \cite[Sec. 2]{Gui23a} and \cite[Sec. 11]{Gui23b}.

Let $U\subset C$ be open, and choose $\eta\in\scr O(U)$ to be \textbf{univalent}, namely, $\eta:U\rightarrow \Cbb$ is a holomorphic injective map, and hence a biholomorphism from $U$ to $\eta(U)$. (More generally, one can choose $\eta\in\scr O(U)$ to be locally univalent, which is equivalent to that $d\eta$ is nowhere zero on $U$.) Choose another univalent $\mu\in\scr O(U)$. Then for each $p\in U$, we can define an element $\varrho(\eta|\mu)_p\in\Gbb$ \index{zz@$\varrho(\eta\lvert\mu)$} by
\begin{align}
\eta-\eta(p)=\varrho(\eta|\mu)_p\circ\big(\mu-\mu(p)\big).
\end{align}
Thus, if $U$ is an open subset of $\Cbb$ and $\mu$ is the standard coordinate $\id:z\mapsto z$, then the meaning of $\varrho(\eta|\mu)_p$ agrees with that of $\varrho(\alpha|\id)_z$ defined in \eqref{eq3}. Moreover, if we vary $p$ on $U$, then the coefficients of the power series $\eta(\eta|\mu)_p(t)$ of $t$ varies holomorphically. Therefore, when acting on $\Wbb^{\leq n}$,
\begin{align*}
p\in U\mapsto\mc U(\varrho(\eta|\mu))_p\big|_{\Wbb^{\leq n}}\qquad\text{is a holomorphic $\End_\Cbb(\Wbb^{\leq n})$-valued function} 
\end{align*}
We simply write this fact as
\begin{align*}
\mc U(\varrho(\eta|\mu))\in\End(\Wbb^{\leq n})\otimes_\Cbb\scr O(U).
\end{align*}

If $\nu\in\scr O(U)$ is univalent, then one easily checks that $\varrho(\eta|\mu)_p\circ\varrho(\mu|\nu)_p=\varrho(\eta|\nu)_p$. Thus, from Thm. \ref{lb1} we see that the following cocycle relation holds:
\begin{align*}
\mc U(\varrho(\eta|\mu)_p)\circ \mc U(\varrho(\mu|\nu)_p)=\mc U(\varrho(\eta|\nu)_p).
\end{align*}
Thus we may define a holomorphic vector bundle $\scr V^{\leq n}_C$ on $C$ whose fibers are isomorphic to $\Vbb^{\leq n}$, and whose transition functions are given by $\mc U(\varrho(\eta|\mu))$. 

More precisely, $\scr V^{\leq n}_C$ is described as follows. For each open $U\subset C$ and each univalent $\eta\in\scr O(U)$, there is a trivialization, i.e., an equivalence of vector bundles \index{U@$\mc U_\varrho(\eta)$}
\begin{align*}
\mc U_\varrho(\eta):\scr V_U^{\leq n}\xrightarrow{\simeq} \Vbb^{\leq n}\otimes_\Cbb\scr O_U
\end{align*}
whose restriction to each open subset $V\subset U$ equals $\mc U_\varrho(\eta|_U)$. If $\mu\in\scr O(U)$ is also univalent, then
\begin{align}
\mc U_\varrho(\eta)\mc U_\varrho(\mu)^{-1}=\mc U(\varrho(\eta|\mu)): \Vbb^{\leq n}\otimes\scr O_U\xrightarrow{\simeq}\Vbb^{\leq n}\otimes\scr O_U
\end{align}
is the transition function. This defines the \textbf{VOA bundle} (also called \textbf{sheaf of VOA}) $\scr V_C^{\leq n}$ \index{VC@$\scr V^{\leq n}_C,\scr V_C$}. If $m\leq n$, the inclusion map $\Vbb^{\leq m}\hookrightarrow\Vbb^{\leq n}$ makes $\scr V_C^{\leq m}$ naturally a subbundle of $\scr V_C^n$. Thus we can take the union $\scr V_C:=\bigcup_{n\in\Nbb}\scr V^{\leq n}_C$, or more precisely, the direct limit,
\begin{align*}
\scr V_C:=\varinjlim_{n\in\Nbb}\scr V^{\leq n}_C,
\end{align*}
also called a VOA bundle.

To get an idea of what's really going on in the above definition, let's consider VOA bundles on the sphere $\Pbb^1$, our main interest in this article.

\begin{eg}\label{lb3}
Let $\zeta:z\mapsto z$ be the standard coordinate of $\Pbb^1$. Then $\Pbb^1$ can be covered by two charts: $\Pbb^1=U\cup V$ where $U=\Cbb,V=\Pbb^1\setminus\{0\}$, and $\zeta\in\scr O(U),1/\zeta\in\scr O(V)$ are univalent. For each $\gamma\in\Cbb^\times=U\cap V$, we compute
\begin{align}
\varrho(1/\zeta|\zeta)_\gamma=\varrho(\zeta|1/\zeta)_{1/\gamma}=\vartheta_\gamma
\end{align}
where \index{zz@$\vartheta_\gamma(z)=(\gamma+z)^{-1}-\gamma^{-1}$}
\begin{align}
\vartheta_\gamma(z):=\frac{1}{\gamma+z}-\frac 1\gamma.
\end{align}
For each $\tau\in\Cbb$, let $\alpha_\tau(z)=\frac{z}{1-\tau z}$. Then $\vartheta_\gamma(z)=\alpha_\gamma(-\gamma^{-2}z)=-\gamma^{-2}\alpha_{-\gamma^{-1}}(z)$. Then by Exp. \ref{lb2} and Thm. \ref{lb1}, we have 
\begin{align}
\mc U_0(\varrho(1/\zeta|\zeta))_\gamma=\mc U_0(\vartheta_\gamma)=e^{\gamma L_1}(-\gamma^{-2})^{L_0}=(-\gamma^{-2})^{L_0}e^{-\gamma^{-1}L_1}
\end{align}
which equals $\mc U(\varrho(1/\zeta|\zeta))$ on $\Vbb$.

$\mc U(\varrho(1/\zeta|\zeta))$ gives the transition function of $\scr V_{\Pbb^1}^{\leq n}$. More precisely, we have trivializations of (holomorphic) vector bundles $\mc U_\varrho(\zeta):\scr V_U^{\leq n}\xrightarrow{\simeq}\Vbb^{\leq n}\otimes\scr O_U$ and $\mc U_\varrho(1/\zeta):\scr V_V^{\leq n}\xrightarrow{\simeq}\Vbb^{\leq n}\otimes\scr O_V$ related by the transition function
\begin{gather*}
\mc U_\varrho(1/\zeta)\mc U_\varrho(\zeta):\Vbb^{\leq n}\otimes\scr O_{\Cbb^\times}\xrightarrow{\simeq}\Vbb^{\leq n}\otimes\scr O_{\Cbb^\times}\\
\gamma\in\Cbb^\times\mapsto e^{\gamma L_1}(-\gamma^{-2})^{L_0}\in\End(\Vbb^{\leq n})
\end{gather*}
\end{eg}

\subsection{Definition of conformal blocks}

\begin{df}
Let $n\in\Zbb_+$, and let $C$ be a compact Riemann surface. By an \textbf{$N$-pointed compact Riemann surface}, we mean the data
\begin{align*}
\fk X=(C;x_\blt)=(C;x_1,\dots,x_N)
\end{align*}
where $x_1,\dots,x_N$ (often abbreviated to $x_\blt$) are distinct points of $C$ (called \textbf{punctures} or \textbf{marked points}).

If, in addition to the above data $(C;x_\blt)$, we associate to each marked point $x_i$ a  (holomorphic) univalent function on a neighborhood $U_i$ of $x_i$ satisfying $\eta_i(x_i)=0$ (we call such $\eta_i$ a \textbf{local coordinate at $x_i$}), we call the data
\begin{align*}
\fk X=(C;x_\blt;\eta_\blt)=(C;x_1,\dots,x_N;\eta_1,\dots,\eta_N)
\end{align*}
an \textbf{$N$-pointed compact Riemann surface with (local) coordinates}.

If $C=\Pbb^1$ (or more generally, if $C$ is a disjoint union of $\Pbb^1$), and if the local coordinate $\eta_i$ extends to a bihholomorphism $\eta_i:\Pbb^1\rightarrow\Pbb^1$, we call $\eta_i$ a \textbf{M\"obius (local) coordinate} at $x_i$. If $x_i\neq\infty$, then $\eta_i(z)=\frac{\lambda(z-x_i)}{1-\tau(z-x_i)}$ for some $\gamma,\tau\in\Cbb$ and $\gamma\neq 0$.  \hfill\qedsymbol
\end{df}

\begin{ass}\label{lb7}
Unless otherwise stated, we assume that in an $N$-pointed compact Riemann surface $(C;x_\blt)$,  each connected component of $C$ contains at least one marked point.
\end{ass}

\begin{df}
If $\fk X=(C;x_\blt;\eta_\blt)$ and $\fk Y=(C';y_\blt;\mu_\blt)$ are $N$-pointed compact Riemann surfaces with local coordinates, we say that $\fk X$ is \textbf{equivalent} to $\fk Y$ if there is a biholomorphism $\varphi:C\rightarrow C'$ such that for each $1\leq i\leq N$ we have that $\varphi(x_i)=y_i$ and that $\eta_i=\mu_i\circ\varphi$ on a neighborhood of $x_i$.
\end{df}

Let $\fk X=(C;x_\blt;\eta_\blt)$ be an $N$-pointed compact Riemann surface with local coordinates. We associate to each marked point $x_i$ a $\Vbb$-module $\Wbb_i$. We write
\begin{align*}
\Wbb_\blt=\Wbb_1\otimes\cdots\otimes \Wbb_N
\end{align*}
for simplicity. Also, $\mbf w\in\Wbb_\blt$ or $w\in\Wbb_\blt$ means an element of $\Wbb_\blt$, but \index{W@$\Wbb_\blt=\Wbb_1\otimes\cdots\otimes\Wbb_N,w_\blt=w_1\otimes\cdots\otimes w_N$}
\begin{align*}
w_\blt\in\Wbb_\blt\text{ means an element of $\Wbb_\blt$ of the form }w_1\otimes\cdots\otimes w_N.
\end{align*}
A conformal block is then an element of the dual space $\Wbb_\blt^*=(\Wbb_1\otimes\cdots\otimes\Wbb_N)^*$ invariant under the action of
\begin{align*}
H^0(C,\scr V_C\otimes\omega_C(\star x_\blt)):=\bigcup_{n\in\Nbb} H^0(C,\scr V_C^{\leq n}\otimes\omega_C(\star x_\blt))
\end{align*}
(recall \eqref{eq11} for the meaning of the notation). The action is described as follows.

Choose any $\sigma\in H^0(C,\scr V_C{^{\leq n}}\otimes\omega_C(\star x_\blt))$. Choose a neighborhood $U_i$ of $x_i$ on which $\eta_i$ is defined (and univalent), and assume that $x_j\notin U_i$ if $i\neq j$. We first define the action of $\sigma$ on $\Wbb_i$. For that purpose, it suffices to assume that $\sigma\in H^0(U,\scr V_U^{\leq n}\otimes\omega_C(\star x_i))$.
\begin{df}
Under the trivialization map $\mc U_\varrho(\eta_i)$, $\sigma$ becomes $\mc U_\varrho(\eta_i)\sigma\in H^0(U,\Vbb^{\leq n}\otimes\scr O_U(\star x_i))$, which can therefore be written as a finite sum
\begin{align*}
\mc U_\varrho(\eta_i)\sigma=\sum_k v_kf_kd\eta_i
\end{align*}
for some $v_k\in\Vbb^{\leq n}$ and $f_k\in H^0(U,\scr O_U(\star x_i))$. (Namely, $f_k$ is a meromorphic function on $U$ with possible poles at $x_i$.) Take the Laurent series expansion
\begin{align*}
f_k=\sum_{n\in\Zbb} f_{k,n}\cdot \eta_i^n\qquad(f_{k,n}\in\Cbb)
\end{align*}
Note that $f_{k,n}=0$ for sufficiently negative $n$. Then the linear action of $\sigma$ on any $w_i\in\Wbb_i$ is defined to be
\begin{align}
&\sigma\cdot w_i=\Res_{\eta_i=0}~Y(\mc U_\varrho(\eta_i)\sigma,\eta_i)w_i  \nonumber\\
=&\sum_k \Res_{\eta_i=0}~f_k\cdot Y(v_k,\eta_i)w_id\eta_i  \nonumber\\
=&\sum_k\sum_{n\in\Zbb} \Res_{z=0}~f_{k,n}z^nY(v_k,z)w_idz  \nonumber\\
=&\sum_k\sum_{n\in\Zbb} f_{k,n}Y(v_k)_nw_i   \label{eq5}
\end{align}
\end{df}

\begin{df}
The linear action of $H^0(C,\scr V_C\otimes\omega_C(\star x_\blt))$ on $\Wbb_\blt=\Wbb_1\otimes\cdots\otimes\Wbb_N$ is defined as follows. For each $\sigma\in H^0(C,\scr V_C\otimes\omega_C(\star x_\blt))$, choose $n\in\Nbb$ such that $\sigma\in H^0(C,\scr V_C^{\leq n}\otimes\omega_C(\star x_\blt))$. By linearity, it suffices to define the action of $\sigma$ on any $w_\blt=w_1\otimes\cdots\otimes w_N$. This is defined by be
\begin{align}
\sigma\cdot w_\blt=\sum_{i=1}^N w_1\otimes\cdots\otimes w_{i-1}\otimes(\sigma\cdot w_i)\otimes w_{i+1}\otimes\cdots\otimes w_N \label{eq9}
\end{align}
where $\sigma\cdot w_i$ is defined by \eqref{eq5}.
\end{df}

\begin{df}
Let $\fk X=(C;x_\blt;\eta_\blt)$ be an $N$-pointed compact Riemann surfaces with local coordinates, and associate a $\Vbb$-module $\Wbb_i$ to each marked point $x_\blt$. A linear functional $\upphi:\Wbb_\blt\rightarrow\Cbb$ vanishing on
\begin{align*}
H^0(C,\scr V_C\otimes\omega_C(\star x_\blt))\Wbb_\blt=\Span_\Cbb\big\{\sigma\cdot w_\blt:\sigma\in H^0(C,\scr V_C^{\leq n}\otimes\omega_C(\star x_\blt)),n\in\Nbb,w_\blt\in\Wbb_\blt\big\}
\end{align*}
is called a \textbf{conformal block} associated to $\fk X,\Wbb_\blt$. The space of all such elements, namely \index{TX@$\scr T_{\fk X}^*(\Wbb_\blt)$, the space of conformal blocks}
\begin{align*}
\scr T^*_{\fk X}(\Wbb_\blt):=\big(\Wbb_\blt/H^0(C,\scr V_C\otimes\omega_C(\star x_\blt))\Wbb_\blt\big)^*
\end{align*}
is called the \textbf{space of conformal blocks} associated to $\fk X$ and $\Wbb_\blt$.
\end{df}

\subsection{Examples in genus $0$}

\begin{eg}\label{lb4}
Let $\zeta$ be the standard coordinate of $\Pbb^1$. Let $N\in\Nbb$, and consider the following $(N+1)$-pointed compact Riemann surface with M\"obius coordinates
\begin{align}
\fk X=(\Pbb^1;z_1,\dots,z_N,\infty;\zeta-z_1,\dots,\zeta-z_N,1/\zeta).  \label{eq18}
\end{align}
Using the trivialization $\mc U_\varrho(\zeta):\scr V_{\Cbb}^{\leq n}\rightarrow \Vbb^{\leq n}\otimes\scr O_\Cbb$ and regarding $v\in\Vbb^{\leq n}$ as a constant section of $\Vbb^{\leq n}\otimes\scr O_\Cbb$, it is not hard to see that
\begin{align*}
&H^0(\Pbb,\scr V_{\Pbb^1}\otimes\omega_C(\star z_\blt+\star\infty))\\
=&\Span_\Cbb\{f\cdot\mc U_\varrho(\zeta)^{-1}vd\zeta:f\in H^0(\Pbb^1,\scr O_{\Pbb^1}(\star z_\blt+\star\infty)),v\in\Vbb \}.
\end{align*}
It is also clear that
\begin{align}
H^0(\Pbb^1,\scr O_{\Pbb^1}(\star z_\blt+\star\infty))=\Cbb[\zeta,(\zeta-z_1)^{-1},\dots,(\zeta-z_N)^{-1}]
\end{align}

Take power series expansions
\begin{subequations}\label{eq12}
\begin{gather}
f=\sum_{n\in\Zbb} f_{i,n} (z-z_i)^n\qquad\text{when $|z-z_i|$ is small}  \label{eq96}\\
f=\sum_{n\in\Zbb}f_{\infty,n}z^n\qquad\text{when $|z|$ is large} \label{eq97}
\end{gather}
\end{subequations}
Associate $\Wbb_i$ to $x_i$ and $\Wbb_\infty$ to $\infty$. Then the actions of
\begin{align*}
\sigma=f\cdot\mc U_\varrho(\zeta)^{-1}vd\zeta
\end{align*}
on $w_i\in\Wbb_i$ and on $w_\infty\in\Wbb_\infty$ are
\begin{subequations}
\begin{gather}
\sigma\cdot w_i=\sum_{n\in \Zbb}f_{i,n}Y(v)_nw_i \label{eq7} \\
\sigma\cdot w_\infty=-\sum_{n\in\Zbb} f_{\infty,n}Y'(v)_n w_\infty \label{eq8}
\end{gather}
\end{subequations}
(We will explain \eqref{eq8} below.) Then $\sigma(w_1\otimes\cdots\otimes w_N\otimes w_\infty)$ is defined as in \eqref{eq9}. $\scr T_{\fk X}^*(\Wbb_\blt\otimes\Wbb_\infty)$ is the set of linear functionals on $\Wbb_\blt\otimes\Wbb_\infty$ vanishing on all such $\sigma\cdot(w_\blt\otimes w_N)$. \hfill\qedsymbol
\end{eg}

\begin{proof}[Proof of \eqref{eq8}]
Since we are computing the action on $\Wbb_\infty$, we need to use the local coordinate $1/\zeta$ at $\infty$. Then, by Exp. \ref{lb3},
\begin{align*}
\mc U_\varrho(1/\zeta)_z\sigma(z)=\mc U_\varrho(1/\zeta)_z f(z)\mc U_\varrho(\zeta)^{-1}_zvd\zeta=f(z)\mc U(\varrho(1/\zeta|\zeta)_z)vdz=f(z)\mc U(\vartheta_z)vdz,
\end{align*}
which equals $f(z)e^{zL_1}(-z^{-2})^{L_0}vdz$. Recall the definition of $Y'$ in \eqref{eq10}. We compute
\begin{align*}
&\sigma\cdot w_\infty=\Res_{1/\zeta=0}~Y(\mc U_\varrho(1/\zeta)\sigma,1/\zeta)w_\infty \\
=&\Res_{1/z=0}~Y(f(z)\mc U(\vartheta_z)v,1/z)w_\infty dz=\Res_{1/z=0}~f(z)Y'(v,z)w_\infty dz
\end{align*}
Note that $\Res_{z=0}g(z)dz=g_{-1}=-\Res_{1/z=0}g(z)dz$ for all $g=\sum_{n\in\Zbb}g_nz^n\in\Cbb[[z^{\pm1}]]$. So
\begin{align*}
\sigma\cdot w_\infty=-\sum_{n\in\Zbb}f_{\infty,n}z^nY'(v,z)w_\infty dz=-\sum_{n\in\Zbb} f_{\infty,n}Y'(v)_n w_\infty.
\end{align*}
\end{proof}

\begin{rem}\label{lb14}
In Exp. \ref{lb4}, we often regard a conformal block $\upphi\in\scr T_{\fk X}^*(\Wbb_\blt\otimes\Wbb_\infty)$ as a linear map $\mc Y_\upphi:\Wbb_1\otimes\cdots\otimes\Wbb_N\rightarrow\Wbb_\infty^*$. (Note that $\Wbb_\infty^*$ equals the algebraic completion of $\Wbb'$.) Then the condition that $\upphi$ vanishes on all $\sigma\cdot(w_\blt\otimes w_N)$ is equivalent to saying that for each $v\in\Vbb$ and $f\in \Cbb[\zeta,(\zeta-z_1)^{-1},\dots,(\zeta-z_N)^{-1}]$, in terms of the expansions \eqref{eq12}, the following \textbf{Jacobi identity} holds in $\Wbb_\infty^*=\ovl{\Wbb_\infty'}$:
\begin{align}\label{eq17}
\begin{aligned}
&\sum_{n\in\Zbb}f_{\infty,n}Y_{\Wbb_\infty'}(v)_n\mc Y_{\upphi}(w_1\otimes\cdots\otimes w_N)\\
=&\sum_{i=1}^N\sum_{n\in\Zbb} f_{i,n} \mc Y_\upphi(w_1\otimes\cdots\otimes Y_{\Wbb_i}(v)_n w_i\otimes\cdots w_N).
\end{aligned}
\end{align}
In the special case that $N=2,z_1=0,z_2=z$, we write
\begin{align*}
\upphi(w_1\otimes w_2\otimes w_\infty)=\bk{\mc Y_\upphi(w_2,z)w_1,w_\infty}
\end{align*}
and call $\mc Y_\upphi(\cdot,z):\Wbb_2\otimes\Wbb_1\rightarrow\Wbb_\infty^*$ a \textbf{type $\Wbb_\infty'\choose \Wbb_2\Wbb_1$ intertwining operator  at $z$}.
\end{rem}

\begin{eg}\label{lb30}
Let $\fk X=(\Pbb^1;0,\infty;\zeta,1/\zeta)$ where $\zeta$ is the standard coordinate of $\Cbb$. Let $\Wbb_1,\Wbb_2$ be $\Vbb$-modules. If $A\in\Hom_\Vbb(\Wbb_1,\Wbb_2)$, then the linear functional
\begin{align*}
\Wbb_1\otimes\Wbb_2'\rightarrow\Cbb\qquad w_1\otimes w'_2\mapsto\bk{Aw_1,w'_2}
\end{align*}
is a conformal block associated to $\fk X$ by  Jacobi identity \eqref{eq17}. Conversely, if $\upphi\in\scr T_{\fk X}^*(\Wbb_1\otimes\Wbb'_2)$, we let $A:\Wbb_1\rightarrow \ovl{\Wbb_2}$ be the linear map such that $\upphi(w_1\otimes w'_2)=\bk{Aw_1,w'_2}$ holds for all $w_1\in\Wbb_1,w'_2\in\Wbb'_2$. Then Jacobi identity \eqref{eq17} implies that $A$ intertwines the action of any vertex operator $Y(v)_n$ where $v\in\Vbb,n\in\Zbb$. In particular, $[L_0,A]=0$. This implies that $A$ sends $\Wbb_1$ into $\Wbb_2$, and that $A\in\Hom_\Vbb(\Wbb_1,\Wbb_2)$. We conclude that there is a canonical isomorphism
\begin{align*}
\Hom_\Vbb(\Wbb_1,\Wbb_2)\simeq \scr T_{(\Pbb^1;0,\infty;\zeta,1/\zeta)}^*(\Wbb_1\otimes\Wbb'_2)
\end{align*}
\end{eg}

\begin{eg}\label{lb31}
Let $z\in\Cbb^\times$ and $\fk P=(\Pbb^1;0,z,\infty;\zeta,\zeta-z,1/\zeta)$. Let $\Wbb_1,\Wbb_2$ be $\Vbb$-modules. If $A\in\Hom_\Vbb(\Wbb_1,\Wbb_2)$, then the linear functional
\begin{align}
\Wbb_1\otimes\Vbb\otimes\Wbb'_2\rightarrow\Cbb   \qquad   w_1\otimes v\otimes w'_2\mapsto \bk{AY(v,z)w_1,w'_2}   \label{eq49}
\end{align}
is a conformal block associated to $\fk P$ by Jacobi identity \eqref{eq17}. Conversely, if $\upphi:\Wbb_1\otimes\Vbb\otimes\Wbb'_2\rightarrow\Cbb$ is a conformal block associated to $\fk P$. Then by Jacobi identity \eqref{eq17} or by Prop. \ref{lb20}, $\upphi$ is uniquely determined by its restriction to $\Wbb_1\otimes\id\otimes\Wbb'_2$ where $\id\in\Vbb$ is the vacuum vector. Using \eqref{eq17}, one checks easily that $w_1\otimes w'_2\in\Wbb_1\otimes\Wbb'_2\mapsto \upphi(w_1\otimes\id\otimes w'_2)$ is a conformal block associated to $(\Pbb^1;0,\infty;\zeta,1/\zeta)$, which is of the form $w_1\otimes w'_2\mapsto \bk{Aw_1,w'_2}$ due to Exp. \ref{lb30}. So by the uniqueness already mentioned, $\upphi$ must be of the form \eqref{eq49}. We conclude that there is a canonical isomorphism
\begin{align*}
\Hom_\Vbb(\Wbb_1,\Wbb_2)\simeq \scr T_{(\Pbb^1;0,z,\infty;\zeta,\zeta-z,1/\zeta)}^*(\Wbb_1\otimes\Vbb\otimes\Wbb'_2)
\end{align*}
\end{eg}

\begin{eg}[{\cite[Prop. 3.3]{Li94}}]\label{lb73}
Let $\Wbb$ be a $\Vbb$-module. Recall that $\Wbb_{(s)}$ is the $L_0$-weight $s$ subspace. Then Exp. \ref{lb4} shows $H^0(\Pbb,\scr V_{\Pbb^1}\otimes\omega_C(\star\infty))=\{\mc U_\varrho(\zeta)^{-1}v\zeta^nd\zeta:n\in\Nbb,v\in\Vbb\}$ and hence
\begin{align}
\scr T_{(\Pbb^1;\infty;1/\zeta)}(\Wbb)=\frac\Wbb{Y'(\Vbb)_{\geq0}\Wbb}:=\frac{\Wbb}{\{Y_\Wbb'(v)_nw:n\in\Nbb,v\in\Vbb,w\in\Wbb\}}
\end{align}
by \eqref{eq97}. We claim that $\Wbb_{(0)}\hookrightarrow\Wbb$ descends to a linear isomorphism
\begin{align}
\frac{\Wbb_{(0)}}{L_1\cdot \Wbb_{(1)}}\xlongrightarrow{\simeq} \frac\Wbb{Y'(\Vbb)_{\geq0}\Wbb}  \label{eq98}
\end{align}
Therefore the canonical surjective map $\Wbb^*\rightarrow\Wbb_{(0)}^*$ restricts to a linear isomorphism
\begin{align}
\scr T^*_{(\Pbb^1;\infty;1/\zeta)}(\Wbb)\xlongrightarrow{\simeq}\left(\frac{\Wbb_{(0)}}{L_1\cdot \Wbb_{(1)}} \right)^*
\end{align}
\end{eg}

\begin{proof}[Proof of \eqref{eq98}]
Let $\Xbb=Y'(\Vbb)_{\geq0}\Wbb$. Since $Y'(\cbf)_1$ equals $L_0$ acting on $\Wbb$, $\Xbb$ contains $L_0\Wbb=\bigoplus_{s\in\Cbb\setminus\{0\}}\Wbb_{(s)}$. Since $Y(\cbf)_0=L_{-1}$ and hence $Y'(\cbf)_0=L_1$, we have $L_1\cdot\Wbb_{(1)}\subset\Xbb$. So \eqref{eq98} is well-defined and surjective. To show that \eqref{eq98} is injective, it suffices to choose any $s\in\Cbb,n\in\Nbb$ and homogeneous $v\in\Vbb,w\in\Wbb$ satisfying $Y'(v)_nw\in\Wbb_{(0)}$ and show that $Y'(v)_nw\in L_1\cdot\Wbb$. By translation property, $[L_{-1},Y_{\Wbb'}(v)_n]=-nY_{\Wbb'}(v)_{n-1}$, whose transpose implies $Y'_\Wbb(v)_n=(n+1)^{-1}[L_1,Y'_\Wbb(v)_{n+1}]$. So
\begin{align*}
Y'(v)_n w+L_1\Wbb\subset \Cbb\cdot Y'(v)_{n+1}L_1 w+L_1\Wbb\subset\cdots\subset \Cbb\cdot Y'(v)_{n+k}L_1^k w+L_1\Wbb
\end{align*}
where the RHS is $0$ if $k$ is larger than the $L_0$-weight of $w$.
\end{proof}

\begin{eg}[{\cite[Prop. 3.4]{Li94}}]\label{lb74}
Let $\Wbb$ be a $\Vbb$-module. If $\upphi\in \left(\Wbb_{(0)}/L_1\Wbb_{(1)}\right)^* $, then by Exp. \ref{lb73}, for each $v\in\Vbb,w\in\Wbb$, $\upphi(Y'(v,z)w)$ belongs to $\Cbb[[z]]$ since $\upphi$ vanishes on $Y'(\Vbb)_{\geq 0}\Wbb$. Using Jacobi identity to expand $Y(Y(u)_nv)_{-1}$ and taking transpose, one sees that
\begin{align}
\Psi(\upphi):\Vbb\otimes\Wbb\rightarrow\Cbb\qquad v\otimes w\mapsto \lim_{z\rightarrow 0}\upphi(Y'(v,z)w)=\upphi(Y'(v)_{-1}w)
\end{align} 
belongs to $\scr T^*_{(\Pbb^1;0,\infty;\zeta,1/\zeta)}(\Vbb\otimes\Wbb)$, i.e. it gives a $\Vbb$-module morphism $\Vbb\rightarrow\Wbb'$. This gives a linear map
\begin{align}
\Psi:\left(\Wbb_{(0)}/L_1\Wbb_{(1)}\right)^*\xlongrightarrow{\simeq}\scr T^*_{(\Pbb^1;0,\infty;\zeta,1/\zeta)}(\Vbb\otimes\Wbb)\simeq \Hom_\Vbb(\Vbb,\Wbb')
\end{align}
which is bijective. It is injective because $\Psi(\upphi)(\id\otimes w)=\upphi(w)$. It is surjective because each $A\in\Hom_\Vbb(\Vbb,\Wbb')$ equals $\Psi(\upphi)$ where $\upphi:\Wbb\rightarrow\Cbb,w\mapsto\bk{A\id,w}$. Its inverse is
\begin{align}
\Hom_\Vbb(\Vbb,\Wbb')\xlongrightarrow{\simeq}\left(\Wbb_{(0)}/L_1\Wbb_{(1)}\right)^*\qquad\qquad A\mapsto A\id
\end{align}
\end{eg}

\begin{rem}
The above examples suggest that adding a distinct marked point to a pointed surface $\fk X$ and associating the vacuum module $\Vbb$ to that point does not essentially change the space of conformal blocks. (For instance, if $z\in\Cbb^\times$, we have a canonical isomorphism $(\Wbb_{(0)}/L_1\Wbb_{(1)})^*\simeq \scr T^*_{(\Pbb^1;0,z,\infty;\zeta,\zeta-z,1/\zeta)}(\Vbb\otimes\Vbb\otimes\Wbb)$.) This property is true in general and called \textbf{propagation of conformal blocks}. See \cite[Thm. 6.1]{Zhu94} or \cite[Thm. 10.3.1]{FB04} or \cite[Cor. 7.5]{Gui21} for rigorous statements and proofs. See also \cite{Gui23b} Sec. 16 (especially Thm. 16.7) for intuitive explanations.
\end{rem}

We give an interesting application of Exp. \ref{lb74}.

\begin{df}
A $\Vbb$-module $\Wbb$ is called \textbf{M\"obius unitarizable} if $\Wbb$ has an inner product $\bk{\cdot|\cdot}$ satisfying that $\bk{L_nw_1|w_2}=\bk{w_1|L_{-n}w_2}$ for all $w_1,w_2\in\Wbb$ and $n\in\{-1,0,1\}$. Choosing $n=0$, we see that the $L_0$-grading of $\Wbb$ is orthogonal under the inner product.
\end{df}

\begin{rem}
Suppose that a $\Vbb$-module $\Wbb$ is M\"obius unitary under an inner product $\bk{\cdot|\cdot}$. Let $\Co:\Wbb\rightarrow\Wbb'$ be the antilinear isomorphism such that $\bk{w_1,\Co w_2}=\bk{w_1|w_2}$ if $w_1,w_2\in\Wbb$. Define the inner product on $\Wbb'$ such that $\Co$ is antiunitary. Then $\Wbb'$ is M\"obius unitary under this inner product: for each $n=0,\pm1$, $\bk{w_1,L_n\Co w_2}=\bk{L_{-n}w_1,\Co w_2}=\bk{L_{-n}w_1|w_2}=\bk{w_1|L_nw_2}=\bk{w_1,\Co L_nw_2}$ and so $L_n\Co w=\Co L_n w$ for all $w\in\Wbb$. Thus $\bk{L_n\Co w_1|\Co w_2}=\bk{\Co L_n w_1|\Co w_2}=\bk{w_2| L_n w_1}=\bk{L_{-n}w_2|w_1}=\bk{\Co w_1|\Co L_{-n}w_2}=\bk{\Co w_1|L_{-n}\Co w_2}$. 
\end{rem}

\begin{thm}\label{lb75}
Let $\Wbb$ be a M\"obius unitarizable $\Vbb$-module. Let $\Sp(L_0)$ be the set of eigenvalues of $L_0$ on $\Wbb$. Then $\Sp(L_0)\subset[0,+\infty)$ and $L_1\Wbb_{(1)}=0$. Moreover, we have a linear isomorphism 
\begin{align*}
\Hom_\Vbb(\Vbb,\Wbb)\xlongrightarrow{\simeq}\Wbb_{(0)}\qquad A\mapsto A\id
\end{align*} 
\end{thm}

\begin{proof}
That $\Sp(L_0)\subset[0,+\infty)$ is well-known: Suppose $w\in\Wbb_{(s)}$ and $w\neq0$. Since, by assumption, the eigenvalues of $L_0$ have lower-bounded real parts, we can find $k\in\Nbb$ such that $\wtd w:=L_1^kw\neq 0$ but $L_1^{k+1}w=0$. Since $\bk{L_{-1}\wtd w|L_{-1}\wtd w}=\bk{L_1L_{-1}\wtd w|\wtd w}=\bk{[L_1,L_{-1}]\wtd w|\wtd w}=\bk{2L_0\wtd w|\wtd w}=2s\bk{\wtd w|\wtd w}$, we have $s\geq 0$.

Choose $w_1\in\Wbb_{(1)}$ and $w_0\in \Wbb_{(0)}$. We need to show $\bk{L_1 w_1|w_0}$ is $0$. It suffices to show $L_{-1}w_0=0$. Since $\Wbb_{-s}=0$, similar to the previous argument we have $\bk{L_{-1}w_0|L_{-1}w_0}=\bk{2L_0 w_0|w_0}=0$. This finishes the proof that $L_1\Wbb_{(1)}=0$.

By Exp. \ref{lb74}, we have a linear isomorphism $\Hom_\Vbb(\Vbb,\Wbb')\xlongrightarrow{\simeq}\Wbb_{(0)}^*$, $A\mapsto A\id$. The same is true about $\Hom_\Vbb(\Vbb,\Wbb)$.
\end{proof}

See Thm. \ref{lb76} and Prop. \ref{lb79} for applications of Thm. \ref{lb75}.

\subsection{Basic properties of conformal blocks}

The following fact is fundamental. Recall Rem. \ref{lb10}.

\begin{thm}\label{lb11}
Assume that  $\Vbb$ is $C_2$-cofinite. Then the space of conformal block $\scr T^*_{\fk X}(\Wbb_\blt)$ is finite-dimensional.
\end{thm}

\begin{proof}
This is due to \cite{AN03} under the assumption that the $C_2$-cofinite VOA $\Vbb$ is also quasi-primary generated (which automatically holds when $\Vbb$ is unitary). The general case was proved in \cite[Prop. 5.1.1]{DGT23}. See also \cite[Thm. 7.4]{Gui23a}.
\end{proof}

The following uniqueness result will be helpful.
\begin{pp}\label{lb20}
Assume that $C$ is connected and $N\geq 2$. Assume that $\Wbb_N$ is generated by a subset $\Ebb$. Assume that $\upphi,\uppsi\in\scr T_{\fk X}^*(\Wbb_\blt)$ satisfy that
\begin{align*}
\upphi(w_1\otimes\cdots\otimes w_{N-1}\otimes w_N)=\uppsi(w_1\otimes\cdots\otimes w_{N-1}\otimes w_N)
\end{align*}
for all $w_1\in\Wbb_1,\dots w_{N-1}\in\Wbb_{N-1}$ and $w_N\in\Ebb$. Then $\upphi=\uppsi$.
\end{pp}

\begin{proof}
This was proved in \cite[Prop. 7.2]{Gui21} in a more general setting. In this article, we only need this result in the special case that $C=\Pbb^1$ (and $N=3$). In this case, we can assume that the marked point $x_N$ for $\Wbb_N$ is $\infty$, and  change the local coordinates  (cf. Prop. \ref{lb6}) to \eqref{eq18}. Then this result is an easy consequence of the Jacobi identity \eqref{eq17} (in which we choose $f(z)=z^n$).
\end{proof}

We present two methods for constructing new conformal blocks from old ones. The first one is by changing local coordinates. The second one will be discussed in the next section.

\begin{pp}\label{lb6}
Let $\fk X=(C;x_\blt;\eta_\blt)$ and $\fk Y=(C;x_\blt;\mu_\blt)$ be two $N$-pointed compact Riemann surfaces with coordinates, where the underlying Riemann surface $C$ and the marked points $x_\blt$ are the same. Associate a $\Vbb$-module $\Wbb_i$ to each marked point $x_i$. Then we have an isomorphism of vector spaces
\begin{gather}\label{eq62}
\begin{gathered}
\scr T_{\fk X}^*(\Wbb_\blt)\rightarrow\scr T_{\fk Y}^*(\Wbb_\blt)\\
\upphi\mapsto \upphi\circ\big(\mc U(\eta_1\circ\mu_1^{-1})\otimes\cdots \otimes\mc U(\eta_N\circ\mu_N^{-1})\big)
\end{gathered}
\end{gather}
The same conclusion holds if we replace each $\mc U(\eta_i\circ\mu_i^{-1})$ with $\mc U_0(\eta_i\circ\mu_i^{-1})$.
\end{pp}

Note that each $\eta_i\circ\mu_i$ is an element of $\Gbb$, and $\mc U(\eta_i\circ\mu_i^{-1})$ is an invertible linear map on $\Wbb_i$.

\begin{proof}
This theorem is due to \cite[Thm. 6.5.4]{FB04}. (See also \cite[Thm. 3.2]{Gui23a}.) It is proved using Huang's change of coordinate Thm. \ref{lb5}.
\end{proof}

As a consequence, the dimension of $\scr T_{\fk X}^*(\Wbb_\blt)$ is independent of the local coordinates.

\begin{eg}\label{lb24}
Let $\tau\in\Cbb$, and choose a $\Vbb$-module $\Wbb$. Assume $\fk X=(\Pbb^1;\tau,\infty;\zeta-\tau,1/\zeta)$ where $\zeta$ is the standard coordinate of $\Cbb$. Associate $\Wbb,\Wbb'$ to the marked points $\tau,\infty$ respectively.  Then $\fk X$ is equivalent to $\fk Y=(\Pbb^1;0,\infty;\zeta,1/(\zeta+\tau))$. Let $\fk P=(\Pbb^1;0,\infty;\zeta,1/\zeta)$. Then
\begin{align*}
\Wbb\otimes\Wbb'\rightarrow\Cbb\qquad w\otimes w'\mapsto\bk{w,w'}
\end{align*}
is a conformal block associated to $\fk P$ and $\Wbb,\Wbb'$. (This follows e.g. from Rem. \ref{lb14}.) Thus, by Prop. \ref{lb6} and Exp. \ref{lb2}, 
\begin{gather*}
\Wbb\otimes\Wbb'\rightarrow\Cbb\qquad w\otimes w'\mapsto\bk{w,e^{\tau L_1}w'}=\bk{e^{\tau L_{-1}}w,w'}
\end{gather*}
is a conformal block associated to $\fk X$ and $\Wbb,\Wbb'$.
\end{eg}

\begin{eg}\label{lb48}
Let $\zeta$ be the standard coordinate of $\Cbb$. For each $z\in\Cbb^\times$, consider $\fk P_z=\{\Pbb^1;0,z,\infty;\zeta,\zeta-z,1/\zeta\}$. Associate $\Vbb$-modules $\Wbb_2,\Wbb_1$ and $\Wbb_3'$ (contragredient to $\Wbb_3$) to the marked points $0,z,\infty$ respectively.  According to Rem. \ref{lb14}, an element of $\scr T_{\fk P_1}^*(\Wbb_2,\Wbb_1,\Wbb_3')$ is equivalently a type $\Wbb_3\choose\Wbb_1\Wbb_2$ intertwining operator of $\Vbb$ at $1$. And any such element $\mc Y(\cdot,1)$ of $\scr T_{\fk P_1}^*(\Wbb_2,\Wbb_1,\Wbb_3')$, which is a linear functional
\begin{align*}
w_2\otimes w_1\otimes w_3'\in\Wbb_2\otimes\Wbb_1\otimes\Wbb_3'\qquad\mapsto \qquad \bk{\mc Y(w_1,1)w_2,w_3'},
\end{align*}
can be viewed as a linear map $\Wbb_2\otimes\Wbb_1\rightarrow\ovl\Wbb_3$. 

We let the argument of the $1$ in $\mc Y(\cdot,1)$ be $0$. Choose $\arg z$ for $z$. Since the map $\gamma\in\Pbb^1\mapsto z^{-1}\gamma$ gives an equivalence
\begin{align*}
\fk P_z\simeq \big(\Pbb^1;0,1,\infty;z\zeta,z(\zeta-1),1/(z\zeta)\big)
\end{align*}
according to Prop. \ref{lb6}, 
\begin{gather*}
\mc Y(\cdot,z):\Wbb_1\otimes\Wbb_2\rightarrow \ovl\Wbb_3 \\
\mc Y(w_1,z)w_2=z^{L_0}\mc Y(z^{-L_0}w_1,1)z^{-L_0}w_2 
\end{gather*}
is a type $\Wbb_3\choose\Wbb_1\Wbb_2$ intertwining operator at $z$. We view $\mc Y(\cdot,z)$ as a $\Hom_\Cbb(\Wbb_1\otimes\Wbb_2,\ovl\Wbb_3)$-valued multivalued holomorphic function of $z\in\Cbb^\times$, single valued on $\log z$. Such a function $\mc Y$ is called a \textbf{type $\Wbb_3\choose\Wbb_1\Wbb_2$ intertwining operator} of $\Vbb$. The vector space of all such functions is denoted by \index{I@$\mc I{\Wbb_3\choose\Wbb_1\Wbb_2}$}
\begin{align*}
\mc I{\Wbb_3\choose\Wbb_1\Wbb_2}
\end{align*}
and clearly has the same dimension as $\scr T_{\fk P_1}^*(\Wbb_2,\Wbb_1,\Wbb_3')$. This dimension is known as the \textbf{fusion rule} among $\Wbb_1,\Wbb_2,\Wbb_3$.  \hfill\qedsymbol
\end{eg}

\section{Sewing conformal blocks}\label{lb16}

The second way of constructing new conformal blocks is by sewing. We first review the geometric construction of sewing Riemann surfaces. There are two types of sewing: sewing two connected surfaces (i.e. taking the connected sum), and sewing a connected surface with itself (e.g. sewing a cylinder to get a torus). Both types can be described by the following procedure. We consider simultaneous sewing along several pairs of points. Let
\begin{align}\label{eq16}
&\fk X=(C;x_\blt;x_\star';x_\star'';\eta_\blt;\xi_\star;\varpi_\star)\nonumber\\
=&(C;x_1,\dots,x_N;x'_1,\dots,x_M';x''_1,\dots,x''_M;\eta_1,\dots,\eta_N;\xi_1,\dots,\xi_M;\varpi_1,\dots,\varpi_M)
\end{align}
be an $(N+2M)$-pointed compact Riemann surface with marked points $x_\blt,x_\star',x_\star''$, and $\eta_i$ (resp. $\xi_j$, $\varpi_j$) is a local coordinate of $C$ at $x_i$ (resp. $x_j'$, $x_j''$). Recall that $C$ is not assumed to be connected.

\begin{ass}\label{lb9}
In the sewing procedure, we assume that each connected component of $C$ contains at least one of $x_1,\dots,x_N$ (and not only one of $x_\blt,x'_\star,x''_\star$, as it would be required by Assumption \ref{lb7}). Moreover, we choose neighborhoods $W'_j$ of each $x'_j$ and $W_j''$ of $x''$ such that $\xi_j\in\scr O(W_j'),\varpi_j\in\scr O(W_j'')$ are univalent, and that any two members of
\begin{align*}
\{x_1\},\dots,\{x_N\},W_1',\dots,W_M', W_1'',\dots,W_M''
\end{align*}
do not intersect. (In particular, $x_j'$ (resp. $x_j''$) is the only one of the $N+2M$ marked points that belongs to $W_j'$ (resp. $W_j''$).)
\end{ass}

Recall the notations in \eqref{eq14}.

\begin{df}\label{lb8}
For each $1\leq j\leq M$, choose $r_j,\rho_j>0$ such that
\begin{align}
\xi_j(W_j')\supset\Dbb_{r_j},\qquad \varpi_j(W_j'')\supset\Dbb_{\rho_j}. \label{eq50}
\end{align}
Choose any $q_j\in\Dbb^\times_{r_j\rho_j}$. Write $q_\blt=(q_1,\dots,q_M)$. Define
\begin{gather*}
F_j=\xi^{-1}_j(\ovl\Dbb_{|q_j|/\rho_j})\cup\varpi_j^{-1}(\ovl\Dbb_{|q_j|/r_j})\\
\mc S_{q_\blt} C=\mc S_{q_1,\dots,q_M}=\Big(C\Big\backslash \bigcup_{1\leq j\leq M}F_j\Big)\Big/\thicksim
\end{gather*}
where $\sim$ is defined by identifying the subsets $\xi_j^{-1}(\Abb_{|q_j|/\rho_j,r_j})$ and $\varpi_j^{-1}(\Abb_{|q_j|/r_j,\rho_j})$ (for all $j$) via the rule
\begin{gather}\label{eq15}
\begin{gathered}
y'\in\xi_j^{-1}(\Abb_{|q_j|/\rho_j,r_j})\qquad\thicksim\qquad y''\in\varpi_j^{-1}(\Abb_{|q_j|/r_j,\rho_j})\\
\Updownarrow\\
\xi_j(y')\cdot \varpi_j(y'')=q_j
\end{gathered}
\end{gather}
The equivalence relation $\thicksim$ above is holomorphic. Therefore $\mc S_{q_\blt} C$ is a Riemann surface, which is not hard to see to be compact. We call $\mc S_{q_\blt} C$ the \textbf{sewing of $C$ with parameters $q_\blt$} (along the pairs of points $(x_1',x_1''),\dots,(x_M',x_M'')$). The marked points $x_\blt$ and their local coordinates $\eta_\blt$ become naturally marked points and local coordinates of $\mc S_qC$. We denote \index{Sq@$\mc S_{q_\blt}C,\mc S_{q_\blt}\fk X$}
\begin{align*}
\mc S_{q_\blt}\fk X=(\mc S_{q_\blt}C;x_\blt;\eta_\blt)
\end{align*}
and call it the \textbf{sewing of $\fk X$ with parameters $q_\blt$} (along $(x_\blt',x_\blt'')$ and with respect to the coordinates $\xi_\blt,\varpi_\blt$). Clearly $\mc S_{q_\blt}\fk X$ satisfies Asmp. \ref{lb7}.
\end{df}

Note that $\mc S_{q_\blt}C$ depends only on $q_\blt$,  the points $x_\star',x_\star''$, and their local coordinates $\xi_\star,\varpi_\star$. It does not depend on $x_\blt,\eta_\blt$ or the particular choice of $r_\star,\rho_\star$.

\begin{rem}
The above definition can be described in terms of conformal welding: For each $j$, choose any $\alpha_j>0$ and $\beta_j=|q_j|/\alpha_j$ such that $r_j<\alpha_j<|q_j|/\rho_j$ and $\rho_j<\beta_j<|q_j|/r_j$. Remove from $C$ the open disk centered at $x_j'$ with radius $\alpha_j$, i.e., remove $\xi_j^{-1}(\Dbb_{\alpha_j})$.  Similarly, we remove $\varpi_j^{-1}(\Dbb_{\beta_j})$, the open disk centered at $x_j''$ with radius $\beta_j$. Then we get a compact Riemann surface $\Sigma$ with boundaries $\Gamma_1',\Gamma_1'',\dots,\Gamma_M',\Gamma_M''$ where $\Gamma_j'=\xi_j^{-1}(\alpha_j\Sbb^1)$ and $\Gamma_j''=\varpi_j^{-1}(\beta_j\Sbb^1)$. Gluing each $\Gamma_j'$ with $\Gamma_j''$ via the diffeomorphism $\Gamma_j'\rightarrow \Gamma_j'',y'\mapsto y''$ (where $y'$ and $y''$ are related by $\xi_j(y')\varpi_j(y'')=q_j$), we get a surface with a canonical complex structure determined by that of $\Sigma$. (See \cite[Sec. 2.2.4]{Ten17} and the reference therein for the description of this complex structure.) This surface is $S_{q_\blt}C$. 
\end{rem}

To each marked point $x_i,x_j',x_j''$ of $\fk X$ (cf. \eqref{eq16}) we associate $\Vbb$-modules $\Wbb_i,\Mbb_j,\Mbb_j'$ respectively, where $\Mbb_j'$ is the contragredient module of $\Mbb_j$.  
\begin{df}
Choose $\upphi\in\scr T_{\fk X}^*(\Wbb_\blt\otimes\Mbb_\blt\otimes\Mbb_\blt')$, namely, the linear functional 
\begin{align*}
\upphi:\Wbb_1\otimes\cdots\otimes\Wbb_N\otimes\Mbb_1\otimes\Mbb_1'\otimes\cdots\otimes\Mbb_M\otimes\Mbb_M'\rightarrow\Cbb
\end{align*}
is a conformal block. (We have changed the order of tensor product in order to write $\Mbb_i$ and $\Mbb_i'$ together.) Let
\begin{align*}
\Ibf_j(n)=\text{the identity operator of $\Mbb_j(n)$, considered as an element of $\Mbb_j(n)\otimes \Mbb_j(n)^*$}
\end{align*}
Define a linear map \index{Sq@$\Std_{q_\blt}\upphi$}
\begin{gather*}
\Std_{q_\blt}\upphi:\Wbb_1\otimes\cdots\otimes\Wbb_N\rightarrow \Cbb[[q_1,\dots,q_N]]\\
\Std_{q_\blt}\upphi(w_\blt)=\sum_{n_1,\dots,n_M\in\Nbb}\upphi\big(w_\blt\otimes\Ibf_1(n_1)\otimes\cdots\otimes\Ibf_M(n_M)\big)q_1^{n_1}\cdots q_M^{n_M}
\end{gather*}
called the \textbf{(normalized) sewing of $\upphi$}.
\end{df}

\begin{rem}
Recall that $\Wbb_j(n)$ is the weight-$n$ eigenspace of the normalized Hamiltonian $\wtd L_0$. So, if we let $\Ibf_j$ be the identity operator of $\Mbb_j$, then we can informally write
\begin{align*}
\Std_{q_\blt}\upphi(w_\blt)=\upphi\big(w_\blt\otimes q_1^{\wtd L_0}\Ibf_1\otimes\cdots\otimes q_M^{\wtd L_0}\Ibf_M\big).
\end{align*}
This justifies our usage of the symbol $\Std$. In a similar manner, we can define
\begin{align}
\mc S_{q_\blt}\upphi(w_\blt)=\upphi\big(w_\blt\otimes q_1^{L_0}\Ibf_1\otimes\cdots\otimes q_M^{L_0}\Ibf_M\big) \label{eq19}
\end{align}
in a rigorous way. But this expression is not in $\Cbb[[q_1,\dots,q_M]]$. It is a formal series of $q_1,\dots,q_M$ whose powers are complex numbers but not necessarily natural numbers. If each $\Mbb_j$ is $L_0$-simple and $L_0=\lambda_j+\wtd L_0$ on $\Mbb_j$ (and hence on its contragredient module $\Mbb_j'$), then
\begin{align*}
\mc S_{q_\blt}\upphi(w_\blt)=q_1^{\lambda_1}\cdots q_M^{\lambda_M}\Std_{q_\blt}\upphi(w_\blt).
\end{align*}
We call $\mc S_{q_\blt}\upphi$ the \textbf{(standard) sewing of $\upphi$}. The two types of sewing clearly agree when $q_j=1$ and $\arg q_j=0$. 
\end{rem}

We are still under Asmp. \ref{lb9}. Write \index{D@$\Dbb_{r_\blt\rho_\blt},\Dbb_{r_\blt\rho_\blt}^\times$}
\begin{align*}
\Dbb_{r_\blt\rho_\blt}\coloneq \Dbb_{r_1\rho_1}\times\cdots\times \Dbb_{r_M\rho_M}\qquad \Dbb_{r_\blt\rho_\blt}^\times\coloneq \Dbb_{r_1\rho_1}^\times\times\cdots\times \Dbb_{r_M\rho_M}^\times
\end{align*}

\begin{thm}\label{lb12}
Choose any $\upphi\in\scr T_{\fk X}^*(\Wbb_\blt\otimes\Mbb_\blt\otimes\Mbb_\blt')$.
\begin{enumerate}
\item Suppose that for each $w_\blt\in\Wbb_\blt$, $\Std_{q_\blt}\upphi(w_\blt)$ converges absolutely on $q_\blt\in \Dbb_{r_\blt\rho_\blt}$. Then for each $q_\blt\in\Dbb_{r_\blt\rho_\blt}^\times$, the linear functional $\Std_{q_\blt}\upphi:\Wbb_\blt\rightarrow\Cbb$ is a conformal block associated to $\mc S_{q_\blt}\fk X$. Namely,
\begin{align*}
\Std_{q_\blt}\upphi\in \scr T_{\mc S_{q_\blt}\fk X}^*(\Wbb_\blt).
\end{align*}
\item Assume that $\Vbb$ is $C_2$-cofinite. Then for each $w_\blt\in\Wbb_\blt$, the series $\Std_{q_\blt}\upphi(w_\blt)$ converges absolutely on $\Dbb_{r_\blt\rho_\blt}$.
\end{enumerate}
\end{thm}

By \eqref{eq18} (when each $\Mbb_j$ is $L_0$-simple), this theorem clearly holds if the normalized sewing is replaced by the standard sewing $\mc S_{q_\blt}$.

\begin{proof}
Part 1 was proved in \cite[Thm. 11.3-1]{Gui23a}. Part 2 was proved in \cite[Thm. 11.4 or 13.1]{Gui23a}. In this article, we are mainly interested in the case that both $\fk X$ and $\mc S_{q_\blt}\fk X$ are of genus 0, and that all the local coordinates are M\"obius. In this case, the theorem is due to \cite{Hua05}.
\end{proof}

\begin{co}\label{lb25}
Assume that $\Vbb$ is $C_2$-cofinite. Let $\fk Y=(C;x_\blt;\eta_\blt)$ be an $N$-pointed compact Riemann surface with local coordinates. Associate a $\Vbb$-module $\Wbb_i$ to each $x_i$. Choose a neighborhood $U_i$ of $x_i$ on which $\eta_i$ is defined and univalent, and choose $r_i>0$ such that
\begin{align*}
\eta_i(U_i)\supset\Dbb_{r_i}
\end{align*}
Assume that $U_1,\dots,U_N$ are mutually disjoint. Then for each $\upphi\in\scr T_{\fk Y}^*(\Wbb_\blt)$ and $w_\blt\in\Wbb_\blt$, the series
\begin{align}
\uppsi_{\tau_\blt}(w_\blt)=\upphi(e^{\tau_1 L_{-1}}w_1\otimes\cdots\otimes e^{\tau_N L_{-1}}w_N) \label{eq63}
\end{align}
of $\tau_\blt=(\tau_1,\dots,\tau_N)$ converges absolutely on $\Dbb_{r_\blt}=\Dbb_{r_1}\times\cdots\times\Dbb_{r_N}$. Moreover, for each $\tau_\blt\in\Dbb_{r_\blt}$, $\uppsi_{\tau_\blt}$ is a conformal block associated to $\fk Y_{\tau_\blt}$ and $\Wbb_\blt$ where
\begin{align*}
\fk Y_{\tau_\blt}=\big(C;\eta_1^{-1}(\tau_1),\dots,\eta_N^{-1}(\tau_N);\eta_1-\tau_1,\dots,\eta_N-\tau_N \big)
\end{align*}
\end{co}

\begin{proof}
Choose $\tau_\blt\in\Dbb_{r_\blt}$. Let $\zeta$ be the standard coordinate of $\Cbb$. For each $i$, let $\fk P_i=(\Pbb^1;\tau_i,\infty;\zeta-\tau_i,1/\zeta)$. Associate $\Wbb_i$ to $\tau_i$ and $\Wbb_i'$ to $\infty$. Let $\fk Z=\fk Y\sqcup\fk P_1\sqcup\cdots\sqcup\fk P_N$. Then by Exp. \ref{lb24}, the linear functional
\begin{gather*}
\upomega:\Wbb_1\otimes\cdots\otimes\Wbb_N\otimes\Wbb_1\otimes\Wbb_1'\otimes\cdots\otimes\Wbb_N\otimes\Wbb_N'\rightarrow\Cbb
\end{gather*}
defined by
\begin{align*}
&\upomega(m_1\otimes\cdots\otimes m_N\otimes w_1\otimes w_1'\otimes\cdots\otimes w_N\otimes w_N')\\
=& \upphi(m_1\otimes\cdots\otimes m_N)\bk{w_1,e^{\tau_1L_1}w_1'}\cdots\bk{w_N,e^{\tau_NL_1}w_N'}
\end{align*}
is a conformal block associated to $\fk Z$. Now, we sew $\fk Z$ along the pairs $(x_1,\infty_1),\dots,(x_N,\infty_N)$ where $\infty_i$ denotes the $\infty$ of $\fk P_i$. By Thm. \ref{lb12}, $\mc S_{q_1,\dots,q_N}\upomega$ converges absolutely when $q_i\in\Dbb_{r_i/|\tau_i|}^\times$, and converges to a conformal block associated to $\mc S_{q_\blt}\fk Z$. But $\mc S_{1,\dots,1}\fk Z$ equals $\fk Y_{\tau_\blt}$, and $\mc S_{1,\dots,1}\upomega$ equals $\uppsi_{\tau_\blt}$ (and the absolute convergence of the former is equivalent to that of the latter) if we choose the argument of $1$ to be $0$ and assume for simplicity that $w_1,\dots,w_N$ are homogeneous. This finishes the proof. 
\end{proof}

Recall that if $\Vbb$ is $C_2$-cofinite, then $\Vbb$ has finitely many irreducible modules up to equivalence (cf. \cite{Zhu96,DLM00} or the end of \cite[Sec. 12]{Gui23a}), and the spaces of conformal blocks are finite-dimensional (Thm. \ref{lb11}).

\begin{thm}[\textbf{Sewing-factorization theorem}]\label{lb13}
Assume that $\Vbb$ is $C_2$-cofinite. Let $\mc E$ be a (necessarily finite) set of representatives of equivalence classes of irreducible $\Vbb$-modules. Choose any $q_\blt\in\Dbb_{r_\blt\rho_\blt}^\times$. Then the linear map
\begin{gather}
\begin{gathered}
\fk S_{q_\blt}:\bigoplus_{\Mbb_1,\dots\Mbb_M\in\mc E}\scr T_{\fk X}^*(\Wbb_\blt\otimes\Mbb_\blt\otimes\Mbb_\blt')\rightarrow\scr T_{\mc S_{q_\blt}\fk X}^*(\Wbb_\blt)\\
\bigoplus_\Mbb  \upphi_\Mbb\mapsto\sum_\Mbb \mc S_{q_\blt}\upphi_\Mbb
\end{gathered}
\end{gather}
(which is well defined by Thm. \ref{lb13}) is injective. If $\Vbb$ is also rational, then $\wtd{\fk S}_{q_\blt}$ is bijective. The same conclusions hold if $\mc S_{q_\blt}$ is replaced by the normalized sewing $\Std_{q_\blt}$ of conformal blocks.
\end{thm}

\begin{proof}
By induction on $M$ (i.e. by sewing the surface along a pair of points $x_j',x_j''$ each time) we can assume that $M=1$. Then the theorem is due to \cite[Thm. 2.1]{Gui23a}. Note that in \cite{Gui23a} the surjectivity of $\fk S_{q_\blt}$ (when $\Vbb$ is also rational) was proved by showing that the domain and the codomain of $\fk S_{q_\blt}$ have the same dimension. This result is called the \textbf{factorization of conformal blocks} and was proved in \cite{DGT23}. Again, in our article, we will use Thm. \ref{lb13} mainly in the special case that the surfaces (before and after sewing) are of genus 0 and that the local coordinates are M\"obius. In this case, Thm. \ref{lb13} was essentially established in a series of works \cite{HL95a,HL95b,HL95c,Hua95,Hua05}.\footnote{In these works, the theorem was presented in the form of ``braiding/fusion relations of intertwining operators", e.g. \eqref{eq69}. These relations are slightly weaker than the genus-0 M\"obius version of Thm. \ref{lb13}. But the methods of proving braiding/fusion relations actually imply Thm. \ref{lb13}. In this article, we will mainly use the fusion relations of intertwining operators instead of using Thm. \ref{lb13} directly. The main reason is that we shall generalize the results of this paper to twisted modules of $\Vbb$. For twisted modules and their conformal blocks, the factorization theorem has not yet been proved in general.} The genus-0  factorization property was also proved in \cite{NT05}.
\end{proof}

\section{Unitary VOA modules and conjugate conformal blocks}

In this section, we assume that $\Vbb$ is unitary as defined in \cite{DL14,CKLW18}.

\subsection{Unitary VOAs and unitary modules}

\begin{df}
We say that the VOA $\Vbb$ is \textbf{unitary}, if $\Vbb$ is equipped with an inner product $\bk{\cdot|\cdot}$ (antilinear on the second variable) satisfying $\bk{\id|\id}=1$, and if an antiunitary map $\Theta:\Vbb\rightarrow\Vbb$ (called the \textbf{PCT operator}) satisfying the following conditions is chosen: 
\begin{enumerate}[label=(\arabic*)]
\item For each $u,v\in\Vbb$,
\begin{align*}
\Theta Y(u,z)v=Y(\Theta u,\ovl z)\Theta v.
\end{align*}
In  particular, we have $\Theta L_n=L_n\Theta$, and hence $\Theta\Vbb(n)=\Vbb(n)$ for each $n\in \Nbb$. 
\item For each $u,v_1,v_2\in\Vbb$,
\begin{align*}
\bk{Y(u,z)v_1|v_2}=\bk{v_1|Y'(\Theta u,\ovl z)v_2}
\end{align*}
In particular, we have $\bk{L_n v_1|v_2}=\bk{v_1|L_{-n}v_2}$. By letting $v_2=\id$, we see that $\Theta$ is uniquely determined by the inner product.
\end{enumerate}
\end{df}

It follows automatically that $\Theta^2=1$, and that $\Theta$ preserves the conformal vector:
\begin{align*}
\Theta\cbf=\cbf.
\end{align*}
See \cite{DL14,CKLW18} for details. Moreover, a unitary VOA is self-dual: the isomorphism $\Vbb\rightarrow\Vbb'$ is given by the non-degenerate pairing $\Vbb\otimes\Vbb\rightarrow\Cbb,u\otimes v\mapsto\bk{u|\Theta v}$.

The definition of unitary $\Vbb$-modules is straightforward:

\begin{df}\label{lb78}
A $\Vbb$-module $\Wbb$ is called \textbf{unitary} if it is equipped with an inner product $\bk{\cdot|\cdot}$ (antilinear on the second variable) such that for each $u\in\Vbb,w_1,w_2\in\Wbb$,
\begin{align}
\bk{Y(u,z)w_1|w_2}=\bk{w_1|Y'(\Theta u,\ovl z)w_2}.  \label{eq21}
\end{align}
In particular, we have $\bk{L_nw_1|w_2}=\bk{w_1|L_{-n}w_2}$. We let \index{HW@$\mc H_\Wbb$, the Hilbert space completion of $\Wbb$}
\begin{align*}
\mc H_\Wbb=\text{the Hilbert space completion of $\Wbb$ under $\bk{\cdot|\cdot}$}.
\end{align*}
\end{df}

\begin{rem}
It is clear that $\Wbb_{(s)}$ is orthogonal to $\Wbb_{(t)}$ if $s\neq t$. From this, and from the construction of $\wtd L_0$ in Sec. \ref{lb15}, it follows that $\Wbb(m)$ is orthogonal to $\Wbb(n)$ if $m\neq n$. We conclude that $\bk{\cdot|\cdot}$ is \textbf{compatible with the $L_0$-grading} and the $\wtd L_0$-grading.
\end{rem}

\begin{rem}\label{lb62}
Assume that the $\Vbb$-module is equipped with an inner product $\bk{\cdot|\cdot}$ compatible with the $L_0$-grading or the $\wtd L_0$-grading. Then there is a unique inner product on $\Wbb'$ compatible with its $L_0$-grading such that the map
\begin{gather*}
\Co_\Wbb:\Wbb\rightarrow\Wbb'\qquad w\mapsto \bk{\cdot|w}\\
\text{( so }\bk{w_1|w_2}=\bk{w_1,\Co_\Wbb w_2}   \text{ if $w_1,w_2\in\Wbb$ )}
\end{gather*}
(abbreviated to $\Co$ when the context is clear) is antiunitary. Then, under the natural identification $\Wbb''=\Wbb$ we have
\begin{align*}
\Co_\Wbb^{-1}=\Co_{\Wbb'}
\end{align*}
$\Co_\Wbb$ extends naturally to a linear map between the algebraic completions
\begin{align*}
\Co_\Wbb:\ovl\Wbb\rightarrow \ovl{\Wbb'}=\Wbb^*
\end{align*}
Then $\Wbb$ is unitary if and only if for each $u\in\Vbb,w\in\Wbb$,
\begin{align}\label{eq20}
Y_\Wbb(u,z)w=\Co Y_{\Wbb'}(\Theta u,\ovl z)\Co w
\end{align}
From this we see that the inner product on $\Wbb'$ makes $\Wbb'$ a unitary $\Vbb$-module. 
\end{rem}

\begin{proof}[Proof of \eqref{eq21}$\Leftrightarrow$\eqref{eq20}]
We have
\begin{align*}
\bk{Y(u,z)w_1|w_2}=\bk{Y(u,z)w_1,\Co w_2}=\bk{w_1,Y'(u,z)\Co w_2}=\bk{w_1|\Co Y'(u,z)\Co w_2}
\end{align*}
which equals $\bk{w_1|Y'(\Theta u,\ovl z)w_2}$ iff \eqref{eq20} is always true.
\end{proof}

Recall that if $\Vbb$ is unitary, then by \cite[Prop. 5.3]{CKLW18} (or by Prop. \ref{lb79} with $\Ubb=\Vbb$), $\Vbb$ is simple iff $\Vbb$ is of CFT-type. The following fact is known to experts though we are unable to locate it in the literature. So we provide a proof below.

\begin{thm}\label{lb76}
Let $\Vbb$ be a unitary VOA, and let $\Wbb$ be a unitary $\Vbb$-module. Let $\Sp(L_0)$ be the set of eigenvalues of $L_0$ on $\Wbb$. Then $\Sp(L_0)\subset[0,+\infty)$. 

Assume moreover that $\Vbb$ and $\Wbb$ are irreducible $\Vbb$-modules. Then $\Sp(L_0)\subset(0,+\infty)$ if and only if $\Wbb\nsimeq \Vbb$.
\end{thm}

\begin{proof}
This follows immediately from Thm. \ref{lb75}.
\end{proof}

\subsection{Conjugate conformal blocks}

\begin{df}
Let $C$ be a Riemann surface. Define a new complex structure on $C$ to be the unique one such that a function $f:U\rightarrow\Cbb$ (where $U\subset\Cbb$ is open) is holomorphic under this complex structure if and only if its complex conjugate $\ovl f:U\rightarrow\Cbb,x\mapsto \ovl{f(x)}$ is holomorphic under the original complex structure. Then $C$, together with the new complex structure, is called the \textbf{(complex) conjugate Riemann surface} of $C$ (the original complex structure) and is denoted by $C^*$. Clearly $C^{**}=C$.
\end{df}

\begin{rem}
It is more convenient to regard $C$ and $C^*$ as two different sets. We let $x^*$ denote the point of $C^*$ that corresponds to the point $x\in C$. Then \index{C@$C^*$, conjugate Riemann surface}
\begin{align*}
*:C\rightarrow C^*
\end{align*}
is an anti-equivalence of Riemann surfaces, and
\begin{align*}
C^{**}=C\qquad *\circ*=\id_C:C\rightarrow C
\end{align*}
\end{rem}

\begin{df}
If $U\subset C$ is open and $f\in\scr O(U)$, we define $f^*\in\scr O(U^*)$ \index{f@$f^*\in\scr O(U^*)$, the conjugate of $f\in\scr O(U)$} (where $U^*$ is considered as an open subset of $C^*$) to be
\begin{align*}
f^*(x^*)=\ovl{f(x)}\qquad(\forall x\in U).
\end{align*}
\end{df}

\begin{rem}
Assume that $\eta\in\scr O(U)$ is univalent. Then for each $f\in\scr O(U)$,
\begin{align}
\partial_{\eta^*}f^*=(\partial_\eta f)^*
\end{align}
Using this relation, one easily computes the transition functions for the holomorphic tangent and cotangent bundles of $C^*$ in terms of those of $C$. Also, if $f$ is a meromorphic function on $U$, and if $x\in U$ is such that $\eta(x)=0$, then
\begin{align}\label{eq22}
\Res_{x^*} f^*d\eta^*=\ovl{\Res_{x}fd\eta} 
\end{align}
To see these relations (which can be checked locally), it suffices to assume that $U,U^*\subset\Cbb$, that $*:U\rightarrow U^*,z\mapsto \ovl z$, and that $\eta$ is the standard coordinate of $\Cbb$. Then $\eta^*=\eta$, and if $f=\sum_n a_nz^n$ then $f^*=\sum_n \ovl{a_n}z^n$.
\end{rem}

Let $\fk X=(C;x_1,\dots,x_N;\eta_1,\dots,\eta_N)$ be an $N$-pointed compact Riemann surface with local coordinates. Associate a unitary $\Vbb$-module $\Wbb_i$ to each marked point $x_i$ of $\fk X$.

\begin{df}\label{lb22}
The \textbf{complex conjugate} of $\fk X$ is defined to be \index{X@$\fk X^*$, the conjugate of $\fk X$}
\begin{align*}
\fk X^*=(C^*;x_\blt^*;\eta_\blt^*)=(C^*;x_1^*,\dots,x_N^*;\eta_1^*,\dots,\eta_N^*)
\end{align*}
Clearly $\fk X^{**}=\fk X$.
\end{df}

\begin{df}\label{lb23}
For each $\upphi\in\scr T_{\fk X}^*(\Wbb_\blt)$, define
\begin{gather*}
\upphi^*:\Wbb_\blt'=\Wbb_1'\otimes\cdots\otimes\Wbb_N'\rightarrow\Cbb\\
\upphi^*(\Co w_1\otimes\cdots\otimes\Co w_N)=\ovl{\upphi(w_1\otimes\cdots\otimes w_N)}
\end{gather*} 
(where $w_1\in\Wbb_1,\dots,w_N\in\Wbb_N$). Then
\begin{align}
\upphi^*\in\scr T_{\fk X^*}^*(\Wbb_\blt')  \label{eq23}
\end{align}
We call $\upphi^*$ the \textbf{conjugate conformal block} \index{zz@$\upphi^*$, the complex conjugate of the conformal block $\upphi$} of $\upphi$.
\end{df}

\begin{proof}[Proof of \eqref{eq23}]
There is an antilinear bijection 
\begin{gather*}
*:H^0(C,\scr V_C\otimes\omega_C(\star x_\blt))\rightarrow H^0(C^*,\scr V_{C^*}\otimes\omega_{C^*}(\star x_\blt^*)),\qquad \sigma\mapsto \sigma^*
\end{gather*}
defined as follows. Choose an open subset $U\subset C$ and a univalent $\eta\in\scr O(U)$, and write $\mc U_\varrho(\eta)\sigma$ on $U$ as a finite sum
\begin{align*}
\mc U_\varrho(\eta)\sigma|_U=\sum_k v_kf_kd\eta \qquad(v_k\in\Vbb,f_k\in\scr O(U))
\end{align*}
Then
\begin{align*}
\sigma^*|_{U^*}=\sum_k\Theta v_k\cdot f_k^*d\eta^*
\end{align*}
One checks easily that this is well-defined and independent of the choice of $\eta$. Choose any $w_\blt\in\Wbb_\blt$ and let $\Co w_\blt=\Co w_1\otimes\cdots\otimes\Co w_N$. Then
\begin{align*}
&\upphi^*(\sigma^*\cdot \Co w_\blt)=\sum_{i,k}\Res_{\eta_i^*=0}~f_k^*\upphi^*(\Co w_1\otimes\cdots\otimes Y(\Theta v_k,\eta_i^*)\Co w_i\otimes\cdots\otimes \Co w_N)d\eta_i^* \\
\xlongequal{\eqref{eq20}}&\sum_{i,k}\Res_{\eta_i^*=0}~f_k^*\upphi^*(\Co w_1\otimes\cdots \otimes \Co Y(v_k,\ovl{\eta_i^*}) w_i\otimes\cdots\otimes \Co w_N)d\eta_i^*\\
=&\sum_{i,k}\Res_{\eta_i^*=0}~f_k^*\ovl{\upphi(w_1\otimes\cdots \otimes Y(v_k,\ovl{\eta_i^*}) w_i\otimes\cdots\otimes w_N)}d\eta_i^*\\
\xlongequal{\eqref{eq22}}&\sum_{i,k}\ovl{\Res_{\eta_i=0}~f_k\upphi(w_1\otimes\cdots \otimes Y(v_k,\eta_i) w_i\otimes\cdots\otimes w_N)d \eta_i}\\
=&\ovl{\upphi(\sigma\cdot w_\blt)}
\end{align*}
So $\upphi\in\scr T_{\fk X}^*(\Wbb_\blt)$ if and only if $\upphi^*\in\scr T_{\fk X^*}^*(\Wbb_\blt')$.
\end{proof}

\section{Reflection positive conformal blocks}

In this section, we assume that $\Vbb$ is unitary and $C_2$-cofinite. We only consider unitary modules.

Consider an $(N+1)$-pointed compact Riemann surface with local coordinates
\begin{align}
\fk X=(C;x_\blt,x';\eta_\blt,\xi)=(C;x_1,\dots,x_N,x';\eta_1,\dots,\eta_N,\xi)  \label{eq25}
\end{align}
such that each connected component of $C$ contains one of $x_1,\dots,x_N$. We assume that $\xi$ is univalent on a neighborhood $W'$ of $x'$ that does not contain $x_1,\dots,x_N$. Choose $r>0$ such that
\begin{align*}
\xi(W')\supset\Dbb_r
\end{align*}
Associate unitary $\Vbb$-modules $\Wbb_1,\dots,\Wbb_N,\Mbb$ to $x_1,\dots,x_N,x'$ respectively. Associate their contragredient modules $\Wbb_1',\dots,\Wbb_N',\Mbb'$ respectively to the marked points $x_1^*,\dots,x_N^*,x'^*$ of $\fk X^*$.

\subsection{Sewing $\fk X$ and $\fk X^*$, and reflection positive conformal blocks}

Consider the disjoint union
\begin{align}\label{eq52}
\fk X\sqcup\fk X^*=(C\sqcup C^*;x_\blt,x',x_\blt^*,x'^*;\eta_\blt,\xi,\eta_\blt^*,\xi^*)
\end{align}
Assume $0<q<r^2$, and define the connected sum \index{XX@$\fk X\#_q\fk X^*=\mc S_q(\fk X\sqcup\fk X^*)$}  along the pair of points $x',x'^*$:
\begin{align}\label{eq53}
\fk X\#_q\fk X^*\coloneq\mc S_q(\fk X\sqcup\fk X^*)
\end{align}
which has marked points $x_\blt,x_\blt^*$ and local coordinates $\eta_\blt,\eta_\blt^*$. 

\begin{df}\label{lb33}
Let $\Upomega\in\scr T_{\fk X\#_q\fk X^*}^*(\Wbb_\blt\otimes\Wbb_\blt')$. We say that $\Upomega$ is \textbf{reflection positive}, if for each vector $\wbf\in\Wbb_\blt$ (not necessarily of the form $w_\blt=w_1\otimes\cdots\otimes w_N$) we have
\begin{align}
\Upomega(\wbf\otimes \Co\wbf)\geq0  \label{eq95}
\end{align}
Here, $\Co:\Wbb_\blt\rightarrow\Wbb_\blt$ is the antiunitary map determined by
\begin{align*}
\Co(w_1\otimes\cdots\otimes w_N)=\Co w_1\otimes\cdots\otimes\Co w_N\qquad(\forall w_1\in\Wbb_1,\dots,w_N\in\Wbb_N)
\end{align*}

If assumption \eqref{eq95} is weakened to
\begin{align}
\Upomega(\wbf\otimes \Co\wbf)\in\Rbb\qquad\forall \wbf\in\Wbb_\blt
\end{align} 
or equivalently
\begin{align}
\ovl{\Upomega(\wbf_1\otimes\Co\wbf_2)}=\Upomega(\wbf_2\otimes\Co\wbf_1)\qquad\forall \wbf_1,\wbf_2\in\Wbb_\blt
\end{align}
we call $\Upomega$ \textbf{self-conjugate}. (This is a special case of Def. \ref{lb67}.) \hfill\qedsymbol
\end{df}

Choose $\upphi,\uppsi\in\scr T_{\fk X}^*(\Wbb_\blt\otimes\Mbb)$. Then $\upphi\otimes\uppsi^*:\Wbb_\blt\otimes\Mbb\otimes\Wbb_\blt'\otimes\Mbb'\rightarrow\Cbb$ is clearly a conformal block:
\begin{align*}
\upphi\otimes\uppsi^*\in\scr T_{\fk X\sqcup\fk X^*}^*(\Wbb_\blt\otimes\Mbb\otimes\Wbb_\blt'\otimes\Mbb')
\end{align*}
The corresponding sewing of $\upphi\otimes\uppsi^*$, written as \index{zz@$\upphi\#_q\uppsi^*=\mc S_q(\upphi\otimes\uppsi^*)$}
\begin{align*}
\upphi\#_q\uppsi^*\coloneq \mc S_q(\upphi\otimes\uppsi^*):\Wbb_\blt\otimes\Wbb_\blt'\rightarrow\Cbb
\end{align*}
satisfies that for any $\wbf_1,\wbf_2\in\Wbb_\blt$,
\begin{align}\label{eq24}
&\upphi\#_q\uppsi^*(\wbf_1\otimes \Co\wbf_2)=\sum_{s\geq 0}\sum_{\alpha\in\fk A_s} q^s\upphi(\wbf_1\otimes e_{\alpha,s})\uppsi^*(\Co\wbf_2\otimes \Co e_{\alpha,s})\nonumber\\
=&\sum_{s\geq 0}\sum_{\alpha\in\fk A_s} q^s\upphi(\wbf_1\otimes e_{\alpha,s})\ovl{\uppsi(\wbf_2\otimes e_{\alpha,s})}
\end{align}
where $\{e_{\alpha,s}\}_{\alpha\in\fk A_s}$ is an orthonormal basis of the $L_0$-eigenspace $\Wbb_{(s)}$, and we choose $\arg q=0$. Let us express \eqref{eq24} in terms of the inner product of two vectors in a Hilbert space:

\begin{pp}\label{lb17}
Assume $0<q<r^2,\arg q=0$ and $\upphi\in\scr T_{\fk X}^*(\Wbb_\blt\otimes\Mbb)$. Regard $\upphi$ as a linear map \index{T@$T_\upphi$}
\begin{align*}
T_\upphi:\Wbb_\blt\rightarrow \ovl{\Mbb'}
\end{align*}
Then for each $\wbf\in\Wbb_\blt$, $\sqrt q^{L_0}T_\upphi(\wbf)$ belongs to the Hilbert space completion $\mc H_{\Mbb'}$. Moreover, if $\wbf_1,\wbf_2\in\Wbb_\blt$, and if $\uppsi\in\scr T_{\fk X}^*(\Wbb_\blt\otimes\Mbb)$, then
\begin{gather}
\bigbk{\sqrt q^{L_0}T_\upphi(\wbf_1)\big|\sqrt q^{L_0}T_\uppsi(\wbf_2) }=\upphi\#_q\uppsi^*(\wbf_1\otimes\Co\wbf_2)
\end{gather}
\end{pp}

\begin{proof}
We use the notations in \eqref{eq24}. Then
\begin{align*}
\sqrt q^{L_0}T_\upphi(\wbf)=\sum_{s\geq 0}\sum_{\alpha\in\fk A_s}q^{s/2}\upphi(\wbf\otimes e_{\alpha,s})\Co e_{\alpha,s}
\end{align*}
To show that $\sqrt q^{L_0}T_\upphi(\wbf)\in\mc H_{\Mbb'}$, we need to check that the square sum of the above coefficients is finite: By \eqref{eq24},
\begin{align*}
\sum_{s,\alpha}\big|q^{s/2}\upphi(\wbf\otimes e_{\alpha,s})\big|^2=\sum_{s,\alpha} q^s\upphi(\wbf\otimes e_{\alpha,s})\ovl{\upphi(\wbf\otimes e_{\alpha,s})}=\upphi\#_q\upphi^*(\wbf\otimes \Co\wbf)
\end{align*}
converges absolutely by Thm. \ref{lb12}. This proves $\sqrt q^{L_0}T_\upphi(\wbf)\in\mc H_{\Mbb'}$. We compute
\begin{align*}
&\bigbk{\sqrt q^{L_0}T_\upphi(\wbf_1)\big|\sqrt q^{L_0}T_\uppsi(\wbf_2) }  \nonumber\\
=&\sum_{s,t\geq0}\sum_{\alpha\in\fk A_s,\beta\in\fk A_t}\bigbk{q^{s/2}\upphi(\wbf_1\otimes e_{\alpha,s})\Co e_{\alpha,s} \big|q^{s/2}\uppsi(\wbf_2\otimes e_{\beta,t})\Co e_{\beta,t} }\nonumber\\
=&\sum_{s\geq 0}\sum_{\alpha\in\fk A_s} q^s\upphi(\wbf_1\otimes e_{\alpha,s})\ovl{\uppsi(\wbf_2\otimes e_{\alpha,s})}
\end{align*}
which equals \eqref{eq24}.
\end{proof}

\begin{co}
Assume $0<q<r^2,\arg q=0$ and $\upphi\in\scr T_{\fk X}^*(\Wbb_\blt\otimes\Mbb)$. Then $\upphi\#_q\upphi^*$ is reflection positive.
\end{co}

\begin{proof}
Choose $\wbf\in\Wbb_\blt$. Then by Prop. \ref{lb17},
\begin{align*}
\upphi\#_q\upphi^*(\wbf\otimes\Co\wbf)=\big\lVert \sqrt q^{L_0}T_\upphi \wbf   \big\lVert ^2\geq 0.
\end{align*}
\end{proof}

\subsection{Reflection positivity is preserved by change-of-coordinates and translations}

Let us show that reflection positivity is preserved by suitable change of coordinates. Recall that $\fk X=(C;x_\blt,x';\eta_\blt,\xi)=\eqref{eq25}$. Choose a new local coordinate $\mu_i$ for each $x_i$. Let
\begin{align*}
\fk Y=(C;x_\blt,x';\mu_\blt,\xi)=(C;x_1,\dots,x_N,x';\mu_1,\dots,\mu_N,\xi)
\end{align*}
and let $\fk Y\#_q\fk Y^*$ also denote the sewing of $\fk Y$ and $\fk Y^*$ along $x',x'^*$ with respect to $\xi,\xi^*$ and parameter $q$. We assume $0<q<r^2$ and $\arg q=0$.

\begin{pp}\label{lb42}
Let $\Upomega\in\scr T_{\fk X\#_q\fk X^*}^*(\Wbb_\blt\otimes\Wbb_\blt')$ be reflection positive. Define a linear functional $\Upupsilon:\Wbb_\blt\otimes\Wbb_\blt'\rightarrow\Cbb$ by
\begin{align*}
\Upupsilon=\Upomega\circ\big(\mc U_0(\eta_1\circ\mu_1^{-1})\otimes\cdots\otimes\mc U_0(\eta_N\circ\mu_N^{-1})\otimes\mc U_0(\eta_1^*\circ(\mu_1^*)^{-1})\otimes\cdots\otimes\mc U_0(\eta_N^*\circ(\mu_N^*)^{-1}) \big)
\end{align*}
which is an element of $\scr T_{\fk Y\#_q\fk Y^*}^*(\Wbb_\blt\otimes\Wbb_\blt')$ (cf. Prop. \ref{lb6}). Assume that the arguments are chosen such that for each $i$,
\begin{align}
\arg (\eta_i\circ\mu_i^{-1})'(0)=-\arg(\eta_i^*\circ(\mu_i^*)^{-1})'(0).  \label{eq26}
\end{align}
Then $\Upupsilon$ is reflection positive.
\end{pp}

Note that
\begin{align}
(\eta_i\circ\mu_i^{-1})'(0)=\ovl{(\eta_i^*\circ(\mu_i^*)^{-1})'(0)}. \label{eq46}
\end{align}
(Indeed, write $\eta_i=\sum_{n\in\Zbb_+} a_n\mu_i^n$ where $a_n\in\Cbb$. Then $\eta_i^*=\sum_n \ovl{a_n}(\mu_i^*)^n$. Then the above relation says $a_1=\ovl{\ovl{a_1}}$.) Therefore, the arguments can be chosen to satisfy \eqref{eq26}.

\begin{proof}
The fact that
\begin{align*}
\eta_i=\sum_{n\in\Zbb_+} a_n\mu_i^n\qquad\Longrightarrow \qquad\eta_i^*=\sum_n \ovl{a_n}(\mu_i^*)^n
\end{align*}
implies that
\begin{align*}
\mc U_0(\eta_i\circ\mu_i^{-1})=a_0^{L_0}\exp\Big(\sum_{n\geq 1 }c_nL_n\Big)\quad\Rightarrow \quad \mc U_0(\eta_i^*\circ(\mu_i^*)^{-1})=\ovl{ a_0}^{L_0}\exp\Big(\sum_{n\geq 1 }\ovl{c_n}L_n\Big)
\end{align*}
Therefore, by \eqref{eq26}, we have $\Co\cdot a_0^{L_0}=\ovl{a_0}^{L_0}\cdot \Co$, and hence
\begin{align}\label{eq47}
\Co\cdot  \mc U_0(\eta_i\circ\mu_i^{-1})=\mc U_0(\eta_i^*\circ(\mu_i^*)^{-1})\cdot\Co
\end{align}
This relation, together with the definition of $\Upupsilon$, immediately shows that $\Upupsilon$ is reflection positive.
\end{proof}

Reflection positivity is also preserved by translations:

\begin{pp}\label{lb44}
For each $1\leq i\leq N$, choose a neighborhood $U_i$ of $x_i$ on which $\eta_i$ is defined and univalent, and choose $r_i>0$ such that
\begin{align*}
\eta_i(U_i)\supset\Dbb_{r_i}
\end{align*} 
Assume that $U_1,\dots,U_N$ are mutually disjoint. Choose a reflection positive $\Upomega\in\scr T_{\fk X\#_q\fk X^*}(\Wbb_\blt\otimes\Wbb_\blt')$. Choose $\tau_i\in\Dbb_{r_i}$ for each $i$. Let
\begin{gather*}
\fk X_{\tau_\blt}=(C;\eta_1^{-1}(\tau_1),\dots,\eta_N^{-1}(\tau_N);\eta_1-\tau_1,\dots,\eta_N-\tau_N)\\
\fk X_{\ovl{\tau_\blt}}^*=(\fk X_{\tau_\blt})^*=(C^*;(\eta_1^*)^{-1}(\ovl{\tau_1}),\dots,(\eta_N^*)^{-1}(\ovl{\tau_N});\eta_1^*-\ovl{\tau_1},\eta_N^*-\ovl{\tau_N})
\end{gather*}
Consider the linear functional $\Upomega_{\tau_\blt}:\Wbb_\blt\otimes\Wbb_\blt'\rightarrow\Cbb$ defined by
\begin{align}  \label{eq36}
\Upomega_{\tau_\blt}(w_\blt\otimes w'_\blt)=\Upomega\big(e^{\tau_1 L_{-1}}w_1\otimes\cdots\otimes e^{\tau_NL_{-1}}w_N\otimes e^{\ovl {\tau_1} L_{-1}}w_1'\otimes\cdots\otimes e^{\ovl {\tau_N}L_{-1}}w_N' \big)
\end{align}
which (by Cor. \ref{lb25}) converges absolutely (as a power series of $\tau_1,\dots,\tau_N,\ovl{\tau_1},\dots,\ovl{\tau_N}$) and
\begin{align*}
\Upomega_{\tau_\blt}\in\scr T_{\fk X_{\tau_\blt}\#_q \fk X^*_{\ovl{\tau_\blt}}}^*(\Wbb_\blt\otimes\Wbb_\blt')
\end{align*}
Then $\Upomega_{\tau_\blt}$ is reflection positive.
\end{pp}

\begin{proof}
For each $n\in\Nbb$, define a linear map $R_n$ on $\Wbb_\blt$ by
\begin{align*}
R_n w_\blt=\sum_{k_1,\dots k_N=0}^n\frac{(\tau_1L_{-1})^{k_1}w_1}{k_1!}\otimes\cdots\otimes \frac{(\tau_NL_{-1})^{k_N}w_N}{k_N!}
\end{align*}
Choose any $\wbf\in\Wbb_\blt$. Then $\Upomega(R_n\wbf\otimes \Co R_n\wbf)\geq 0$ by the reflection positivity of $\Upomega$. Since \eqref{eq36} converges absolutely, we have
\begin{align*}
\Upomega_{\tau_\blt}(\wbf\otimes\Co\wbf)=\lim_{n\rightarrow\infty}\Upomega(R_n\wbf\otimes \Co R_n\wbf)
\end{align*}
which is $\geq 0$.
\end{proof}

\subsection{Relationship to the positivity of endomorphisms}

We still assume the setting at the beginning of this section. We assume $0<q<r^2$ and $\arg q=0$.

Let $A\in\End_\Vbb(\Mbb')$, namely, $A$ is a linear map on $\Mbb'$ intertwining the action of $\Vbb$. In particular, $[L_0,A]=0$ implies that $A$ preserves the $L_0$-grading of $\Mbb'$. Thus, we can extend $A$ to a linear map on the algebraic completion $\ovl{\Mbb'}=\Mbb^*$.

\begin{rem}\label{lb19}
Assume for simplicity that $\Mbb$ and (hence) $\Mbb'$ are semisimple. Then $A:\Mbb^*\rightarrow\Mbb^*$ restricts to a bounded linear map on the Hilbert space $\mc H_{\Mbb'}$. To see this, we write $\Mbb$ and $\Mbb'$ as finite orthogonal direct sums of irreducible unitary submodules
\begin{align} \label{eq28}
\Mbb=\bigoplus^\perp_n \Xbb_n\otimes_\Cbb \mc I_n\qquad \Mbb'=\bigoplus^\perp_n \Xbb_n'\otimes_\Cbb \mc I_n^*
\end{align}
Here $\Xbb_n$ is irreducible and unitary,  each $\mc I_n$ is a finite-dimensional Hilbert space, and $\Xbb_m\nsimeq\Xbb_n$ if $m\neq n$. Then $A$ on $\Mbb'$ is of the form
\begin{align}
A=\bigoplus_n \id\otimes A_n  \label{eq27}
\end{align}
where $A_n:\mc I_n^*\rightarrow \mc I_n^*$ is linear. So $A$ is bounded and extends uniquely to a bounded linear map on
\begin{align*}
\mc H_{\Mbb'}=\bigoplus_n\mc H_{\Xbb_n'}\otimes \mc I_n^*.
\end{align*}
\end{rem}

Assume $\Mbb$ to be semisimple. Choose $A\in\End_\Vbb(\Mbb')$ and $\upphi\in\scr T_{\fk X}^*(\Wbb_\blt\otimes\Mbb)$. Define $A\upphi\in\scr T_{\fk X}^*(\Wbb_\blt\otimes\Mbb)$ to be
\begin{align}
A\upphi=\upphi\circ(\id_{\Wbb_\blt}\otimes A^\tr). \label{eq61}
\end{align}
where $A^\tr:\Mbb\rightarrow\Mbb$ is the transpose of $A$. Then
\begin{align} \label{eq31}
T_{A\upphi}=A\cdot T_\upphi
\end{align} 
Let $0<q<r^2$ with $\arg q=0$. From Prop. \ref{lb17}, it is clear that if $A$ is \textbf{positive} (namely, if $\bk{A\mbf m'|\mbf m'}\geq 0$ for all $\mbf m'\in\Mbb'$, equivalently, for all $\mbf m'\in\mc H_{\Mbb'}$), then $A\upphi\#_q\upphi^*$ is reflection positive. We record this result:
\begin{align}\label{eq30}
A\geq 0\qquad\Longrightarrow\qquad A\upphi\#_q\upphi^*\text{ is reflection positive}
\end{align}
Our goal is to show the opposite direction under a natural density condition.

\begin{df}
Let $\Wbb$ be a $\Vbb$-module, and let $\Xbb$ be a subspace of the algebraic completion $\ovl\Wbb$. We say that $\Xbb$ is \textbf{$\sigma(\ovl\Wbb,\Wbb')$-dense} (or \textbf{dense in the $\sigma(\ovl\Wbb,\Wbb')$-topology}) if the only vector $w'\in\Wbb'$ that is orthogonal to $\Xbb$ (namely, $\bk{w',w}=0$ for all $w\in \Xbb$) is $0$.
\end{df}

Recall the notation in \eqref{eq25}. The following density result is close in spirit to \cite[Lem. 3.6]{Ten19c}.

\begin{pp}\label{lb18}
Let $\upphi\in\scr T_{\fk X}^*(\Wbb_\blt\otimes\Mbb)$, and assume that the subspace $T_\upphi(\Wbb_\blt)$ (the range of $T_\upphi$) of $\ovl{\Mbb'}$ is $\sigma(\ovl{\Mbb'},\Mbb)$-dense. Assume that $C=\Pbb^1$ and the local coordinate $\xi$ is M\"obius. Then $\sqrt q^{L_0}T_\upphi(\Wbb_\blt)$ is dense in $\mc H_{\Mbb'}$ (under the norm topology). 
\end{pp}

Note that if $T_\upphi(\Wbb_\blt)$ is $\sigma(\ovl{\Mbb'},\Mbb)$-dense in $\ovl{\Mbb'}$, then so is $\sqrt q^{L_0}T_\upphi(\Wbb_\blt)$.

\begin{proof}
Let $\mc X=\sqrt q^{L_0}T_\upphi(\Wbb_\blt)$ which is a subspace of $\mc H_{\Mbb'}$. Let us show that $\mc X^{\cl}=\mc H_{\Mbb'}$. Choose any $\psi\in \mc H_{\Mbb'}$ orthogonal to $\mc X$. We need to show that $\psi=0$. Let $L_0$ act on $\mc H_{\Mbb'}$ as an unbounded self-adjoint operator. 

Step 1. We claim that for each $t\in\Rbb$, $e^{\im tL_0}\psi$ is orthogonal to $\mc X$. Suppose this is true. By a standard argument in spectral theory, $\int_\Rbb f(x)e^{\im xL_0}\psi dx$ is orthogonal to $\mc X$ for all $f\in C_c^\infty(\Rbb)$, and hence the image of $\psi$ under any spectral projection is orthogonal to $\mc X$. This implies $\psi=0$. 

Let us provide more details for the convenience of readers unfamiliar with spectral theory. For each eigenvalue $s\geq0$ of $L_0$, let $P_s$ be the projection of $\mc H_{\Mbb'}$ onto $\Mbb'_{(s)}$. Then it suffices to prove $P_s\psi=0$ for each $s$. Note that $P_s=g(L_0)$ where $g:\Rbb\rightarrow\Rbb$ is a smooth function with compact support such that $g(s)=1$ and $g(t)=0$ if $t\neq s$ is an eigenvalue of $L_0$.  Let $f:\Rbb\rightarrow\Cbb$ be the rapidly decreasing function such that $\int_\Rbb f(x)e^{\im xy}dx=g(y)$ (i.e. the inverse Fourier transform of $g$). Then (for instance) by evaluating with any eigenvalue of $L_0$, one checks
\begin{align*}
P_s\psi=\int_\Rbb f(x)e^{\im x L_0} \psi dx
\end{align*}
where the RHS is norm-convergent as a vector-valued improper Riemann integral. So $P_s\psi$ is orthogonal to $\mc X$. This proves that $P_s\psi=0$ because $\mc X$ is $\sigma(\ovl{\Mbb'},\Mbb)$-dense in $\ovl{\Mbb'}$.

Step 2. Let us prove the claim in step 1 that $e^{\im t L_0}\psi$ is orthogonal to $\mc X$, i.e. orthogonal to $\sqrt q^{L_0}T_\upphi(\wbf)$ for all $\wbf\in\Wbb_\blt$. Note that if $0<\lambda<r/\sqrt q$ (note that $r/\sqrt q>1$) and $\arg\lambda=0$, then $\lambda^{L_0}\sqrt q^{L_0}T_\upphi(\wbf)$ belongs to the Hilbert space $\mc H_{\Mbb'}$ by Prop. \ref{lb17}. Namely, $\sqrt q^{L_0}T_\upphi(\wbf)$ belongs to the domain of the self-adjoint operator $\lambda ^{L_0}$.  Therefore $\tau\mapsto e^{\im\tau L_0}\sqrt q^{L_0}T_\upphi(\wbf)$ is a holomorphic function on a strip $\{\tau\in\Cbb:-\varepsilon<\Imag(\tau)<\varepsilon\}$ for some $\varepsilon>0$ (cf. for instance \cite[Chapter VI, Lem. 2.3]{Tak}).

Suppose we can prove that $T_\upphi(\Wbb_\blt)$ is $L_0$-invariant. Then for each $n\in\Nbb$, $\psi$ is orthogonal to $\sqrt q^{L_0}L_0^nT_\upphi(\wbf)$. Therefore
\begin{align*}
\partial_\tau^n \bigbk{e^{\im\tau L_0}\sqrt q^{L_0}T_\upphi(\wbf)\big|\psi}=\im^n\bigbk{e^{\im\tau L_0}\sqrt q^{L_0}L_0^nT_\upphi(\wbf)\big|\psi}
\end{align*}
equals zero when $\tau=1$. Therefore, since the function $\bk{e^{\im\tau L_0}\sqrt q^{L_0}T_\upphi(\wbf)|\psi}$ of $\tau$ is holomorphic on the strip, this function is constant zero. So $\bk{\sqrt q^{L_0}T_\upphi(\wbf)|e^{-\im tL_0}\psi}=0$ for all $t\in\Rbb$. This finishes the proof.

Step 3. It remains to prove that $T_\upphi(\Wbb_\blt)$ is $L_0$-invariant. Here we use the fact that $C=\Pbb^1$ and that $\xi$ is M\"obius. After a biholomorphic transformation of $\fk X$, we may assume that $x'=\infty$ and $\xi=1/\zeta$ where $\zeta$ is the standard coordinate of $\Cbb$. Namely,
\begin{align*}
\fk X=(\Pbb^1;x_1,\dots,x_N,\infty;\eta_1,\dots,\eta_N,1/\zeta).
\end{align*}
By Prop. \ref{lb6}, we may change each local coordinate $\eta_i$ (where $1\leq i\leq N$) to $\zeta-x_i$. Then the invariance of $T_\upphi(\Wbb_\blt)$ under $L_0$ follows immediately from
\begin{align}\label{eq29}
L_0T_\upphi(w_1\otimes\cdots\otimes w_N)=\sum_{i=1}^N T_\upphi\big(w_1\otimes\cdots\otimes (L_0+x_i L_{-1})w_i\otimes\cdots\otimes w_N\big),
\end{align}
a special case of the Jacobi identity \eqref{eq17}.
\end{proof}

\begin{rem}
In Prop. \ref{lb18}, if we do not assume $C=\Pbb^1$ or that $\xi$ is M\"obius,  proving norm-density would be more difficult due to the absence of a relation like  \eqref{eq29}.
\end{rem}

\begin{thm}\label{lb45}
Let $\upphi\in\scr T_{\fk X}^*(\Wbb_\blt\otimes\Mbb)$, and assume that the subspace $T_\upphi(\Wbb_\blt)$ of $\ovl{\Mbb'}$ is $\sigma(\ovl{\Mbb'},\Mbb)$-dense. Assume that $C=\Pbb^1$ and that the local coordinate $\xi$ is M\"obius. Assume that $\Mbb$ is semisimple, and choose $A\in\End_\Vbb(\Mbb')$. Then the following are equivalent.
\begin{enumerate}[label=(\arabic*)]
\item $A\geq 0$, namely, $\bk{A\mbf m'|\mbf m'}\geq0$ for all $\mbf m'\in\Mbb'$ (equivalently, for all $\mbf m'\in\mc H_{\Mbb'}$).
\item $A\upphi\#\upphi^*$ is reflection positive.
\end{enumerate}
\end{thm}

Note that $A\geq 0$ is equivalent to that $A_n\geq 0$ for each $A_n$ in \eqref{eq27}.

\begin{proof}
(1)$\Rightarrow$(2) was explained in \eqref{eq27}. Assume (2), then by Prop. \ref{lb17} and $T_{A\upphi}=AT_\upphi$, for each $\wbf\in\Wbb_\blt$, we have
\begin{align*}
\bk{A\sqrt q^{L_0}T_\upphi \wbf|\sqrt q^{L_0}T_\upphi\wbf}=\bk{\sqrt q^{L_0}T_{A\upphi}\wbf|\sqrt q^{L_0}T_\upphi\wbf}=A\upphi\#_q\upphi^*(\wbf\otimes\Co\wbf),
\end{align*}
which is $\geq0$ by (2). Therefore $A\geq0$ because $A$ is bounded and $\sqrt q^{L_0}T_\upphi(\Wbb_\blt)$ is dense in $\mc H_{\Mbb'}$ (Prop. \ref{lb18}).
\end{proof}

\begin{lm}\label{lb70}
Under the assumptions of Thm. \ref{lb45}, we have $A=A^*$ if and only if $A\upphi\#\upphi^*$ is self-conjugate (cf. Def. \ref{lb33}).
\end{lm}

\begin{proof}
Same as the proof Thm. \ref{lb45}.
\end{proof}

The remaining task is to find a useful criterion on the $\sigma(\ovl{\Mbb'},\Mbb)$-density of $T_\upphi(\Wbb_\blt)$.

\begin{thm}\label{lb21}
Choose $\upphi\in\scr T_{\fk X}^*(\Wbb_\blt\otimes\Mbb)$. Assume that $\Mbb$ is semisimple, and let  $\Mbb=\bigoplus_n\Xbb_n\otimes_\Cbb\mc I_n$ be a (finite) orthogonal irreducible decomposition as in Rem. \ref{lb19}. For each $n$, define a linear map
\begin{gather}
\begin{gathered}
\Phi_n: \mc I_n\rightarrow \scr T_{\fk X}^*(\Wbb_\blt\otimes\Xbb_n)\\
\gamma\quad\mapsto\quad \Big(\wbf\otimes \fk x\in\Wbb_\blt\otimes\Xbb_n\mapsto \upphi\big(\wbf\otimes (\fk x\otimes\gamma)\big)   \Big)
\end{gathered}
\end{gather}
Then the following are equivalent.
\begin{enumerate}[label=(\arabic*)]
\item For each $n$, the linear map $\Phi_n$ is injective.
\item $T_\upphi(\Wbb_\blt)$ is $\sigma(\ovl{\Mbb'},\Mbb)$-dense in $\ovl{\Mbb'}$.
\end{enumerate}
\end{thm}

The proof we give below is similar to that of \cite[Prop. A.3]{Gui19a}. Indeed, Prop. A.3 is sufficient for the purpose of this article. Thm. \ref{lb21}, as a generalization of Prop. A.3, is stated and proved here for the sake of completeness.

\begin{proof}
Let $\Ybb$ be the subspace of vectors of $\Mbb$ orthogonal to $T_\upphi(\Wbb_\blt)$. If we let $\wtd\Ybb$ denote the $\Vbb$-submodule of $\Mbb$ generated $\Ybb$, then by Prop. \ref{lb20}, the restriction of $\upphi$ to $\Wbb_\blt\otimes\wtd\Ybb$ is zero. This implies that $\Ybb=\wtd\Ybb$. Therefore, $\Ybb$ is a (unitary) $\Vbb$-submodule of $\Mbb$. (Note that when $C=\Pbb^1$ and $\xi$ is M\"obius, similar to Step 3 of the proof of Prop. \ref{lb18}, one can use the Jacobi identity \eqref{eq17} to show that $\Ybb$ is $\Vbb$-invariant.)

Note that (2) holds iff $\Ybb=0$. Thus (2)$\Rightarrow$(1) is obvious: Suppose that (1) is not true. Choose $n$ and choose a non-zero $\gamma\in\mc I_n$ such that $\Phi_n(\gamma)=0$. Then $\Ybb$ contains $\Xbb_n\otimes\Cbb\gamma$, which is non-zero. So (2) is not true.

Now suppose that (2) is not true. Then $\Ybb$ is a non-zero (semisimple) $\Vbb$-submodule of $\Mbb$. Thus, according to the decomposition $\Mbb=\bigoplus_n\Xbb_n\otimes_\Cbb\mc I_n$, $\Ybb$ must contain an irreducible submodule $\Ybb_0$ isomorphic to $\Xbb_n$ for some $n$. The projection of $\Ybb_0$ onto $\Xbb_m\otimes\mc I_m$ must be zero if $m\neq n$ (since $\Xbb_m\nsimeq\Xbb_n$). So $\Ybb_0\subset \Xbb_n\otimes\mc I_n$. So $\Ybb_0=\Xbb_n\otimes\Cbb\gamma$ for some nonzero $\gamma\in\mc I_n$. Hence $\Phi_n(\gamma)=0$. So (1) does not hold.
\end{proof}

\begin{rem}\label{lb87}
Condition (1) in Thm. \ref{lb21} can be written more explicitly. Write
\begin{align*}
\Mbb=\bigoplus_n (\underbrace{\Xbb_n\oplus\cdots\oplus\Xbb_n}_{k_n\text{ pieces}})
\end{align*}
where each $\Xbb_n$ is irreducible, and $\Xbb_m\nsimeq\Xbb_n$ if $m\neq n$. For each $n$ and each $1\leq l\leq k_n$, let $\upphi_{n,l}\in\scr T_{\fk X}^*(\Wbb_\blt\otimes\Xbb_n)$ be the restriction of $\upphi$ to the $l$-th component of $\Xbb_n\oplus\cdots\oplus\Xbb_n$ in the above decomposition. Then condition (1) means that $\upphi_{n,1},\dots\upphi_{n,k_n}$ are linearly independent for each $n$.
\end{rem}

\section{Positive trinions and basic conformal blocks}

Let $\Vbb$ be unitary. Unless otherwise stated,  $\fk X=(C;x_\blt;\eta_\blt)$ denotes an $N$-pointed compact Riemann surface, and to each marked point $x_i$ a unitary $\Vbb$-module $\Wbb_i$ is associated.

\begin{cv}
A $3$-pointed sphere $\Pbb^1$ with local coordinates is called a \textbf{trinion}. Unless otherwise stated, when sewing two trinions $\fk P=(\Pbb^1;x_1,x_2,x_3;\eta_1,\eta_2,\eta_3)$ and $\fk Q=(\Pbb^1;y_1,y_2,y_3;\mu_1,\mu_2,\mu_3)$, we always assume that the sewing is along their last marked points $x_3,y_3$.
\end{cv}

\subsection{Self-conjugate Riemann surfaces and conformal blocks}

\begin{df}\label{lb51}
Suppose that there is a biholomorphism $\tau:C^*\rightarrow C$. Then $\fk X$ together with the anti-biholomorphism $\varstar=\tau\circ *:C\rightarrow C$ is called \textbf{self-conjugate} if there is a (necessarily unique) bijection $*:\{1,\dots,n\}\rightarrow\{1,\dots,n\}$ satisfying that for each $1\leq i\leq N$ and $x\in C$,
\begin{subequations}
\begin{gather}
\varstar\circ\varstar=\id_C\\
x_i^\varstar=x_{i^*}\label{eq100}\\
\eta_{i^*}(x^\varstar)=\ovl{\eta_i(x)}\label{eq32}
\end{gather}
Clearly $i^{**}=i$. As for the $\Vbb$-modules, we assume also that
\begin{gather}
\Wbb_{i^*}=\Wbb_i'\label{eq33}\\
\Wbb_i=\Wbb_i'=\Vbb\qquad(\text{if }i^*=i) \label{eq34}
\end{gather}
\end{subequations}
where, in \eqref{eq34}, we identify $\Vbb$ and $\Vbb'$ via the unitary isomorphism
\begin{align*}
\Co\Theta:\Vbb\rightarrow\Vbb'
\end{align*}
so that $\Co=\Theta$ on $\Vbb$.
\end{df}

\begin{cv}\label{lb26}
Suppose that $\fk X$ is viewed as a self-conjugate $N$-pointed surface. Then we identify $C^*$ with $C$ via $\tau$ so that $\varstar=*$. We call $*:C\rightarrow C$ the \textbf{involution} of $\fk X$. Then \eqref{eq32} means
\begin{align}
(\eta_i)^*=\eta_{i^*}
\end{align}
We thus have (recall Def. \ref{lb22})
\begin{align}
\fk X^*=(C;x_{1^*},\dots,x_{N^*};\eta_{1^*},\dots,\eta_{N^*})\label{eq35}
\end{align}
Therefore, $\fk X$ and $\fk X^*$ differ by a permutation of marked points.
\end{cv}

The advantage of identifying $\varstar$ with $*$ is indicated by the following example.

\begin{eg}
If we let $\fk X$ be \eqref{eq25} (which is $(N+1)$-pointed) and assume the setting of \eqref{eq52} and $\eqref{eq53}$, then $\fk X\#_q\fk X^*$ is naturally a self-conjugate $2N$-pointed compact Riemann surface with local coordinates: the involution $\varstar$ on the sewn Riemann surface $\mc S_q(C\sqcup C^*)$ is defined by sending any point $p\in C$ which is not discarded in the sewing process to $p^*$ in $C^*$.
\end{eg}

Recall Def. \ref{lb23} about the definition of conjugate conformal blocks. If $\upphi\in\scr T_{\fk X}^*(\Wbb_\blt)$, then $\upphi^*:\Wbb_{1^*}\otimes\cdots\otimes\Wbb_{N^*}\rightarrow\Cbb$ is its complex conjugate. Therefore $\upphi^*$, composed with the permutation $\Wbb_\blt\rightarrow\Wbb_{\blt^*}$ (sending $w_1\otimes\cdots\otimes w_N$ to $w_{1^*}\otimes\cdots\otimes w_{N^*}$), is clearly an element of $\scr T_{\fk X}^*(\Wbb_\blt)$. Whenever this conformal block equals $\upphi$, we say that $\upphi$ is \textbf{self-conjugate}. To be more explicit, we make the following definition:

\begin{df}\label{lb67}
Let $\fk X$ be self-conjugate. A conformal block $\upphi\in\scr T_{\fk X}^*(\Wbb_\blt)$ is called \textbf{self-conjugate} if for each $w_1\in\Wbb_1,\dots,w_N\in\Wbb_N$ we have
\begin{align}
\upphi(\Co w_{1^*}\otimes \cdots\otimes\Co w_{N^*})=\ovl{\upphi(w_1\otimes\cdots\otimes w_N)} 
\end{align}
\end{df}

\subsection{Positive trinions}

Recall that an anti-biholomorphism $\varstar:C\rightarrow C$ satisfying $\varstar\circ \varstar=\id_C$ is called an \textbf{involution  of $C$}. Let us classify involutions of $\Pbb^1$ having fixed points. (Note that the involution $z\mapsto -1/\ovl z$ does not have fixed points.) Choose $x_1,x_2,x_3\in\Pbb^1$ such that $x_1^\varstar=x_2$ and $x_3^*=x_3$. By a M\"obius transformation, it suffices to assume for instance $x_1=0,x_2=\infty,x_3=1$.

\begin{eg}\label{lb27}
Let $\varstar:\Pbb^1\rightarrow\Pbb^1$ be an involution having fixed points. Let $\theta\in\Rbb$ and $t\in\Rbb\cup\{\infty\}$.
\begin{enumerate}[label=(\alph*)]
\item If $0^\varstar=\infty$ and $(e^{\im\theta})^\varstar=e^{\im\theta}$, then $z^\varstar=1/\ovl z$ for all $z\in\Pbb^1$.
\item If $\im^\varstar=-\im$ and $t^\varstar=t$, then $z^\varstar=\ovl z$ for all $z\in\Pbb^1$.
\end{enumerate}
\end{eg}

Consequently, if $x_1,x_2,x_3\in\Pbb^1$ are chosen, then there exists a unique involution of $\Pbb^1$ fixing $x_3$ and exchanging $x_1,x_2$.

\begin{proof}
Choose an involution $\varstar$. Then the composition of $\varstar$ and $z\mapsto \ovl z$ is a M\"obius transformation of $\Pbb^1$. It is well-known that M\"obius transformations are uniquely determined by their values on three distinct points. So the involutions $\varstar$ satisfying (a) resp. (b)  are unique. And clearly $\varstar:z\mapsto 1/\ovl z$ (resp. $\varstar:z\mapsto \ovl z$) satisfies the requirements in (a) resp. (b). 
\end{proof}

Recall Conv. \ref{lb26}.

\begin{df}\label{lb85}
Let $\fk P=(\Pbb^1;x_1,x_2,x_3;\eta_1,\eta_2,\eta_3)$ be a self-conjugate trinion. Assume that $\eta_3$ is M\"obius. Assume that
\begin{align*}
x_1^*=x_2\qquad x_3^*=x_3
\end{align*}
(and hence the bijection $*$ on $\{1,2,3\}$ satisfies $1^*=2,3^*=3$, cf. \eqref{eq100}). Then we automatically have
\begin{align}
\Imag(\eta_3(x_1))\cdot\Imag(\eta_3(x_2))<0.\label{eq37}
\end{align}
We say that $\fk P$ is a \textbf{positive trinion} (or a positive pair-of-pants) if
\begin{align}\label{eq101}
\Imag(\eta_3(x_1))<0,\qquad\Imag(\eta_3(x_2))>0.
\end{align}
We say that $\fk P$ is a \textbf{standard positive trinion} if
\begin{align*}
\eta_3(x_1)\in\im\Rbb_{<0}\qquad \eta_3(x_2)\in\im\Rbb_{<0}
\end{align*}
\end{df}


\begin{rem}
We will define the geometric positivity of fusion products by sewing basic conformal blocks associated to two positive trinions, cf. Def. \ref{lb39}. This definition is unchanged if both positive trinions are replaced by negative trinions, i.e., those satisfying $\Imag(\eta_3(x_1))>0$ and $\Imag(\eta_3(x_2))<0$. However, if we define geometric positivity by sewing a positive trinion and a negative trinion, we need to modify Def. \ref{lb39} by replacing the contragredient module $\Wbb_1'$ in \eqref{eq102} with $\Wbb_1$.
\end{rem}

Relation \eqref{eq37} follows from the proof of the next property.

\begin{pp}\label{lb34}
Any positive trinion is equivalent to the positive trinion
\begin{align}
\fk P_-=\big(\Pbb^1;a+b\im,a-b\im,\infty;\eta,\eta^*,1/\zeta\big)  \label{eq38}
\end{align}
where $a,b\in\Rbb$ and $b>0$, $\zeta$ is the standard coordinate of $\Cbb$, $\eta$ is a local coordinate at $a+b\im$, and
\begin{align*}
\eta^*(z)=\ovl{\eta(\ovl z)}
\end{align*}
The involution of $\Pbb^1$ is $z^*=\ovl z$. Moreover, $\fk P_-$ is standard iff $a=0$.
\end{pp}

\begin{proof}
If $\fk X$ is a self-conjugate $3$-pointed sphere with local coordinates satisfying $x_1^*=x_2,x_3^*=x_3$, then $\fk X$ is equivalent to 
\begin{align*}
\fk X_1=(\Pbb^1;\im,-\im,\infty;\mu,\mu^*,\xi)
\end{align*}
where (by Exp. \ref{lb27}) the involution $*$ is given by $z^*=\ovl z$. So $\mu^*(\ovl z)=\ovl{\mu(z)}$. $\xi$ is a M\"obius local coordinate of $\infty$. So $\xi$ is of the form
\begin{align*}
\xi(z)=1/(a+bz)
\end{align*}
where $a,b\in\Cbb$ and $b\neq 0$. Since $\xi^*(z)=1/(\ovl a+\ovl bz)$ equals $\xi(z)$, we have $a,b\in\Rbb$. This proves \eqref{eq37}. Clearly $\fk X_1$ is equivalent to \eqref{eq38} via the biholomorphism $z\in\Pbb^1\mapsto a+bz\in\Pbb^1$. Since $\fk P$ is a positive trinion, we must have $b>0$.
\end{proof}

The subscript $-$ in $\fk P_-$ indicates the fact that the real line $\Rbb$ is the \textbf{equator} of $\fk P_-$ (namely, the set of points fixed by the involution). Likewise, the unit circle $\Sbb^1$ is the equator of $\fk P_\circ$ defined below:

\begin{pp}\label{lb35}
Any positive trinion is equivalent to the positive trinion
\begin{align}
\fk P_\circ=\Big(\Pbb^1;\gamma,1/\ovl\gamma,\im;\mu,\mu^*,\frac{\im(\zeta-\im)}{\zeta+\im}\Big) \label{eq39}
\end{align}
where $\gamma\in\Dbb_1$, $\zeta$ is the standard coordinate of $\Cbb$, $\mu$ is a local coordinate at $\gamma$, and 
\begin{align*}
\mu^*(z)=\ovl{\mu(1/\ovl z)}
\end{align*}
The involution of $\Pbb^1$ is $z^*=1/\ovl z$. Moreover, $\fk P_\circ$ is standard iff $\gamma\in\Dbb_1\cap\im\Rbb$.
\end{pp}

\begin{proof}
\eqref{eq39} is equivalent to \eqref{eq38} via the biholomorphism $z\in\Pbb^1\mapsto \frac{z+\im}{\im(z-\im)}\in\Pbb^1$.
\end{proof}

\subsection{Basic conformal blocks associated to positive trinions}

\begin{thm}\label{lb28}
Let $\Wbb$ be a unitary $\Vbb$-module. There exists, for each positive trinion
\begin{align}
\fk P=(\Pbb^1;x_1,x_2,x_3;\eta_1,\eta_2,\eta_3), \label{eq40}
\end{align}
a self-conjugate conformal block $\upomega_{\Wbb,\fk P}\in\scr T_{\fk P}^*(\Wbb\otimes\Wbb'\otimes\Vbb)$ depending only on the equivalence class of $\fk P$, such that the following properties are satisfied: 
\begin{enumerate}[label=(\alph*)]
\item Let $\zeta$ be the standard coordinate of $\Cbb$. If we choose positive trinion
\begin{align}\label{eq43}
\fk P_\circ=\Big(\Pbb^1;0,\infty,\im;\zeta,1/\zeta,\frac{\im(\zeta-\im)}{\zeta+\im}\Big)
\end{align}
then for each $w_1,w_2\in\Wbb,v\in\Vbb$,
\begin{align}\label{eq41}
\upomega_{\Wbb,\fk P_\circ}(w_1\otimes \Co w_2\otimes v)=\bigbk{Y\big(2^{L_0}e^{-\im L_1}v,\im\big)w_1\big|w_2}
\end{align}
\item If we choose positive trinions $\fk P=\eqref{eq40}$ and $\fk Q=(\Pbb^1;x_1,x_2,x_3;\mu_1,\mu_2,\mu_3)$, then
\begin{align}\label{eq44}
\upomega_{\Wbb,\fk Q}=\upomega_{\Wbb,\fk P}\circ\big(\mc U_0(\eta_1\circ\mu_1^{-1})\otimes \mc U_0(\eta_2\circ\mu_2^{-1})\otimes \mc U_0(\eta_3\circ\mu_3^{-1})\big)
\end{align}
where we have
\begin{gather}
(\eta_1\circ\mu_1^{-1})'(0)=\ovl{(\eta_2\circ\mu_2^{-1})'(0)}  \label{eq45}
\end{gather}  
and the arguments are chosen such that
\begin{gather}
\arg(\eta_1\circ\mu_1^{-1})'(0)=-\arg(\eta_2\circ\mu_2^{-1})'(0)  \label{eq48}
\end{gather}
\end{enumerate}
We call $\upomega_{\Wbb,\fk P}$  \index{zz@$\upomega_{\Wbb,\fk P}$, the basic conformal block} the \textbf{basic conformal block associated to $\Wbb$ and the positive trinion $\fk P$}.
\end{thm}

Basic conformal blocks are crucial to the definition of geometric positivity in Def. \ref{lb39}. The most important property about a basic conformal block $\upomega_{\Wbb,\fk P}$ is that when $\Wbb$ is irreducible, $\upomega_{\Wbb,\fk P}$ is the unique (up to positive scalar multiplications) conformal block associated to $\Wbb,\fk P$ sending each $w\otimes\Co w\otimes\id$ (where $w\in\Wbb\setminus\{0\}$) to a  positive number. This will be explained in Prop. \ref{lb32}.

\begin{proof}
By Rem. \ref{lb14}, the linear functional
\begin{gather*}
\Wbb\otimes\Wbb'\otimes \Vbb\mapsto \Cbb\\
w_1\otimes \Co w_2\otimes v\mapsto \bk{Y(v,\im)w_1|w_2}=\bk{Y(v,\im)w_1,\Co w_2}
\end{gather*}
is a conformal block associated to $(\Pbb^1;0,\infty,\im;\zeta,1/\zeta,\zeta-\im)$. By Prop. \ref{lb6}, Exp. \ref{lb2}, and the fact that $z-\im=\alpha\big(\frac{\im(z-\im)}{z+\im}\big)$ where $\alpha(z)=\frac {2z}{1+\im z}$, we see that the linear functional $\upomega_{\Wbb,\fk P_\circ}$ defined by \eqref{eq41} is a conformal block associated to $\fk P_\circ$ and $\Wbb,\Wbb',\Vbb$.

To show that $\upomega_{\Wbb,\fk P_\circ}$ is self-conjugate, we need to show that
\begin{align}
\ovl{\upomega_{\Wbb,\fk P_\circ}(w_2\otimes\Co w_1\otimes \Theta v)}=\upomega_{\Wbb,\fk P_\circ}(w_1\otimes \Co w_2\otimes v) \label{eq42}
\end{align}
for all $w_1,w_2\in\Wbb$ and $v\in\Vbb$. If we define a linear functional $\upomega':\Wbb\otimes\Wbb'\otimes\Vbb\rightarrow\Cbb$ using the LHS of \eqref{eq42}, then $\upomega'$ is a conformal block associated to $\fk P_\circ$ by Def. \ref{lb23}. Thus, in order to prove \eqref{eq42} (namely,  to prove $\upomega'=\upomega_{\Wbb,\fk P_\circ}$), by Prop. \ref{lb20}, it suffices to check 
\begin{align*}
\ovl{\upomega_{\Wbb,\fk P_\circ}(w_2\otimes\Co w_1\otimes \id)}=\upomega_{\Wbb,\fk P_\circ}(w_1\otimes \Co w_2\otimes \id)
\end{align*}
for all $w_1,w_2\in\Wbb$, where $\id\in\Vbb$ is the vacuum vector. To prove this relation , we use \eqref{eq41} to compute that
\begin{align*}
\ovl{\upomega_{\Wbb,\fk P_\circ}(w_2\otimes\Co w_1\otimes \id)}=\ovl{\bk{w_2|w_1}}=\bk{w_1|w_2}=\upomega_{\Wbb,\fk P_\circ}(w_1\otimes \Co w_2\otimes \id)
\end{align*}
We are done with the proof of part (a).

We now proceed to prove part (b). First, note that \eqref{eq45} can be proved in the same way as \eqref{eq46}. Now, to define $\upomega_{\Wbb,\fk Q}$ for an arbitrary positive trinion $\fk Q=(\Pbb^1;x_\blt;\mu_\blt)$, we find a (unique) M\"obius map sending $0,\infty,\im$ to $x_1,x_2,x_3$ respectively, and use this map to find local coordinates $\eta_1,\eta_2,\eta_3$ such that $\fk P_\circ=\eqref{eq43}$ is equivalent to $\fk P$. Then we define $\upomega_{\Wbb,\fk Q}$ using  \eqref{eq44} in which $\upomega_{\Wbb,\fk P}$ is replaced by $\upomega_{\Wbb,\fk P_\circ}$. Note that this definition is independent of the choice of arguments provided that \eqref{eq48} is satisfied!

To prove that $\upomega_{\Wbb,\fk Q}$ is self-conjugate, it suffices to show as above that
\begin{align*}
\ovl{\upomega_{\Wbb,\fk Q}(w_2\otimes\Co w_1\otimes \id)}=\upomega_{\Wbb,\fk Q}(w_1\otimes \Co w_2\otimes \id).
\end{align*}
Let $\alpha=\eta_1\circ\mu_1^{-1}$ and $\alpha^*=\eta_1^*\circ(\mu_1^*)^{-1}$. (Namely, $\alpha^*(z)=\ovl{\alpha(\ovl z)}$.) Then since $\eta_2=\eta_1^*,\mu_2=\mu_1^*$,  and since \eqref{eq48} holds, we have $\Co \mc U_0(\alpha)=\mc U_0 (\alpha^*)\Co$ as in \eqref{eq47}. Thus, by \eqref{eq42},
\begin{align*}
&\ovl{\upomega_{\Wbb,\fk Q}(w_2\otimes\Co w_1\otimes \id)}=\ovl{\upomega_{\Wbb,\fk P_\circ}\big(\mc U_0(\alpha)w_2\otimes \mc U_0(\alpha^*)\Co w_1\otimes \id\big)}\\
=&\ovl{\upomega_{\Wbb,\fk P_\circ}\big(\mc U_0(\alpha)w_2\otimes \Co\mc U_0(\alpha) w_1\otimes \id\big)}=\upomega_{\Wbb,\fk P_\circ}\big(\mc U_0(\alpha) w_1\otimes \Co\mc U_0(\alpha)w_2\otimes \id\big)\\
=&\upomega_{\Wbb,\fk P_\circ}\big(\mc U_0(\alpha) w_1\otimes \mc U_0(\alpha^*)\Co w_2\otimes \id\big)=\upomega_{\Wbb,\fk Q}(w_1\otimes \Co w_2\otimes \id).
\end{align*}
Thus, we have defined a self-conjugate conformal block $\upomega_{\Wbb,\fk Q}$ for each positive trinion $\fk Q$. 

Finally, for each positive trinions $\fk P,\fk Q$ as in part (b), relation \eqref{eq44} follows from \eqref{eq48} and Cor. \ref{lb29}. This finishes the proof of part (b).
\end{proof}

\subsection{Conformal blocks positively proportional to basic ones}

In this subsection, we let $\Wbb$ be a unitary \emph{simple} $\Vbb$-module. Let
\begin{align*}
\fk P=(\Pbb^1;x_1,x_2,x_3;\eta_1,\eta_2,\eta_3)
\end{align*}
be a positive trinion. By Exp. \ref{lb31}, we have
\begin{align*}
\dim \scr T_{\fk P}^*(\Wbb\otimes\Wbb'\otimes\Vbb)=1
\end{align*}
Therefore, every $\upphi\in\scr T_{\fk P}^*(\Wbb\otimes\Wbb'\otimes\Vbb)$ is proportional to the standard conformal block $\upomega_{\Wbb,\fk P}$.

\begin{df}
We say that $\upphi\in\scr T_{\fk P}^*(\Wbb\otimes\Wbb'\otimes\Vbb)$ is \textbf{positively proportional} (resp. \textbf{strictly positively proportional}) to $\upomega_{\Wbb,\fk P}$ if there exists $\lambda\geq0$ (resp. $\lambda>0$) such that $\upphi=\lambda\upomega_{\Wbb,\fk P}$.
\end{df}

\begin{pp}\label{lb32}
Let $\upphi\in\scr T_{\fk P}^*(\Wbb\otimes\Wbb'\otimes\Vbb)$, and choose a non-zero vector $w\in\Wbb$. Then the following are equivalent.
\begin{enumerate}[label=(\arabic*)]
\item $\upphi$ is positively proportional (resp. strictly positively proportional) to $\upomega_{\Wbb,\fk P}$.
\item $\upphi(w\otimes \Co w\otimes \id)$ is $\geq 0$ (resp. $>0$).
\end{enumerate}
\end{pp}

\begin{proof}
We write $\upomega_{\Wbb,\fk P}$ as $\upomega$ for simplicity. Choose $\lambda\in\Cbb$ such that $\upphi=\lambda\upomega$. To prove the equivalence of (1) and (2), it suffices to show that $\upomega(w\otimes\Co w\otimes\id)>0$. By performing a biholomorphism, we may assume
\begin{align*}
\fk P=(\Pbb^1;0,\infty,\im;\eta,\eta^*,\xi)
\end{align*}
where $\eta^*(z)=\ovl{\eta(1/\ovl z)}$. Let $\alpha\in\Gbb$ be the inverse map of $\eta$, which is also a local coordinate at $0$. Then by Thm. \ref{lb28}, 
\begin{align*}
&\upomega(w\otimes\Co w\otimes\id)=\upomega_{\Wbb,\fk P_0}(\mc U_0(\alpha)w\otimes \mc U_0(\alpha^*)\Co w\otimes \id)\\
\xlongequal{\eqref{eq47}}&\upomega_{\Wbb,\fk P_0}(\mc U_0(\alpha)w\otimes \Co\mc U_0(\alpha) w\otimes \id)=\bk{\mc U_0(\alpha) w|\mc U_0(\alpha) w}
\end{align*}
which is $>0$ because $w\neq 0$ and $\mc U_0(\alpha)$ is invertible.
\end{proof}

\begin{co}\label{lb43}
Choose a neighborhood $U$ of $x_1$ on which $\eta_1$ is defined and univalent. Let $U^*=\{x^*:x\in U\}$, which is a neighborhood of $x_2=x_1^*$ on which $\eta_2=\eta_1^*$ is defined and univalent. Assume that $U\cap U^*=\emptyset$.  Choose $r>0$ such that
\begin{align*}
\eta_1^{-1}(U)\supset\Dbb_r
\end{align*}
and hence $\eta_2^{-1}(U^*)\supset\Dbb_r$. Choose $\tau\in\Dbb_r$. Define a linear map $\upphi:\Wbb\otimes\Wbb'\otimes\Vbb\rightarrow\Cbb$ by 
\begin{align*}
\upphi(w_1\otimes \Co w_2\otimes v)=\upomega_{\Wbb,\fk P}\big(e^{\tau L_{-1}}w_1\otimes e^{\ovl\tau L_{-1}}\Co w_2\otimes v \big)
\end{align*}
which (by Cor. \ref{lb25}) converges absolutely to an element of $\scr T_{\fk Q}^*(\Wbb\otimes\Wbb'\otimes\Vbb)$ where
\begin{align*}
\fk Q=(\Pbb^1;\eta_1^{-1}(\tau),\eta_2^{-1}(\ovl\tau),x_3;\eta_1-\tau,\eta_2-\ovl\tau,\eta_3)
\end{align*}
is clearly a positive trinion. Then $\upphi$ is strictly positively proportional to $\upomega_{\Wbb,\fk Q}$.
\end{co}

\begin{proof}
Let $R_n=\sum_{k=0}^n\frac{(\tau L_{-1})^k}{k!}$. Then
\begin{align*}
&\upphi(w\otimes\Co w\otimes\id)=\upomega_{\Wbb,\fk P}(e^{\tau L_{-1}}w\otimes \Co e^{\tau L_{-1}}w\otimes \id)\\
=&\lim_{n\rightarrow \infty} \upomega_{\Wbb,\fk P}(R_n w\otimes \Co R_n w\otimes\id)
\end{align*}
which is $\geq 0$ because $\upomega_{\Wbb,\fk P}(R_n w\otimes \Co R_n w\otimes\id)\geq 0$ by Prop. \ref{lb32}. So $\upphi=\lambda\upomega_{\Wbb,\fk P}$ for some $\lambda\geq 0$ by Prop. \ref{lb32}. 

Let us show that $\lambda>0$. Suppose $\lambda=0$. Then $\upphi=0$. We write the original $\tau$ as $\tau_0$ and let $\tau$ denote a complex variable. Consider the holomorphic function
\begin{align*}
f(\tau,z)=\upomega_{\Wbb,\fk P}\big(e^{\tau L_{-1}}w_1\otimes e^{z L_{-1}}\Co w_2\otimes v \big)
\end{align*} 
on $\Dbb_r\times\Dbb_r$. Then for each $m,n\in\Nbb$,
\begin{align*}
&\partial_\tau^m\partial_z^n f(\tau,z)|_{\tau_0,\ovl{\tau_0}}=\upomega_{\Wbb,\fk P}\big(e^{\tau_0 L_{-1}}L_{-1}^mw_1\otimes e^{\ovl{\tau_0} L_{-1}}L_{-1}^n\Co w_2\otimes v \big)\\
=&\upphi(L_{-1}^mw_1\otimes L_{-1}^n\Co w_2\otimes v)=0.
\end{align*}
So $f= 0$. Then $f(0,0)=0$ implies $\upomega_{\Wbb,\fk P}=0$, impossible.
\end{proof}

\section{Geometric positivity of fusion products}

In this section, we let $\Vbb$ be a unitary $C_2$-cofinite VOA. We shall only consider semisimple unitary modules of $\Vbb$. Let $\zeta$ be the standard coordinate of $\Cbb$.

\subsection{Geometric positivity of fusion}\label{lb49}

Choose positive trinions
\begin{align*}
\fk P=(\Pbb^1;x_1,x_2,x_3;\eta_1,\eta_2,\eta_3)\qquad \fk Q=(\Pbb^1;y_1,y_2,y_3;\mu_1,\mu_2,\mu_3)
\end{align*}
Recall that $\eta_3,\mu_3$ are M\"obius. By Prop. \ref{lb34} (or by $\eta_3^*=\eta_3,\mu_3^*=\mu_3$),
\begin{align*}
\eta_3(x_1)=\ovl{\eta_3(x_2)}\qquad \mu_3(y_1)=\ovl{\mu_3(y_2)}.
\end{align*}

We want to sew $\fk P\sqcup\fk Q$ along $x_3,y_3$. So we need to choose $\delta,\rho>0$ such that Asmp. \ref{lb9} and \eqref{eq50} are satisfied. (We write the $r$ in \eqref{eq50} as $\delta$ here.) Therefore, we let $\delta,\rho\in(0,+\infty)$ be
\begin{align*}
\delta=|\eta_3(x_1)|\qquad \rho=|\mu_3(y_1)|
\end{align*}
Then for each $0<q<\delta\rho$, we can define the sewing
\begin{align*}
\fk P\#_q\fk Q=\mc S_q(\fk P\sqcup\fk Q)
\end{align*}
along $x_3,y_3$ with parameter $q$. Corresponding to this sewing, if $\upphi\in\scr T_{\fk P}^*(\Wbb_1'\otimes\Wbb_1\otimes\Vbb)$ and $\uppsi\in \scr T_{\fk Q}^*(\Wbb_2\otimes\Wbb_2'\otimes\Vbb)$ where $\Wbb_1,\Wbb_2$ are unitary $\Vbb$-modules, we can define sewing
\begin{align*}
\upphi\#_q\uppsi=\mc S_q(\upphi\otimes\uppsi)
\end{align*}
which, by Cor. \ref{lb25}, converges absolutely (as a power series of $q$) to an element 
\begin{align*}
\scr T_{\fk P\#_q\fk Q}^*(\Wbb_1'\otimes\Wbb_1\otimes\Wbb_2\otimes\Wbb_2')
\end{align*}
Note that this sewing is independent of the choice of $\arg q$ since $L_0$ has integral spectrum on $\Vbb$.

\begin{df}\label{lb39}
Choose unitary semisimple $\Vbb$-modules $\Wbb_1,\Wbb_2$. Associate $\Wbb_1',\Wbb_1,\Vbb$ to the marked points $x_1,x_2,x_3$ and $\Wbb_2,\Wbb_2',\Vbb$ to $y_1,y_2,y_3$ respectively. Choose $0<q<\delta\rho$ with $\arg q=0$. We say that $\Wbb_1$ and $\Wbb_2$ have \textbf{geometrically positive fusion (product)} (or that $\Wbb_1\boxtimes\Wbb_2$ is \textbf{geometrically positive}) if the conformal block
\begin{align*}
\upomega_{\Wbb_1',\fk P}\#_q\upomega_{\Wbb_2,\fk Q}:\Wbb_1'\otimes\Wbb_1\otimes\Wbb_2\otimes\Wbb_2'\rightarrow\Cbb
\end{align*}
associated to $\fk P\#_q\fk Q$ is \textbf{reflection positive} in the sense that for each $n\in\Zbb_+$ and each $w_1,\dots,w_n\in\Wbb_1$, $\wtd w_1,\dots,\wtd w_n\in\Wbb_2$, we have
\begin{align}\label{eq102}
\sum_{k,l=1}^n\upomega_{\Wbb_1',\fk P}\#_q\upomega_{\Wbb_2,\fk Q}\big(\Co w_l\otimes  w_k\otimes  \wtd w_k \otimes \Co\wtd w_l\big)\geq0
\end{align}
\end{df}

\begin{lm}\label{lb40}
The notion of geometric positivity of the fusion product of $\Wbb_1,\Wbb_2$ is independent of the choice of the positive trinions $\fk P,\fk Q$ and the sewing parameter $q$.
\end{lm}

We will explain this lemma in the next subsection. We will also show that $\fk P\#_q\fk Q$ equals $\fk R\#_{q'}\fk R^*$ for some  trinion $\fk R$ and some $q'>0$, and that the reflection positivity defined above agrees with that defined in Def. \ref{lb33}. Let us first make some simplifications.

\begin{rem}\label{lb36}
Given $0<q<\delta\rho$, we can find some $q_1,q_2$ with $\arg q_1=\arg q_2=0$ such that
\begin{gather*}
q=q_1q_2\qquad 0<q_1<\delta\qquad 0<q_2<\rho
\end{gather*}
(e.g. $q_1=\sqrt{q\delta/\rho}$ and $q_2=\sqrt{q\rho/\delta}$). Let $\wtd\eta_3=\eta_3/q_1$ and $\wtd\mu_3=\mu_3/q_2$. Then
\begin{align*}
\wtd{\fk P}=(\Pbb^1;x_1,x_2,x_3;\eta_1,\eta_2,\wtd\eta_3)\qquad \wtd{\fk Q}=(\Pbb^1;y_1,y_2,y_3;\mu_1,\mu_2,\wtd\mu_3)
\end{align*}
are positive trinions, and we clearly have
\begin{gather*}
\wtd{\fk P}\#_1\wtd{\fk Q}=\fk P\#_q\fk Q\\
\upomega_{\Wbb_1',\wtd{\fk P}}=\upomega_{\Wbb_1',\fk P}\circ(\id\otimes\id\otimes q_1^{L_0})\qquad \upomega_{\Wbb_2,\wtd{\fk Q}}=\upomega_{\Wbb_2,\fk Q}\circ(\id\otimes\id\otimes q_2^{L_0})\\
\upomega_{\Wbb_1',\wtd{\fk P}}\#_1\upomega_{\Wbb_2,\wtd{\fk Q}}=\upomega_{\Wbb_1',\fk P}\#_q\upomega_{\Wbb_2,\fk Q}
\end{gather*}
Thus, in the study of geometric positivity of fusion, 
\begin{align*}
\text{we may assume that $q=1$ and $\arg q=0$ (and hence $\delta\rho>1$).}
\end{align*}
\end{rem}

\begin{rem}\label{lb41}
Choose orthogonal irreducible decompositions
\begin{align}
\Wbb_1=\bigoplus_i \Wbb_{1,i}\qquad \Wbb_2=\bigoplus_\alpha \Wbb_{2,\alpha} \label{eq75}
\end{align}
By the description of basic conformal blocks in Thm. \ref{lb28}, 
the restriction of $\upomega_{\Wbb'_1,\fk P}$ to $\Wbb_{1,i}'\otimes\Wbb_{1,j}\otimes\Vbb$ is zero if $i\neq j$, and similarly, the restriction of $\upomega_{\Wbb_2,\fk Q}$ to $\Wbb_{2,\alpha}\otimes\Wbb_{2,\beta}'\otimes\Vbb$ is zero $\alpha\neq\beta$. It follows that
\begin{gather*}
\upomega_{\Wbb_1',\fk P}\#_q\upomega_{\Wbb_2,\fk Q}\text{ is reflection positive}\\
\Updownarrow\\
\upomega_{\Wbb_{1,i}',\fk P}\#_q\upomega_{\Wbb_{2,\alpha},\fk Q}\text{ is reflection positive for each }i,\alpha
\end{gather*}
Thus, in the study of geometric positivity of fusion, we may assume that the unitary $\Vbb$-module $\Wbb_1,\Wbb_2$ are irreducible.
\end{rem}

\subsection{The canonical equivalence $\fk P\#_1\fk Q\simeq\fk R\#_1\fk R^*$}\label{lb46}

Let $\fk P,\fk Q$ be as in the previous subsection. By Prop. \ref{lb34} and \ref{lb35}, we can assume that
\begin{subequations}\label{eq58}
\begin{gather}
\fk P=\big(\Pbb^1;\ovl{\theta_1},\theta_1,\infty;\eta,\eta^*,1/\zeta\big)  \label{eq56}\\
\fk Q=\big(\Pbb^1;z_2,1/\ovl{z_2},\im;\mu,\mu^*,\varpi\big)  \label{eq57}
\end{gather}
\end{subequations}
where $\zeta$ is the standard coordinate of $\Cbb$, $\theta_1,z_2\in\Cbb$ satisfy
\begin{align}
\Imag(\theta_1)<0,\qquad |z_2|<1,  \label{eq51}
\end{align}
and also
\begin{gather*}
\eta^*(z)=\ovl{\eta(\ovl z)}\qquad \mu^*(z)=\ovl{\mu(1/\ovl z)}\\
\varpi=\frac{\im(\zeta-\im)}{\zeta+\im}
\end{gather*}
We assume the condition $\delta\rho>1$ in Rem. \ref{lb36}, which in the present setting reads
\begin{align}
|\theta_1|<|\varpi(z_2)|.  \label{eq59}
\end{align}

\begin{rem}
One easily calculates that the inverse of $\varpi$ equals the reciprocal of $\varpi$:
\begin{align*}
\varpi^{-1}=1/\varpi=\frac{\zeta+\im}{\im(\zeta-\im)}
\end{align*}
We set
\begin{align*}
z_1=\varpi^{-1}(\theta_1)=1/\varpi(\theta_1).
\end{align*}
Then \eqref{eq51} is equivalent to
\begin{align*}
|z_1|<1,\qquad |z_2|<1.
\end{align*}
We will use the symbol $z_1$ frequently in the following discussions.
\end{rem}

We make some crucial geometric observations:

\begin{obs}\label{lb37}
Take sewing $\fk P\#_1\fk Q$ along $\infty$ and $\im$. Namely (cf. Def. \ref{lb8}), we remove $\infty\in\fk P_1$ and $\im\in\fk Q$, and glue the remaining part using the relation that
\begin{gather*}
\gamma_1\in\fk P\setminus\{\infty\}\quad\thicksim\quad \gamma_2\in\fk Q\setminus\{\im\}\\
\Updownarrow\\
\gamma_1=\varpi(\gamma_2)
\end{gather*}
Then we have an equivalence
\begin{equation}\label{eq54}
\tcboxmath{\begin{tikzcd}
\fk P\#_1\fk Q=(\Pbb^1\#_1\Pbb^1;\ovl{\theta_1},\theta_1,z_2,1/\ovl{z_2};\eta,\eta^*,\mu,\mu^*) \arrow[d, "\simeq"] \\
\fk S\coloneq\big(\Pbb^1;1/\ovl{z_1},z_1,z_2,1/\ovl{z_2};\eta\circ\varpi,\eta^*\circ\varpi,\mu,\mu^*\big)    \arrow[u]                 
\end{tikzcd}}
\end{equation}
by making the identifications
\begin{gather*}
\gamma_1\in\fk P\setminus\{\infty\}\quad\simeq\quad \varpi^{-1}(\gamma_1)\in\fk S\\
\gamma_2\in\fk Q\setminus\{\im\}\quad\simeq\quad \gamma_2\in\fk S
\end{gather*}
Notice that
\begin{align*}
\eta^*\circ\varpi(z)=\ovl{\eta\circ\varpi(1/\ovl z)}
\end{align*}
since, by Exp. \ref{lb27}, the pullback of the involution $z\mapsto \ovl z$ along $\varpi$ is $z\mapsto 1/\ovl z$.
\end{obs}

\begin{rem}
The $\fk S$ in Obs. \ref{lb37} is a self-conjugate $4$-pointed sphere with local coordinates, if we define the involution $\fk S$ to be
\begin{align*}
*: z\in\fk S\mapsto 1/\ovl z
\end{align*}
This involution agrees with that of $\fk P$ and $\fk Q$. In general, if we sew a pair of self-conjugate pointed compact Riemann surfaces with local coordinates, we also get a self-conjugate one.
\end{rem}

\begin{obs}\label{lb38}
We continue the discussion in Obs. \ref{lb37}. Let
\begin{align}
\fk R=(\Pbb^1;z_1,z_2,\infty;\eta^*\circ\varpi,\mu,1/\zeta)  \label{eq72}
\end{align}
Then we can make an explicit realization of $\fk R^*$:
\begin{align*}
\fk R^*=(\Pbb^1;1/\ovl{z_1},1/\ovl{z_2},0;\eta\circ\varpi,\mu^*,\zeta)
\end{align*}
The sewing $\fk R\#_1\fk R^*$ is given by
\begin{gather*}
\gamma_1\in\fk R\setminus\{\infty\}\quad\thicksim\quad \gamma_2\in\fk R^*\setminus\{0\}\\
\Updownarrow\\
\gamma_1=\gamma_2
\end{gather*}
Then we have a canonical identification
\begin{align*}
\tcboxmath{\fk R\#_1\fk R^*\simeq\big(\Pbb^1;z_1,z_2,1/\ovl{z_1},1/\ovl{z_2};\eta^*\circ\varpi,\mu,\eta\circ\varpi,\mu^*\big)}
\end{align*}
We change the order of marked points by moving the third one $1/\ovl{z_1}$ to the third position. Then we obtain a canonical equivalence $\fk R\#_1\fk R^*\simeq\fk S$, which is also equivalent to $\fk P\#_1\fk Q$ through \eqref{eq54}.
\end{obs}

We summarize the above observations by the formula
\begin{align}
\fk P\#_1\fk Q\simeq \fk R\#_1\fk R^*  \label{eq55}
\end{align}
We emphasize that in the above equivalence, we have changed the order of marked points. (In the case of positive trinions, the order cannot be altered. But in the current situation, the order is immaterial.)

\begin{rem}
Using Def. \ref{lb33}, we can define reflection positivity for elements of $\scr T_{\fk R\#_1\fk R^*}^*(\Wbb_1\otimes\Wbb_2\otimes\Wbb_1'\otimes\Wbb_2')$. This can be translated to the reflection positivity of elements of $\scr T_{\fk P\#_1\fk Q}^*(\Wbb_1'\otimes\Wbb_1\otimes\Wbb_2\otimes\Wbb_2')$  thanks to the equivalence \eqref{eq55}. It is clear that this notion of reflection positivity agrees with the one in Def. \ref{lb39}.
\end{rem}

\begin{proof}[\textbf{Proof of Lem. \ref{lb40}}]
By Rem. \ref{lb36} and \ref{lb41}, we may assume that $q=1,\arg q=0$ and that $\Wbb_1,\Wbb_2$ are irreducible. Choose $\fk P=\eqref{eq56}$ and $\fk Q=\eqref{eq57}$ and assume that $\upomega_{\Wbb_1',\fk P}\#_1\upomega_{\Wbb_2,\fk Q}$ is reflection positive. By Thm. \ref{lb28}-(b) and Prop. \ref{lb42}, if we change the local coordinates $\eta,\mu$ of $\fk P,\fk Q$, the reflection positivity is preserved.

Thus, we may assume that $\eta=\zeta-\ovl{\theta_1}$ and $\mu=\zeta-z_2$ in \eqref{eq58}. By Cor. \ref{lb43} and Prop. \ref{lb44}, the reflection positivity of $\upomega_{\Wbb_1',\fk P}\#_1\upomega_{\Wbb_2,\fk Q}$ is preserved if we choose $\tau_1,\tau_2\in\Dbb_\varepsilon$ (for some small $\varepsilon>0$) and replace $\theta_1$ by $\theta_1+\tau_1$ and replace $z_2$ by $z_2+\tau_2$. Namely, the reflection positivity is preserved by local translation. Thus, to prove that the reflection positivity is irrelevant to the choice of $\theta_1$ and $z_2$, it suffices to show that the set of $(\theta_1,z_2)$ satisfying \eqref{eq51} and \eqref{eq59} is connected, namely, that
\begin{align*}
O=\{(\theta_1,z_2)\in\Cbb^2:\Imag(\theta_1)<0,|z_2|<1,|\theta_1|<|\varpi(z_2)|\}
\end{align*}
is connected. But this follows from the fact that $\id\times\varpi$ sends $O$ biholomorphically to 
\begin{align*}
O'=\{(\theta_1,\theta_2)\in\Cbb^2;\Imag(\theta_1)<0,\Imag(\theta_2)<0,|\theta_1|<|\theta_2|\}
\end{align*}
which is clearly connected.
\end{proof}

\subsection{The canonical inner product on $\Wbb_1\boxtimes_{\fk R}\Wbb_2$}\label{lb47}

We assume that $\Vbb$ is unitary, $C_2$-cofinite, and rational. We assume that all irreducible $\Vbb$-modules are unitary. Let $\Wbb_1,\Wbb_2$ be unitary $\Vbb$-modules, which are automatically semisimple since $\Vbb$ is rational. 

In this subsection, we use the Sewing-factorization Thm. \ref{lb13} to define a canonical inner product on the $\Vbb$-module $\Wbb_1\boxtimes\Wbb_2$ when this fusion is geometrically positive. We will see in the next section that when $\Wbb_1,\Wbb_2$ are irreducible, this inner product equals, up to a positive scalar multiplication, the inner product defined  in \cite{Gui19b}. If the fusion of all unitary irreducible $\Vbb$-modules are geometrically positive, we will call $\Vbb$ a \textbf{completely unitary} VOA. In this case, the inner products on all $\Wbb_1\boxtimes\Wbb_2$ defined in this way make the braided tensor category of $\Vbb$-modules unitary (\cite[Thm. 7.8]{Gui19b}). Due to this fact, one is expected to show that the  connections on the bundles of conformal blocks (as defined in \cite[Chapter 17]{FB04}) is projectively unitary, \emph{and that the sewing map $\fk S_q$ in Thm. \ref{lb13} is also unitary}, as pointed out in \cite[Thm. 10.10]{Kir98}. This subsection is not necessarily needed in later proofs, and can be skipped on first reading.

For each equivalence class of irreducible unitary $\Vbb$-modules we choose a representative $\Mbb$, and let all these representatives form a (necessarily finite) set $\mc E$. If for each $\Mbb\in\mc E$ we choose a basis $(\upphi_{\Mbb,\alpha})_{\alpha\in\mc A_\Mbb}$ of the vector space $\scr T_{\fk R}^*(\Wbb_1\otimes\Wbb_2\otimes\Mbb')$ (which is finite-dimensional by Thm. \ref{lb11}), then by Thm. \ref{lb13},
\begin{align}
\{\upphi_{\Mbb,\alpha}\#_1\upphi_{\Mbb,\beta}^*:\Mbb\in\mc E,\alpha,\beta\in\mc A_{\Mbb}\}  \label{eq60}
\end{align}
form a basis of $\scr T_{\fk R\#_1\fk R^*}^*(\Wbb_1\otimes\Wbb_2\otimes\Wbb_1'\otimes\Wbb_2')$.

We define the tensor product module
\begin{align}\label{eq66}
\Wbb_1\boxtimes_{\fk R}\Wbb_2=\bigoplus_{\Mbb\in\mc E}\Mbb\otimes\scr T_{\fk R}(\Wbb_1\otimes\Wbb_2\otimes\Mbb')
\end{align}
where
\begin{align*}
\scr T_{\fk R}(\Wbb_1\otimes\Wbb_2\otimes\Mbb')\text{ is the dual space of }\scr T_{\fk R}^*(\Wbb_1\otimes\Wbb_2\otimes\Mbb')
\end{align*}
So we have
\begin{align*}
(\Wbb_1\boxtimes_{\fk R}\Wbb_2)'=\bigoplus_{\Mbb\in\mc E}\Mbb'\otimes\scr T_{\fk R}^*(\Wbb_1\otimes\Wbb_2\otimes\Mbb')
\end{align*}
Define the canonical conformal block $\Uppsi:\Wbb_1\otimes\Wbb_2\otimes (\Wbb_1\boxtimes_{\fk R}\Wbb_2)'\rightarrow\Cbb$ associated to $\fk R$ by
\begin{align}
\Uppsi\big(w_1\otimes w_2\otimes (m'\otimes\upphi)\big)=\upphi(w_1\otimes w_2\otimes m') \label{eq65}
\end{align}
where $w_1\in\Wbb_1,w_2\in\Wbb_2,m'\in\Mbb'$, and $\upphi\in\scr T_{\fk R}^*(\Wbb_1\otimes\Wbb_2\otimes\Mbb')$. 

Choose an inner product for $\scr T_{\fk R}^*(\Wbb_1\otimes\Wbb_2\otimes\Mbb')$ for each $\Mbb\in\mc E$, which gives a corresponding inner product of $\scr T_{\fk R}^*(\Wbb_1\otimes\Wbb_2\otimes\Mbb')$. Then $\Wbb_1\boxtimes_{\fk R}\Wbb_2$ is equipped with an inner product under which it is a unitary $\Vbb$-module. By the fact that \eqref{eq60} form a basis, each element of $\scr T_{\fk R\#_1\fk R^*}^*(\Wbb_1\otimes\Wbb_2\otimes\Wbb_1'\otimes\Wbb_2')$ can be written as $A\Uppsi\#_1\Uppsi^*$ for a unique $A\in\End_\Vbb(\Wbb_1\boxtimes_{\fk R}\Wbb_2)$. (Recall \eqref{eq61} for the notation.)

We choose $A$ to be the unique one such that
\begin{align}
\upomega_{\Wbb_1',\fk P}\#_1\upomega_{\Wbb_2,\fk Q}=A\Uppsi\#_1\Uppsi^*. \label{eq64}
\end{align}
More precisely, if we let $\{\upphi_{\Mbb,\alpha}\}_{\alpha\in\mc A_\Mbb}$ be an orthonormal basis of $\scr T_{\fk R}^*(\Wbb_1\otimes\Wbb_2\otimes\Mbb')$, write
\begin{align*}
A=\bigoplus_{\Mbb\in\mc E}\id_\Mbb\otimes A_\Mbb
\end{align*}
according to the decomposition \eqref{eq66} where $A_\Mbb\in\End \scr T_{\fk R}(\Wbb_1\otimes\Wbb_2\otimes\Mbb')$, and let $(A_{\Mbb,\alpha,\beta})$ be the matrix representation of $A_\Mbb$ with respect to the dual orthonormal basis of $\{\upphi_{\Mbb,\alpha}\}_{\alpha\in\mc A_\Mbb}$, then
\begin{align}
\upomega_{\Wbb_1',\fk P}\#_1\upomega_{\Wbb_2,\fk Q}=\sum_{\Mbb\in\mc E}\sum_{\alpha,\beta}A_{\Mbb,\alpha,\beta}\upphi_{\Mbb,\beta}\#_1\upphi_{\Mbb,\alpha}^*. \label{eq67}
\end{align}
Since $\upomega_{\Wbb_1',\fk P}$ and $\upomega_{\Wbb_2,\fk Q}$ are self-conjugate conformal blocks (Thm. \ref{lb28}),   it is not hard to see that  $\upomega_{\Wbb_1',\fk P}\#_1\upomega_{\Wbb_2,\fk Q}$ is self conjugate. Thus $A=A^*$ by Lem. \ref{lb70}.

\begin{rem}\label{lb57}
Choose some new positive trinions $\wtd{\fk P},\wtd {\fk Q}$ as in Subsec. \ref{lb46}, and choose $\wtd{\fk R}$ accordingly so that we have a canonical equivalence $\wtd{\fk P}\#_1\wtd {\fk Q}\simeq\wtd{\fk R}\#_1\wtd{\fk R}^*$. Then one can obtain  $\wtd{\fk R}$ from $\fk R$ making some changes of coordinates and moving the marked points. We fix such process of constructing $\wtd{\fk R}$ from $\fk R$. Then using formulas \eqref{eq62} and \eqref{eq63}, we can construct $\wtd\upphi\in\scr T_{\wtd{\fk R}}^*(\Wbb_1\otimes\Wbb_2\otimes\Mbb')$ from each $\upphi\in \scr T_{\fk R}^*(\Wbb_1\otimes\Wbb_2\otimes\Mbb')$. We thus obtain a linear isomorphism
\begin{gather*}
\scr T_{\fk R}^*(\Wbb_1\otimes\Wbb_2\otimes\Mbb')\rightarrow \scr T_{\wtd{\fk R}}^*(\Wbb_1\otimes\Wbb_2\otimes\Mbb'),\qquad\upphi\mapsto\wtd\upphi
\end{gather*}
Thus, we can define an inner product on $\scr T_{\wtd{\fk R}}^*(\Wbb_1\otimes\Wbb_2\otimes\Mbb')$ such that the above map is unitary. (Note that the above map is not unique since it depends on the process we choose to obtain $\wtd{\fk R}$ from $\fk R$.) This inner product is described by the fact that if $\{\upphi_{\Mbb,\alpha}\}_{\alpha\in\mc A_\Mbb}$ is an orthonormal basis of $\scr T_{\fk R}^*(\Wbb_1\otimes\Wbb_2\otimes\Mbb')$ then $\{\wtd\upphi_{\Mbb,\alpha}\}_{\alpha\in\mc A_\Mbb}$ is an orthonormal basis of $\scr T_{\wtd{\fk R}}^*(\Wbb_1\otimes\Wbb_2\otimes\Mbb')$. Thus, we obtain a unitary equivalence of modules
\begin{align*}
U:\Wbb_1\boxtimes_{\fk R}\Wbb_2\rightarrow\Wbb_1\boxtimes_{\wtd{\fk R}}\Wbb_2
\end{align*}

Let $\wtd\Psi:\Wbb_1\otimes\Wbb_2\otimes(\Wbb_1\boxtimes_{\wtd{\fk R}}\Wbb_2)'\rightarrow\Cbb$ be the canonical conformal block associated to $\wtd{\fk R}$ defined in a similar way as \eqref{eq65}. Choose a self-adjoint $\wtd A\in\End_\Vbb(\Wbb_1\boxtimes_{\wtd{\fk R}}\Wbb_2)$ such that
\begin{align}
\upomega_{\Wbb_1',\wtd{\fk P}}\#_1\upomega_{\Wbb_2,\wtd{\fk Q}}=\wtd A\wtd\Uppsi\#_1\wtd\Uppsi^*
\end{align}
By Thm. \ref{lb28}-(b) and Cor. \ref{lb43}, and by using the orthnormal bases $\{\upphi_{\Mbb,\alpha}\}_{\alpha\in\mc A_\Mbb}$ and $\{\wtd\upphi_{\Mbb,\alpha}\}_{\alpha\in\mc A_\Mbb}$ mentioned above,  we see that if $\Wbb_1,\Wbb_2$ are irreducible, then 
\begin{align*}
\wtd A=\lambda\cdot  UAU^{-1}
\end{align*}
for some $\lambda>0$. Thus, even if $\Wbb_1,\Wbb_2$ are not assumed irreducible, or if we choose a completely different inner product for $\scr T_{\wtd{\fk R}}^*(\Wbb_1\otimes\Wbb_2\otimes\Mbb')$ and define $\wtd\Uppsi$ and $\wtd A$ accordingly, the signature of each Hermitian matrix $\wtd A_\Mbb$ in the decomposition
\begin{align*}
\wtd A=\bigoplus_{\Mbb\in\mc E}\id_\Mbb\otimes \wtd A_\Mbb
\end{align*}
agrees with that of $A_\Mbb$. In particular, $A$ is invertible iff $\wtd A$ is so, and $A$ is positive iff $\wtd A$ is so. \hfill\qedsymbol
\end{rem}

According to the above long remark, if  $A$ is invertible and positive for one choice of $\fk P,\fk Q$, it is so for any other $\fk P,\fk Q$. We will see in the next section that $A$ is automatically invertible for some choice of $\fk P,\fk Q$. (This is a consequence of Huang's rigidity theorem (\cite{Hua08}) for the tensor category of $\Vbb$-modules.) 

By Thm. \ref{lb45} and \ref{lb21}, we know that $A\geq 0$ iff $\Wbb_1\boxtimes\Wbb_2$ is geometrically positive. Suppose $\Wbb_1\boxtimes\Wbb_2$ is geometrically positive. Then we can choose a unique inner product on $\scr T_{\fk R}^*(\Wbb_1\otimes\Wbb_2\otimes\Mbb')$ (for each $\Mbb\in\mc E$) such that $A$ becomes the identity operator. In this way, we get the correct inner product on $\Wbb_1\boxtimes_{\fk R}\Wbb_2$. Under this inner product, we have
\begin{align}
\upomega_{\Wbb_1',\fk P}\#_1\upomega_{\Wbb_2,\fk Q}=\Uppsi\#_1\Uppsi^*.  \label{eq68}
\end{align}

\begin{rem}
Recall the conformal block $\Uppsi\in\scr T_{\fk R}^*(\Wbb_1\otimes\Wbb_2\otimes(\Wbb_1\boxtimes_{\fk R}\Wbb_2)')$. As in Prop. \ref{lb17}, we consider the corresponding linear map into the algebraic completion:
\begin{align*}
T_\Uppsi:\Wbb_1\boxtimes\Wbb_2\rightarrow \ovl{\Wbb_1\boxtimes_{\fk R}\Wbb_2}
\end{align*}
By Prop. \ref{lb18} and Thm. \ref{lb21}, the range of $T_\Uppsi$ is norm-dense in the Hilbert space completion $\mc H_{\Wbb_1\boxtimes_{\fk R}\Wbb_2}$.  Thus, by \eqref{eq68} and Prop. \ref{lb17}, for each $n\in\Zbb_+$, $w_1,\dots,w_n\in\Wbb_1$, and $\wtd w_1,\dots,\wtd w_n\in\Wbb_2$, if we set
\begin{align*}
\wbf=\sum_{k=1}^n w_k\otimes \wtd w_k,\quad\mbf m=\sum_{l=1}^n m_l\otimes\wtd m_l\qquad\in \Wbb_1\otimes\Wbb_2 
\end{align*}
we have
\begin{align}
\bigbk{T_{\Uppsi}(\wbf)\big|T_{\Uppsi}(\mbf m) }=\sum_{k,l=1}^n\upomega_{\Wbb_1',\fk P}\#_1\upomega_{\Wbb_2,\fk Q}\big(\Co m_l\otimes  w_k\otimes  \wtd w_k \otimes \Co\wtd m_l\big)
\end{align}
This formula uniquely determines the inner product on $\Wbb_1\boxtimes_{\fk R}\Wbb_2$. It tells us, roughly speaking, how to put a (possibly degenerate) inner product on $\Wbb_1\otimes\Wbb_2$ so that its Hilbert space completion is $\mc H_{\Wbb_1\boxtimes_{\fk R}\Wbb_2}$.
\end{rem}

\section{Equivalence to the algebraic positivity of fusion products}\label{lb69}

In this section, we assume that $\Vbb$ is CFT-type, unitary, $C_2$-cofinite, and rational. As in Subsec. \ref{lb47}, for each equivalence class of irreducible $\Vbb$-modules we choose a representative $\Mbb$, and let all these representatives form a (necessarily finite) set $\mc E$.

As usual, we let $\zeta$ denote the standard coordinate of $\Cbb$.

\subsection{Various intertwining operators}

Recall the definition of intertwining operators in Exp. \ref{lb48}. Let $\Wbb_1,\Wbb_2,\Wbb_3$ be unitary $\Vbb$-modules. For each $\mc Y\in\mc I{\Wbb_3\choose\Wbb_1\Wbb_2}$, note that for each $z\in\Cbb^\times$ with chosen $\arg z$, $\mc Y(\cdot,z)$ is a linear map $\Wbb_1\otimes\Wbb_2\rightarrow\ovl{\Wbb_3}$. Define the \textbf{conjugate intertwining operator} of $\mc Y$ to be 
\begin{gather*}
\mc Y^*\in\mc I{\Wbb_3'\choose\Wbb_1'\Wbb_2'}\\
\mc Y^*(\Co w_1,\ovl z)\Co w_2=\Co\cdot \mc Y(w_1,z)w_2
\end{gather*}
where $w_1\in\Wbb_1,w_2\in\Wbb_2$, and we set
\begin{align}
\arg\ovl z=\arg z^{-1}=-\arg z  \label{eq92}
\end{align}
Thus, if $w_3\in\Wbb_3$ then
\begin{align}
\bk{\mc Y^*(\Co w_1,\ovl z)\Co w_2|\Co w_3}=\bk{w_3|\mc Y(w_1,z)w_2}\label{eq84}
\end{align}

Define the \textbf{contragredient intertwining operator of $\mc Y$} to be  \index{Y@$\mc Y^*,\mc Y^c$} 
\begin{gather}
\mc Y^c\in\mc I{\Wbb_3'\choose \Wbb_1\Wbb_2'}\nonumber \\
\bk{\mc Y(w_1,z)w_2,w_3'}=\bk{w_2,\mc Y^c(\mc U_0(\vartheta_z)w_1,z^{-1})w_3'} \label{eq85}
\end{gather}
for each $w_1\in\Wbb_1,w_2\in\Wbb_2,w_3\in\Wbb_3'$, where
\begin{align}
\mc U_0(\vartheta_z)=e^{zL_1}(-z^{-2})^{L_0}  \label{eq77}
\end{align}
in which we assume
\begin{align}
\arg (-z^{-2})=\pi-2\arg z  \label{eq78}
\end{align}


The \textbf{adjoint intertwining operator} of $\mc Y$, denoted by $\mc Y^\dagger\in\mc I{\Wbb_2\choose \Wbb_1'\Wbb_3}$, \index{Y@$\mc Y^\dagger$} is defined to be
\begin{align*}
\mc Y^\dagger=\mc Y^{c*}\qquad\in \mc I{\Wbb_2\choose \Wbb_1'\Wbb_3}
\end{align*}
We let $\Gamma_\Wbb$ \index{zz@$\Gamma_\Wbb$} be the unique element of $\mc I{\Vbb\choose \Wbb_1'\Wbb_1}$ (called the \textbf{annihilation operator} associated to $\Wbb$) satisfying (cf. \cite[(1.40)]{Gui19a})
\begin{align}
\bigbk{\Gamma_{\Wbb_1}(w_1',z)w_1,\id}=\bigbk{w_1,e^{z^{-1}L_1}\mc U_0(\vartheta_z)w_1' }. \label{eq83}
\end{align}
for all $w_1\in\Wbb_1,w_1'\in\Wbb_1'$. (Note that the uniqueness is due to Prop. \ref{lb20}.)

Finally, notice that the vertex operation $Y_{\Wbb_2}$ of $\Wbb_2$ is an element of $\mc I{\Wbb_2\choose \Vbb\Wbb_2}$.

\subsection{Algebraic positivity of fusion}

We continue the discussion from the last subsection. Let $\Wbb_1,\Wbb_2$ be unitary $\Vbb$-modules. Let
\begin{align*}
\Wbb_1\boxtimes\Wbb_2=\bigoplus_{\Mbb\in\mc E}\Mbb\otimes\mc I{\Mbb\choose\Wbb_1\Wbb_2}^*
\end{align*}
Assume that $\Wbb_1\boxtimes\Wbb_2$ is unitarizable. This means that $\Mbb$ is unitarizable whenever $\mc I{\Mbb\choose\Wbb_1\Wbb_2}$ is nonzero. Fix a unitary structure on each such $\Mbb$. Fix an arbitrary inner product on $\mc I{\Mbb\choose\Wbb_1\Wbb_2}$ and the corresponding one on $\mc I{\Mbb\choose\Wbb_1\Wbb_2}^*$. Then $\Wbb_1\boxtimes\Wbb_2$ is a unitary $\Vbb$-module. Define
\begin{gather*}
\Pi\in\mc I{\Wbb_1\boxtimes\Wbb_2\choose\Wbb_1~\Wbb_2}\\[0.5ex]
\bk{\Pi(w_1,z)w_2,m'\otimes\mc Y}=\bk{\mc Y(w_1,z)w_2,m'}
\end{gather*}
if $w_1\in\Wbb_1,w_2\in\Wbb_2,m'\in\Mbb'$, $\mc Y\in\mc I{\Mbb\choose\Wbb_1\Wbb_2}$, and $m'\otimes\nu$ is regarded as an element of
\begin{align*}
(\Wbb_1\boxtimes\Wbb_2)'=\bigoplus_{\Mbb\in\mc E}\Mbb'\otimes\mc I{\Mbb\choose\Wbb_1\Wbb_2}
\end{align*}

According to \cite[Chapter 6]{Gui19b}, there is a unique invertible bounded linear map $A\in\End_\Vbb(\Wbb_1\boxtimes\Wbb_2)$ such that for each $z_1,z_1'\in\Cbb^\times$ satisfying
\begin{subequations}
\begin{gather}
0<|z_1'-z_1|<|z_1|<|z_1'|   \label{eq70}\\ 
\arg (z_1'-z_1)=\arg z_1=\arg z_1'   \label{eq71}
\end{gather}
\end{subequations}
and for all $w_1\in\Wbb_1,w_1'\in\Wbb_1'$ we have the fusion relation
\begin{align}
\tcboxmath{Y_{\Wbb_2}\big(\Gamma_{\Wbb_1}(w_1',z_1'-z_1)w_1,z_1 \big)=\Pi^\dagger(w_1',z_1') A\Pi(w_1,z_1)}   \label{eq69}
\end{align}
More precisely, if for each unitary $\Vbb$-module $\Wbb$ and each $r\geq 0$ we let $P_r$ be the projection of $\ovl{\Wbb}$ onto its $L_0$-weight $r$ eigenspace $\Wbb_{(r)}$, then for each $w_2\in\Wbb_2,w_2'\in\Wbb_2'$, we have
\begin{align*}
\sum_{n\in\Nbb}\bigbk{Y_{\Wbb_2}\big(P_n\Gamma_{\Wbb_1}(w_1',z_1'-z_1)w_1,z_1 \big)w_2,w_2' }=\sum_{r\geq 0}\bigbk{\Pi^\dagger(w_1',z_1')P_r A\Pi(w_1,z_1)w_2,w_2'}
\end{align*}
where both sides converge absolutely. (Indeed, the absolute convergence is implied by Thm. \ref{lb12}.) We note that, as pointed out in \cite{Gui19b},  the invertibility of $A$ is due to the rigidity of the tensor category of $\Vbb$-modules proved in \cite{Hua08}.

\begin{df}\label{lb84}
We say that $\Wbb_1$ and $\Wbb_2$ have \textbf{algebraically positive fusion (product)} (or that $\Wbb_1\boxtimes\Wbb_2$ is \textbf{algebraically positive}) if the operator $A$ in \eqref{eq69} is positive, namely, $\bk{A\nu|\nu}\geq0$ for all $\nu\in\Wbb_1\boxtimes\Wbb_2$ (and hence for all $\nu\in\mc H_{\Wbb_1\boxtimes\Wbb_2}$).
\end{df}

\subsection{Main results}\label{lb71}

In the rest of this section, we assume the settings and use freely the notations in Subsec. \ref{lb46}. 

We choose the $z_1$  in \eqref{eq69} to be the same one in Subsec. \ref{lb46}, and choose $z_1'=1/\ovl{z_1}$. But here we shall only consider the special case that
\begin{align*}
z_1=s\im,\qquad z_1'=1/\ovl{z_1}=s^{-1}\im\qquad(0<s<1)
\end{align*}
Let $-t\im=\theta_1=\varpi(z_1)=\im(z_1-\im)/(z_1+\im)$. So $t>0$ satisfies
\begin{align*}
t=\frac{1-s}{1+s}\qquad s=\frac{1-t}{1+t}
\end{align*}

\begin{rem}\label{lb58}
Under the assumption $s<1$, \eqref{eq70} means $0<s^{-1}-s<s$, namely
\begin{align}
\frac{1}{\sqrt 2}<s<1\qquad\text{equivalently}\qquad 0<t<3-2\sqrt 2  \label{eq73}
\end{align}
\end{rem}

Recall $\displaystyle\varpi=\frac{\im(\zeta-\im)}{\zeta+\im}$ and $\varpi^{-1}=1/\varpi$. It is useful to note
\begin{align*}
\varpi\circ\varpi=\varpi^{-1}\circ\varpi^{-1}=1/\zeta
\end{align*}
where we recall that $\zeta$ is the standard coordinate of $\Cbb$. In \eqref{eq58}, we choose $z_2=0,\mu=\zeta$. Then  \eqref{eq59} is automatically satisfied since $|t|<1$. We choose $\eta$ to satisfy
\begin{align*}
\eta^*\circ\varpi=\zeta-s\im.
\end{align*}
(Note that $\eta^*\circ\varpi$ is the local coordinate of $\fk S$ at $z_1=s\im$ (cf. \eqref{eq54}). And $\eta^*(z)=\ovl{\eta(\ovl z)}$.) Namely $\eta^*=\varpi^{-1}-s\im$ and hence $\eta=\varpi+s\im$. To summarize, we assume that \eqref{eq58} is
\begin{subequations}  \label{eq74}
\begin{gather}
\fk P=\big(\Pbb^1;t\im,-t\im,\infty;\varpi+s\im,\varpi^{-1}-s\im,1/\zeta\big)  \\
\fk Q=\big(\Pbb^1;0,\infty,\im;\zeta,1/\zeta,\varpi\big)  
\end{gather}
\end{subequations}
Then  \eqref{eq54} says
\begin{align}
\fk P\#_1\fk Q\simeq\fk S=\big(\Pbb^1;s^{-1}\im,s\im,0,\infty;1/\zeta+s\im,\zeta-s\im,\zeta,1/\zeta \big)
\end{align}
And \eqref{eq72} becomes
\begin{align}
\fk R=\big(\Pbb^1;s\im,0,\infty; \zeta-s\im,\zeta,1/\zeta \big)
\end{align}

Define
\begin{gather}
\begin{gathered}
\Uppsi\in\scr T_{\fk R}^*(\Wbb_1\otimes\Wbb_2\otimes(\Wbb_1\boxtimes\Wbb_2)')\\
\Uppsi(w_1\otimes w_2\otimes \nu')=\bk{\Pi(w_1,s\im)w_2, \nu'}
\end{gathered}
\end{gather}
where $w_1\in\Wbb_1,w_2\in\Wbb_2$, $\nu'\in(\Wbb_1\boxtimes\Wbb_2)'$, and we choose
\begin{align*}
\arg(s\im)=\arg(s^{-1}\im)=\pi/2
\end{align*}

The proof of the following theorem is deferred to Sec. \ref{lb68}

\begin{thm}\label{lb50}
Assume that $\Wbb_1,\Wbb_2$ are unitary irreducible $\Vbb$-modules and $\Wbb_1\boxtimes\Wbb_2$ is a unitarizable $\Vbb$-module. Choose a unitary structure on $\Wbb_1\boxtimes\Wbb_2$. Choose $s$ satisfying \eqref{eq73}, and choose positive trinions $\fk P,\fk Q$ defined by \eqref{eq74}. Let $A$ be as in \eqref{eq69}. Then, in view of the canonical equivalence $\fk P\#_1\fk Q\simeq\fk R\#_1\fk R^*$ (cf. \eqref{eq55}), there exists $\lambda>0$ (possibly depending on $s$) such that
\begin{align}
\upomega_{\Wbb_1',\fk P}\#_1\upomega_{\Wbb_2,\fk Q}=\lambda\cdot A\Uppsi\#_1\Uppsi^*  \label{eq76}
\end{align}
More precisely, for each $w_1,\wtd w_1\in\Wbb_1$ and $w_2,\wtd w_2\in\Wbb_2$,
\begin{align}
\upomega_{\Wbb_1',\fk P}\#_1\upomega_{\Wbb_2,\fk Q}(\Co\wtd w_1\otimes w_1\otimes w_2\otimes\Co\wtd w_2)=\lambda\cdot A\Uppsi\#_1\Uppsi^*(w_1\otimes w_2\otimes\Co\wtd w_1\otimes\Co\wtd w_2 )  \label{eq79}
\end{align}
\end{thm}

\begin{co}\label{lb72}
Assume that $\Wbb_1,\Wbb_2$ are unitary $\Vbb$-modules and $\Wbb_1\boxtimes\Wbb_2$ is a unitarizable $\Vbb$-module. Choose a unitary structure on $\Wbb_1\boxtimes\Wbb_2$. Then $A=A^*$. Moreover, $\Wbb_1\boxtimes\Wbb_2$ is geometrically positive if and only if it is algebraically positive. 
\end{co}

\begin{proof}
It suffices to assume that $\Wbb_1,\Wbb_2$ are irreducible. Fix an inner product on $\Wbb_1\boxtimes\Wbb_2$ so that $\Wbb_1\boxtimes\Wbb_2$ is a unitary $\Vbb$-module. Let $A$ be as in \eqref{eq69}. 

We first note that by Thm. \ref{lb21} and Rem. \ref{lb87}, the conformal block $\Uppsi$ satisfies the same density assumption on $\upphi$ in Thm. \ref{lb45} (with $\Mbb=(\Wbb_1\boxtimes\Wbb_2)'$). Now, by Thm. \ref{lb28}, $\upomega_{\Wbb_1',\fk P}$ and $\upomega_{\Wbb_2,\fk Q}$ are self-conjugate. So \eqref{eq79} defines a self-conjugate conformal block. Therefore, $A=A^*$ by Lem. \ref{lb70}. We set the $\upphi$ in Thm. \ref{lb45} to be $\Uppsi$. Then Thm. \ref{lb45} says that $\Wbb_1\boxtimes\Wbb_2$ is algebraically positive iff \eqref{eq76} is reflection positive. The reflection positivity of the left hand side of \eqref{eq76} means precisely that $\Wbb_1\boxtimes\Wbb_2$ is geometrically positive.
\end{proof}

\begin{df}\label{lb59}
Assume that $\Wbb_1\boxtimes\Wbb_2$ is a unitarizable $\Vbb$-module. If the fusion product $\Wbb_1\boxtimes \Wbb_2$ is either geometrically or algebraically positive, we say that $\Wbb_1\boxtimes \Wbb_2$ is \textbf{positive}.
\end{df}

\begin{thm}\label{lb60}
Let $\Wbb_1,\Wbb_2$ be unitary irreducible $\Vbb$-modules. Assume that $\Wbb_1\boxtimes\Wbb_2$ is an irreducible unitarizable $\Vbb$-module. Then the fusion $\Wbb_1\boxtimes\Wbb_2$ is positive.
\end{thm}

\begin{proof}
Fix an inner product on $\Wbb_1\boxtimes\Wbb_2$ so that $\Wbb_1\boxtimes\Wbb_2$ is a unitary $\Vbb$-module. Since $\Wbb_1\boxtimes\Wbb_2$ is irreducible, the self-adjoint operator $A$ (cf. Cor. \ref{lb72}) in  \eqref{eq76} is a non-zero real number $a$. Assume that $a<0$. Let us find a contradiction.

Step 1. By Thm. \ref{lb50} and \ref{lb45}, for any $s$ satisfying \eqref{eq73} and any $\fk P,\fk Q$ defined as in \eqref{eq74}, $-\upomega_{\Wbb_1',\fk P}\#_1\upomega_{\Wbb_2,\fk Q}$ is reflection positive. Thus, if we define positive trinion
\begin{align*}
\fk X=(\Pbb^1;t\im,-t\im,\infty;\eta,\eta^*,1/\zeta)
\end{align*}
where $\eta$ is a local coordinate at $t\im$ and $\eta^*(z)=\ovl{\eta(\ovl z)}$, then by Prop. \ref{lb42} and Thm. \ref{lb28}-(b), $-\upomega_{\Wbb_1',\fk X}\#_1\upomega_{\Wbb_2,\fk Q}$ is reflection positive.

Step 2. Choose any $0<q<1$ and $\arg q=0$. Let us use Step 1 to show that $-\upomega_{\Wbb_1',\fk P}\#_q\upomega_{\Wbb_2,\fk Q}$ is reflection positive. Notice that
\begin{align*}
\wtd{\fk P}=\big(\Pbb^1;t\im,-t\im,\infty;\varpi+s\im,\varpi^{-1}-s\im,1/q\zeta\big)
\end{align*}
is a positive trinion and is equivalent (via the map $z\mapsto qz$) to
\begin{align*}
\wtd{\fk X}=(\Pbb^1;qt\im,-qt\im,\infty;\eta,\eta^*,1/\zeta)
\end{align*}
where $\displaystyle\eta(z)=\frac{\im(z-q\im)}{z+q\im}+s\im$. (Note that we have $0<qt<3-2\sqrt 2$, cf. Rem. \ref{lb58}.) Thus, by Step 1, $-\upomega_{\Wbb_1',\wtd{\fk P}}\#_1\upomega_{\Wbb_2,\fk Q}$ is reflection positive.

Clearly $\wtd{\fk P}\#_1\fk Q=\fk P\#_q\fk Q$. By Thm. \ref{lb28}-(b), the basic conformal block $\upomega_{\Wbb_1',\wtd{\fk P}}:\Wbb_1'\otimes\Wbb_1\otimes\Vbb\rightarrow\Cbb$ satisfies
\begin{align*}
\upomega_{\Wbb_1',\wtd{\fk P}}=\upomega_{\Wbb_1',\fk P}\circ(\id\otimes\id\otimes q^{L_0}).
\end{align*}
Therefore $-\upomega_{\Wbb_1',\fk P}\#_q\upomega_{\Wbb_2,\fk Q}$ equals $-\upomega_{\Wbb_1',\wtd{\fk P}}\#_1\upomega_{\Wbb_2,\fk Q}$, which is reflection positive.

Step 3. Choose nonzero vectors $w_1\in\Wbb_1,w_2\in\Wbb_2$. Then by Step 2, for each $0<q<1$ with $\arg q=0$ we have
\begin{align*}
&0\leq -\upomega_{\Wbb_1',\wtd{\fk P}}\#_1\upomega_{\Wbb_2,\fk Q}(\Co w_1\otimes w_1\otimes w_2\otimes \Co w_2)\\
=&- \sum_{n\in\Nbb}\wick{\upomega_{\Wbb_1',\wtd{\fk P}}(\Co w_1\otimes w_1\otimes P_n\c1-)\cdot \upomega_{\Wbb_2,\fk Q}(w_2\otimes \Co w_2\otimes P_n\c1-)}\\
=&- \sum_{n\in\Nbb}q^n\cdot \wick{\upomega_{\Wbb_1',\fk P}(\Co w_1\otimes w_1\otimes P_n\c1-)\cdot \upomega_{\Wbb_2,\fk Q}(w_2\otimes \Co w_2\otimes P_n\c1-)}
\end{align*}
Let $q\rightarrow 0$. We obtain
\begin{align*}
-\upomega_{\Wbb_1',\fk P}(\Co w_1\otimes w_1\otimes \id)\cdot \upomega_{\Wbb_2,\fk Q}(w_2\otimes \Co w_2\otimes \id)\geq0.
\end{align*}
By Prop. \ref{lb32}, $\upomega_{\Wbb_1',\fk P}(\Co w_1\otimes w_1\otimes \id)$ and $\upomega_{\Wbb_2,\fk Q}(w_2\otimes \Co w_2\otimes \id)$ are $>0$. This is impossible.
\end{proof}

\begin{rem}
In Thm. \ref{lb60}, the assumption on the irreducibility of $\Wbb_1\boxtimes\Wbb_2$ is essential. Without this assumption, the proof of Thm. \ref{lb60} will only imply that some (but not all) eigenvalues of $A$ are strictly positive. Namely, it will only imply that $-A\geq0$ does not hold.
\end{rem}

\section{Proof of Theorem \ref{lb50}}\label{lb68}

\subsection{Part I}

We prove Thm. \ref{lb50} by relating the two sides of \eqref{eq69} to those of \eqref{eq76}. We write for simplicity that
\begin{align*}
\Wbb_{12}=\Wbb_1\boxtimes\Wbb_2.
\end{align*}

First, we relate the RHS. We know that $\Uppsi$ is essentially just $\Pi(\cdot,s\im)$ with $\arg s\im=\pi/2$. So we need to relate the conjugate conformal block $\Uppsi^*:\Wbb_1'\otimes\Wbb_2'\otimes\Wbb_{12}\rightarrow\Cbb$ and the adjoint intertwining operator $\Pi^\dagger=\Pi^{c*}=\mc I{\Wbb_2\choose \Wbb_1'\Wbb_{12}}$. Notice that by \eqref{eq77} and \eqref{eq78},
\begin{align*}
\mc U_0(\vartheta_{s\im})=e^{s\im L_1}s^{-2L_0}\qquad(\arg s=0)
\end{align*} 
If we let $\arg(s^{-1}\im)=\pi/2$, then $\mc U_0(\vartheta_{s^{-1}\im})=e^{s^{-1}\im L_1}s^{2L_0}$ equals $s^{2L_0}e^{s\im L_1}$ by \eqref{eq80}. So
\begin{align*}
\mc U_0(\vartheta_{s^{-1}\im})^{-1}\Co=\Co \mc U_0(\vartheta_{s\im})
\end{align*}

For each $w_1\in\Wbb_1,w_2\in\Wbb_2,\nu\in\Wbb_{12}$, noticing \eqref{eq92}, we calculate
\begin{align}
&\bk{\Pi(w_1,s\im)w_2|\nu}=\bk{\Pi(w_1,s\im)w_2,\Co \nu}\xlongequal{\eqref{eq85}}\bk{w_2,\Pi^c(\mc U_0(\vartheta_{s\im})w_1,(s\im)^{-1})\Co \nu}\nonumber\\
=&\bk{\Pi^c(\mc U_0(\vartheta_{s\im})w_1,(s\im)^{-1})\Co \nu|\Co w_2}\xlongequal{\eqref{eq84}}\bk{w_2|\Pi^\dagger(\Co\mc U_0(\vartheta_{s\im}) w_1,s^{-1}\im)\nu}\nonumber\\
=&\bk{w_2|\Pi^\dagger(\mc U_0(\vartheta_{s^{-1}\im})^{-1} \Co w_1,s^{-1}\im)\nu}
\end{align}
Thus,
\begin{align*}
&\Psi^*(\Co w_1\otimes\Co w_2\otimes \nu)=\ovl{\Psi(w_1\otimes w_2\otimes \Co\nu)}=\ovl{\bk{\Pi(w_1,s\im)w_2|\nu}}\\
=&\bk{\Pi^\dagger(\mc U_0(\vartheta_{s^{-1}\im})^{-1} \Co w_1,s^{-1}\im)\nu|w_2}
\end{align*}
Therefore, for each $w_1,\wtd w_1\in\Wbb_1$ and $w_2,\wtd w_2\in\Wbb_2$,
\begin{align}
A\Uppsi\#_1\Uppsi^*(w_1\otimes w_2\otimes\Co\wtd w_1\otimes\Co\wtd w_2)=\bigbk{\Pi^\dagger(\mc U_0(\vartheta_{s^{-1}\im})^{-1} \Co \wtd w_1,s^{-1}\im) A\Pi(w_1,s\im)w_2\big|\wtd w_2}  \label{eq87}
\end{align}

\subsection{Part II}

Next, we compare the LHS. This is the most nontrivial part of the proof. Let
\begin{align*}
\fk A=(\Pbb^1;s\im,-\im;\zeta-s\im,\varpi^{-1})
\end{align*}
Since $\fk A\simeq(\Pbb^1;0,\infty;\zeta,\mu)$ for some M\"obius coordinate $\mu$ at $\infty$, using the change of coordinate formula \eqref{eq62}, we can find a unique $\uptau$ such that
\begin{gather*}
\uptau\in\scr T_{\fk A}^*(\Vbb\otimes\Vbb')\qquad \uptau(v\otimes v')=\bk{a^{L_0}e^{bL_1}v,v'}
\end{gather*}
for some complex numbers $a\neq 0$ and $b$. $\uptau$ enjoys the property that for all $v\in\Vbb$,
\begin{align}
\uptau(v\otimes\Co\id)=\bk{v|\id}\qquad\uptau(\id\otimes\Co v)=\bk{\id|v}  \label{eq82}
\end{align}

\begin{diss}\label{lb52}
We sew $\fk A$ and $\fk Q$ along $-\im\in\fk A$  and $\im\in\fk Q$  (with respect to their local coordinates $\varpi^{-1}$ and $\varpi$). Note that since $0<s<1$, one can choose open disks centered at $-\im\in\fk A$ and at $\im\in\fk Q$ satisfying Asmp. \ref{lb9} such that the product of their radii are greater than $q=1$. Then we have an equivalence
\begin{gather*}
\fk A\#_1\fk Q\simeq (\Pbb^1;s\im,0,\infty;\zeta-s\im,\zeta,1/\zeta)\\
\gamma_0\in \fk A\setminus\{-\im\}\simeq \gamma_0\\
\gamma_2\in  \fk Q\setminus\{\im\}\simeq  \gamma_2
\end{gather*}
Since $\upomega_{\Wbb_2,\fk Q}\in\scr T_{\fk Q}^*(\Wbb_2\otimes\Wbb_2'\otimes\Vbb)$, we obtain the corresponding sewing
\begin{gather*}
\uptau\#_1\upomega_{\Wbb_2,\fk Q}\in\scr T_{\fk A\#_1\fk Q}^*(\Vbb\otimes\Wbb_2\otimes\Wbb_2')\\
\uptau\#_1\upomega_{\Wbb_2,\fk Q}(v\otimes w_2\otimes\Co\wtd w_2)=\sum_{n\in\Nbb}\wick{\uptau(v\otimes P_n \c1-)\cdot \upomega_{\Wbb_2,\fk Q}(w_2\otimes\Co\wtd w_2\otimes P_n \c1 -) }
\end{gather*}
where the two $-$ are contracted. (Cf. Thm. \ref{lb12}.)
\end{diss}

\begin{lm}\label{lb55}
For each $v\in\Vbb,w_2,\wtd w_2\in\Wbb_2$,
\begin{align}
\uptau\#_1\upomega_{\Wbb_2,\fk Q}(v\otimes w_2\otimes\Co\wtd w_2)=\bk{Y_{\Wbb_2}(v,s\im)w_2,\Co\wtd w_2}  \label{eq81}
\end{align}
\end{lm}

\begin{proof}
Choose $w_2,\wtd w_2\in\Wbb_2$. Then $\bk{Y_{\Wbb_2}(\id,s\im)w_2,\Co \wtd w_2}=\bk{w_2|\wtd w_2}$. By \eqref{eq82},
\begin{align*}
&\uptau\#_1\upomega_{\Wbb_2,\fk Q}(\id \otimes w_2\otimes\Co\wtd w_2)=\sum_{n\in\Nbb}\wick{\uptau(\id\otimes P_n \c1 -) \cdot\upomega_{\Wbb_2,\fk Q}(w_2\otimes\Co\wtd w_2\otimes P_n \c1 -)}\\
=&\sum_{n\in\Nbb}\wick{\langle\id, P_n \c1 -\rangle \cdot\upomega_{\Wbb_2,\fk Q}(w_2\otimes\Co\wtd w_2\otimes P_n \c1 -)}=\upomega_{\Wbb_2,\fk Q}(w_2\otimes\Co\wtd w_2\otimes \id)
\end{align*}
which is $\bk{w_2|\wtd w_2}$ by Thm. \ref{lb28}-(a). So \eqref{eq81} is true when $v=\id$. Thus \eqref{eq81} is true for all $v\in\Vbb$ by Prop. \ref{lb20}, since both sides of \eqref{eq81} define elements of $\scr T_{\fk Q}^*(\Wbb_2\otimes\Wbb_2'\otimes\Vbb)$.
\end{proof}

Let
\begin{align*}
\fk B=\Big(\Pbb^1;s^{-1}\im-s\im,0,\infty;\frac 1{\zeta+s\im}+s\im,\zeta,1/\zeta\Big)
\end{align*}
Then we have
\begin{align*}
\frac 1{\zeta+s\im}+s\im=\vartheta_{s^{-1}\im}\circ(\zeta-(s^{-1}\im-s\im))
\end{align*}
Therefore, by Prop. \ref{lb6}, we have a conformal block satisfying ($\forall w_1,\wtd w_1\in\Wbb_1,v\in\Vbb$)
\begin{gather}
\upkappa\in\scr T_{\fk B}^*(\Wbb_1'\otimes\Wbb_1\otimes\Vbb')\nonumber\\
\upkappa(\Co\wtd w_1\otimes w_1\otimes \Co v)=\bk{\Gamma_{\Wbb_1}(\mc U_0(\vartheta_{s^{-1}\im})^{-1}\Co\wtd w_1,s^{-1}\im-s\im)w_1|v} \label{eq86}
\end{gather}
where we choose $\arg(s^{-1}\im-s\im)=\pi/2$ so that \eqref{eq71} is satisfied.

\begin{diss}\label{lb53}
We sew $\fk B$ and $\fk A$ along $\infty\in\fk B$ and $s\im\in\fk A$ (with respect to their local coordinates $1/\zeta$ and $\zeta-s\im$). Since $1/2<s<1$ by \eqref{eq73}, one can choose open disks centered at $\infty\in\fk B$ and $s\im\in\fk A$ satisfying Asmp. \ref{lb9} such that the product of their radii are greater than $q=1$. Then we have an equivalence
\begin{gather*}
\fk B\#_1\fk A\simeq \fk P=\big(\Pbb^1;t\im,-t\im,\infty;\varpi+s\im,\varpi^{-1}-s\im,1/\zeta\big)\\
\gamma_1\in\fk B\setminus\{\infty\}\simeq \varpi(\gamma_1+s\im)\in\fk P\\
\gamma_0\in\fk A\setminus\{s\im\}\simeq \varpi(\gamma_0)\in\fk P
\end{gather*}
With respect to this sewing, and using the canonical identification $\Co\Theta:\Vbb\xrightarrow{\simeq}\Vbb'$ (required in Def. \ref{lb51}), we have
\begin{gather*}
\upkappa\#_1\uptau\in\scr T_{\fk P}^*(\Wbb_1'\otimes\Wbb_1\otimes\Vbb)\\
\upkappa\#_1\uptau(\Co\wtd w_1\otimes w_1\otimes v)=\sum_{n\in\Nbb} \wick{\upkappa(\Co\wtd w_1\otimes w_1\otimes P_n\c1 -)\cdot \uptau(P_n\c1 -,v)}
\end{gather*}
\end{diss}

\begin{lm}\label{lb56}
There exists $\lambda_1>0$ such that
\begin{align}
\upomega_{\Wbb_1',\fk P}=\lambda_1 \cdot \upkappa\#_1\uptau
\end{align}
\end{lm}

\begin{proof}
Choose a non-zero lowest $L_0$-weight vector $w_1\in\Wbb_1$ (so that $L_1w_1=0$). Then by \eqref{eq82} and \eqref{eq83},
\begin{align*}
&\upkappa\#_1\uptau(w_1\otimes\Co w_1\otimes\id)=\sum_{n\in\Nbb}\wick{\langle \Gamma_{\Wbb_1}(\mc U_0(\vartheta_{s^{-1}\im})^{-1}\Co w_1,s^{-1}\im-s\im)w_1,P_n\c1 -\rangle\cdot \uptau(P_n\c1 -,\id)}\\
=&\sum_{n\in\Nbb}\wick{\langle \Gamma_{\Wbb_1}(\mc U_0(\vartheta_{s^{-1}\im})^{-1}\Co w_1,s^{-1}\im-s\im)w_1,P_n\c1 -\rangle\cdot \langle P_n\c1 -,\id\rangle}\\
=&\bk{\Gamma_{\Wbb_1}(\mc U_0(\vartheta_{s^{-1}\im})^{-1}\Co w_1,s^{-1}\im-s\im)w_1,\id}\\
=&\bk{w_1,e^{(s^{-1}\im-s\im)^{-1}L_1}\mc U_0(\vartheta_{s^{-1}\im-s\im})\cdot\mc U_0(\vartheta_{s^{-1}\im})^{-1}\Co w_1 ) }
\end{align*}
By $L_1w_1=0$ and \eqref{eq77}, the above expression equals
\begin{align*}
\bk{w_1,(s^{-1}-s)^{-2L_0}s^{-2L_0}\Co w_1}=(s^{-1}-s)^{-2d}s^{-2d}\lVert w_1\lVert^2> 0
\end{align*}
if we let $L_0w_1=dw_1$. Thus, Prop. \ref{lb32} proves the lemma.
\end{proof}

\subsection{Part III}

\begin{lm}\label{lb54}
The sewing $\upkappa\#_1\uptau\#_1\upomega_{\Wbb_2,\fk Q}$ converges absolutely. More precisely, for each $w_1,\wtd w_1\in\Wbb,w_2,\wtd w_2\in\Wbb_2$, the following double series converges absolutely:
\begin{align*}
&\upkappa\#_1\uptau\#_1\upomega_{\Wbb_2,\fk Q}(\Co\wtd w_1\otimes w_1\otimes w_2\otimes\Co\wtd w_2)\\
=&\sum_{n,k\in\Nbb} \wick{\upkappa(\Co\wtd w_1\otimes w_1\otimes P_n\c1 -)\cdot \uptau(P_n\c1 -,}\wick{P_k \c1 -)\cdot \upomega_{\Wbb_2,\fk Q}(w_2\otimes\Co\wtd w_2\otimes P_k \c1 -) }
\end{align*}
\end{lm}

\begin{proof}
This follows from Thm. \ref{lb12}, as long as we can check that Asmp. \ref{lb9} is satisfied, and that the $r_j,\rho_j$ as in Def. \ref{lb8} can be chosen such that $r_1\rho_1>1,r_2\rho_2>1$ (because the sewing parameters are $q_1=q_2=1$). Note that we are sewing $\fk B\sqcup\fk A\sqcup \fk Q$ simultaneously along two pairs of points: the first pair is $\infty\in\fk B$ and $s\im\in\fk A$, and the second one is $-\im\in\fk A$ and $\im\in\fk Q$. The local coordinates for these points are respectively $\xi_1=1/\zeta,\mu_1=\zeta-s\im,\xi_2=\varpi^{-1},\mu_2=\varpi$. (See Discussions \ref{lb53} and \ref{lb52}.)

As for the first pair of points,  we choose neighborhoods $W_1'\ni \infty$ and $W_1''\ni s\im$ as follows. Note that if $1/\sqrt 2<s<1$ (cf. \eqref{eq73}) then $0<(s^{-1}-s)<1/\sqrt 2$. Let $W_1'\subset\fk B$ and $W_1''\subset\fk A$ be
\begin{gather*}
W_1'=\eta_1^{-1}(\Dbb_{\sqrt 2})=\{z\in\Pbb^1:|1/z|<\sqrt 2\}\qquad\subset\fk B\\
W_1''=\mu_1^{-1}(\Dbb_{as})=\{z\in\Pbb^1:|z-s\im|<as\}\qquad\subset\fk A
\end{gather*}
where $0<a<1$ is such that $\sqrt 2as>1$. Such $a$ exists because $\sqrt 2s>1$. As for the second pair, we choose neighborhoods $W_2'\ni -\im,W_2''\ni \im$ to be
\begin{gather*}
W_2'=\eta_2^{-1}(\Dbb_b)=\{z\in\Pbb^1:|\varpi^{-1}(z)|<b\}\qquad\subset\fk A\\
W_2''=\mu_2^{-1}(\Dbb_1)=\{z\in\Pbb^1:|\varpi(z)|<1\}\qquad\subset\fk Q
\end{gather*}
where $b>1$ is such that $W_1''$ and $W_2'$ (which are both open discs in $\fk A$) do not intersect. Such $b$ exists because the closure of $W_1''$ (which is $\{z\in\Pbb^1:|z-s\im|\leq as\}$) does not intersect the closure of $\eta_2^{-1}(\Dbb_1)$ (which is $\{z\in\Pbb^1:\Imag(z)\leq 0\}$).

Now, Asmp. \ref{lb9} is satisfied (namely, the disks $W_1,W_1'',W_2',W_2''$ are mutually disjoint, and each disk contains only one one of the marked points of $\fk B\sqcup\fk A\sqcup\fk P$), and we have $r_1\rho_1>1,r_2\rho_2>1$ where $r_1=\sqrt 2,\rho_1=as,r_2=b,\rho_2=1$.
\end{proof}

\begin{proof}[\textbf{Proof of Thm. \ref{lb50}}]
By Lem. \ref{lb56}, for any $w_1,\wtd w_1\in\Wbb_1,w_2,\wtd w_2\in\Wbb_2$ we have
\begin{align*}
&\upomega_{\Wbb_1',\fk P}\#_1\upomega_{\Wbb_2,\fk Q}(\Co\wtd w_1\otimes w_1\otimes w_2\otimes\Co\wtd w_2)\\
=&\sum_{k\in\Nbb} \upomega_{\Wbb_1',\fk P}(\Co\wtd w_1\otimes w_1\otimes \wick{P_k \c1 -)\cdot \upomega_{\Wbb_2,\fk Q}(w_2\otimes \Co\wtd w_2\otimes P_k \c1-)}  \\
=&\lambda_1\sum_{n,k\in\Nbb} \wick{\upkappa(\Co\wtd w_1\otimes w_1\otimes P_n\c1 -)\cdot \uptau(P_n\c1 -,}\wick{P_k \c1 -)\cdot \upomega_{\Wbb_2,\fk Q}(w_2\otimes\Co\wtd w_2\otimes P_k \c1 -) }
\end{align*}
which converges absolutely as a double series by Lem. \ref{lb54}. (In particular, the two infinite sums commute.) By Lem. \ref{lb55} and \eqref{eq86}, this expression equals
\begin{align*}
&\lambda_1\sum_{n\in\Nbb} \wick{\upkappa(\Co\wtd w_1\otimes w_1\otimes P_n\c1 -)\cdot \langle Y_{\Wbb_2}(P_n \c1-,s\im)w_2,\Co\wtd w_2 \rangle}\\
=&\lambda_1\sum_{n\in\Nbb} \wick{\langle \Gamma_{\Wbb_1}(\mc U_0(\vartheta_{s^{-1}\im})^{-1}\Co\wtd w_1,s^{-1}\im-s\im)w_1, P_n \c1- \rangle\cdot \langle Y_{\Wbb_2}(P_n \c1-,s\im)w_2,\Co\wtd w_2 \rangle}\\
=&\lambda_1 \bigbk{Y_{\Wbb_2}\big( \Gamma_{\Wbb_1}(\mc U_0(\vartheta_{s^{-1}\im})^{-1}\Co\wtd w_1,s^{-1}\im-s\im)w_1,s\im\big)w_2\big|\wtd w_2}
\end{align*}
By \eqref{eq69}, this expression equals $\lambda_1$ times the RHS of $\eqref{eq87}$. The proof is completed.
\end{proof}

\section{Application to orbifold VOAs}

Recall that if  $\Vbb$ is CFT-type, self-dual (e.g. when $\Vbb$ is unitary), $C_2$-cofinite, and  rational (namely, if $\Vbb$ is \textbf{strongly rational}), then $\Mod(\Vbb)$ of $\Vbb$-modules is a modular tensor category \cite{Hua08}. 

\begin{df}
If $\Vbb$ is a CFT-type, unitary, $C_2$-cofinite, and rational VOA, we say that $\Vbb$ is \textbf{completely unitary} if all irreducible $\Vbb$-modules are unitarizable, and if for each pair of irreducible $\Vbb$-modules $\Wbb_1,\Wbb_2$, the fusion product $\Wbb_1\boxtimes\Wbb_2$ is positive (recall Def. \ref{lb50}).
\end{df}

\begin{rem}
Note that if $\Wbb$ is an irreducible $\Vbb$-module, then the unitary structures on $\Wbb$ are clearly unique up to positive scalar multiplications. Also, it was shown in \cite{Gui19b} that if $\Vbb$ is completely unitary then the category $\Mod^\uni(\Vbb)$ of unitary $\Vbb$-modules is a unitary modular tensor category.
\end{rem}

\subsection{General results}

In this subsection, we assume that $\Vbb$ is a CFT-type unitary $C_2$-cofinite and rational VOA. Let $G$ be a finite group of unitary automorphisms of $\Vbb$. (Thus, each $g\in G$ acts unitarily on $\Vbb$, and satisfies $gY(v,z)=Y(gv,z)g$ for all $v\in\Vbb$.) It is natural to ask whether the fixed point subalgebra $\Vbb^G=\{v\in V:gv=v,\forall g\in G\}$ is completely unitary. (Certainly $\Vbb^G$ is unitary.) 

Notice that by \cite{Miy15,CM16}, if $G$ is solvable, then $\Vbb^G$ is $C_2$-cofinite and rational. Then, by \cite{Hua09}, $\Mod(\Vbb)$ is a modular tensor category.

\begin{df}
We say that $\Vbb$ is a \textbf{unitary holomorphic VOA} if $\Vbb$ is CFT-type, unitary, $C_2$-cofinite, rational, and if every irreducible $\Vbb$-module is isomorphic to $\Vbb$.
\end{df}

\begin{thm}\label{lb61}
Let $\Vbb$ be a unitary holomorphic VOA, and let $G$ be a finite \emph{cyclic} group (i.e. $G\simeq\Zbb_n$ for some $n$) of unitary automorphisms of $\Vbb$. Assume that every irreducible $\Vbb^G$-module is unitarizable. Then $\Vbb^G$ is completely unitary.
\end{thm}

We will discuss the non-cyclic or even non-abelian case in future works.

\begin{proof}
Since $G$ is cyclic, $\Mod(\Vbb)$ is pointed (namely, the fusion product of any two irreducible $\Vbb$-modules $\Wbb_1\boxtimes\Wbb_2$ is irreducible). This fact is due to \cite[Prop. 5.6]{vEMS20}. It also follows from the general fact that if $\Vbb$ is holomorphic and if $\Vbb^G$ is $C_2$-cofinite and rational (which is automatic when $G$ is solvable \cite{Miy15,CM16}) then $\Mod(\Vbb)$ is isomorphic to the twisted Drinfeld double $D^\omega(G)$ for some $\omega\in H^3(G,\Cbb^\times)$ (cf. \cite[Thm. 6.2]{DNR21},  or also \cite{Kir04,McR21} for a general discussion of the relationship between the categories of $\Vbb^G$-modules and twisted $\Vbb$-modules without assuming that $\Vbb$ is holomorphic), together with the fact that if $G$ is cyclic then $D^\omega(G)$ is pointed because $H^2(G,\Cbb^\times)$ is trivial (\cite[Cor. 3.6]{MN01}). Choose any unitary structures on $\Wbb_1$ and $\Wbb_2$. Then the fusion product $\Wbb_1\boxtimes\Wbb_2$ is positive by Thm. \ref{lb60}. So $\Vbb^G$ is completely unitary.
\end{proof}

In practice, one can show that all irreducible $\Vbb^G$-modules are unitarizable by showing that all irreducible $G$-twisted modules are unitarizable. For the readers' convenience, we recall the following definition of twisted modules when $\Vbb^G$ is $C_2$-cofinite. This definition can be easily translated to the tensor-categorical language as in \cite{Kir02,Kir04,McR21}. For the definition for general VOAs, see for instance \cite[Sec. 7]{DL96}.

\begin{df}
Let $g\in G$. A \textbf{$g$-twisted module} denotes $(\mc W,Y_{\mc W})$, where $\mc W$ is a $\Vbb^G$-module, and $Y_{\mc W}$ is type $\mc W\choose\Vbb~\mc W$-intertwining operator of $\Vbb^G$ satisfying the following conditions:
\begin{enumerate}[label=(\arabic*)]
\item For each $v_1,v_2\in\Vbb$ and $z_1,z_2\in\Cbb^\times$ satisfying $0<|z_1-z_2|<|z_2|<|z_1|$ and $\arg(z_1-z_2)=\arg z_2=\arg z_1$, we have the fusion relation
\begin{align}
Y_{\mc W}(v_1,z_1)Y_{\mc W}(v_2,z_2)=Y_{\mc W}\big(Y(v_1,z_1-z_2)v_2,z_2 \big)
\end{align}
understood in the same way as \eqref{eq69}, namely, for each $w\in\mc W,w'\in\mc W'$ we have
\begin{align*}
\sum_{r\in\Cbb}\bigbk{Y_{\mc W}(v_1,z_1)P_rY_{\mc W}(v_2,z_2)w,w'}=\sum_{n\in\Nbb}\bigbk{Y_{\mc W}\big(P_nY(v_1,z_1-z_2)v_2,z_2 \big)w,w'}
\end{align*}
where both sides converge absolutely by \cite{Hua05} or Thm. \ref{lb12}.
\item For each $v\in\Vbb$ and $z\in\Cbb^\times$ with chosen $\arg z$, if we let $e^{2\im\pi}z$ have argument $2\pi+\arg z$, then
\begin{align}
Y_{\mc W}(v,z)=Y_{\mc W}(gv,e^{2\im\pi}z)  \label{eq88}
\end{align}
\end{enumerate}
\end{df}

Similar to untwisted modules, we abbreviate $Y_{\mc W}$ to $Y$ when the context is clear.

\begin{rem}
The above definition of twisted modules is exactly the same as that in \cite{McR21}, and also applies to the general case that $G$ is a finite group of (non-necessarily unitary) automorphisms of a (non-necessarily unitary) VOA $\Vbb$ satisfying that $\Vbb^G$ is $C_2$-cofinite. It also agrees with the usual definition using (algebraic) Jacobi identity, which does not require $\Vbb^G$ to be $C_2$-cofinite.  See \cite{Hua10}. 
\end{rem}

\begin{rem}
The above definition depends only on the automorphism $g$ but not on the group $G$. In fact, if we let $H=\bk{g}$ be the finite cyclic group generated by $g$, then the definition of $g$-twisted modules using $H$ is clearly equal to the one using $G$.
\end{rem}

\begin{df}
Let $g$ be a unitary automorphism of $\Vbb$ with finite order. A $g$-twisted $\Vbb$-module $\mc W$, together with an inner product on $\mc W$, is called a \textbf{unitary $g$-twisted $\Vbb$-module}, if for each $v\in\Vbb,w_1,w_2\in\mc W$ and $z\in\Cbb^\times$ with chosen $\arg z$ we have
\begin{align}
\bk{Y_{\mc W}(v,z)w_1|w_2}=\bk{w_1|Y_{\mc W}(e^{\ovl zL_1}(-\ovl z^{-2})^{L_0}\Theta v,\ovl z^{-1})w_2  }
\end{align} 
Here, we assume $\arg\ovl z=-\arg z$.
\end{df}

By $G$-twisted $\Vbb$-module, we mean a $g$-twisted $\Vbb$-module where $g\in G$.

\begin{pp}\label{lb66}
Let $\Vbb$ be a CFT-type unitary VOA, and let $G$ be a finite unitary automorphism group of $\Vbb$ such that $\Vbb^G$ is an (automatically unitary) $C_2$-cofinite rational VOA. Suppose that every irreducible $G$-twisted $\Vbb$-module is unitarizable (i.e. admits a (necessarily unique up to scalar multiplications) unitary structure). Then every irreducible $\Vbb^G$-module is unitarizable.
\end{pp}

Therefore, the assumption in Thm. \ref{lb61} on the unitarizable of irreducible $\Vbb^G$-modules is satisfied if every $G$-twisted $\Vbb$-module is unitarizable. 

\begin{proof}
This follows from the fact that every irreducible $\Vbb^G$-module is a submodule of a irreducible $G$-twisted $\Vbb$-module (considered as a $\Vbb^G$-module), cf. \cite{DRX17} or \cite{McR21}.
\end{proof}

\begin{rem}
Many examples satisfy the assumptions of Thm. \ref{lb61}. For instance, Lam showed in \cite{Lam23} that if $V$ is an even lattice VOA $V_L$, then for many finite-order unitary automorphisms $g$ of $V$ (including all standard lifts from the isometries of the lattice $L$), all $g^n$-twisted $\Vbb$-modules are unitary (where $n\in\Zbb$). Thus, if $L$ is unimodular (i.e. self-dual) so that $V_L$ is holomorphic, we know from Thm. \ref{lb61} that $V_L^{\bk{g}}$ is completely unitary.
\end{rem}

In the next subsection, we show that another large class of orbifold VOAs satisfy the assumptions of Thm. \ref{lb61}.

\subsection{Examples: permutation orbifold VOAs}

Let $\Vbb$ be unitary. Let $k\in\Zbb_+$. Choose a permutation $g\in\Aut\{1,2,\dots,k\}$. Then $g$ is naturally an automorphism of $\Vbb^{\otimes k}$:
\begin{align*}
g(v_1\otimes\cdots\otimes v_k)=v_{g^{-1}(1)}\otimes\cdots\otimes v_{g^{-1}(k)}
\end{align*}
Note that the PCT operator $\Theta$ on $\Vbb^{\otimes k}$ is $\Theta\otimes\cdots\otimes\Theta$.

We first consider the special case that $g=(12\dots k)$. Then irreducible $g$-twisted modules are classified in \cite{BDM02}. We briefly recall the construction. (See also \cite[Sec. 10]{Gui21}. )

Choose a $\Vbb$-module $\Wbb$. Then on the same vector space $\Wbb$ there is a canonical $g$-twisted module structure. To avoid confusions, we write $\Wbb$ as $\mc W$ when we consider it as a $g$-twisted modules. So $\Wbb$ and $\mc W$ are equal as vector spaces. We write the vertex operation of $\mc W$ as $Y_{\mc W}^g$ or simply $Y^g$. The $g$-twisted module $(\mc W,Y_{\mc W}^g)$ is uniquely determined by the fact that if $u\in\Vbb$ then
\begin{align*}
Y^g_{\mc W}(u\otimes\id\otimes\cdots\otimes \id,z)=Y_\Wbb(\mc U(\delta_{k,z})u,\sqrt[k]z)
\end{align*}
where $\arg\sqrt[k]z=\frac 1k\arg z$ and $\delta_{k,z}\in\Gbb$ is defined by
\begin{equation*}
\delta_{k,z}(t)=(z+t)^{\frac 1k}-z^{\frac 1k}
\end{equation*}
See \cite[Thm. 3.9]{BDM02}. 

\begin{rem}\label{lb63}
Since $\Vbb^{\otimes k}$ is generated by vectors of the form $u\otimes\id\otimes\cdots\otimes \id$ and their permutations by $g^j$ (where $j\in\Zbb$), the uniqueness mentioned above is an easy consequence of \eqref{eq88} and the Jacobi identity for twisted vertex operators (cf. \cite[(3.4)]{DLM98}, or \cite[Rem. 10.1]{Gui21} if one prefers to expand the delta functions). 
\end{rem}

\begin{lm}\label{lb64}
Let $g=(1,\dots,k)$. Assume that $\Wbb$ is a unitary $\Vbb$-module, and define the inner product on $\mc W$ to be the same as that of $\Wbb$. Then the $g$-twisted $\Vbb$-module $(\mc W,Y_{\mc W}^g)$ is unitary.
\end{lm}

\begin{proof}
We need to show that for each $\vbf\in\Vbb^{\otimes k}$ and $w_1,w_2\in\mc W$,
\begin{align}
\bk{Y_{\mc W}^g(\vbf,z)w_1|w_2}=\bk{w_1|Y_{\mc W}^g(e^{\ovl zL_1}(-\ovl z^{-2})^{L_0}\Theta \vbf,\ovl z^{-1})w_2}  \label{eq90}
\end{align}
Equivalently (cf. the end of Rem. \ref{lb62}), if we let $\Co:\Wbb\rightarrow\Wbb'$ denote the canonical antiunitary map, and denote its inverse also by $\Co$, then we need to prove
\begin{align}
Y_{\mc W}^g(\vbf,z)=\Co Y_{\mc W'}^{g^{-1}}(\Theta \vbf,\ovl z)\Co   \label{eq89}
\end{align}
where $Y_{\mc W'}^{g^{-1}}$ is the contragredient intertwining operator of the type $\mc W\choose\Vbb~\mc W$ intertwining operator $Y_{\mc W}^g$ of $\Vbb^{\bk g}$. By \cite[Prop. 3.3]{Hua18}, if we let $\mc W'$ be equal to $\Wbb'$ as vector spaces, then $(\mc W',Y_{\mc W'}^{g^{-1}})$ is a $g^{-1}$-twisted $\Vbb^{\otimes k}$-module. From this, it is easy to see that $(\mc W,\wtd Y_{\mc W}^g)$ is a $g$-twisted module if we define $\wtd Y_{\mc W}^g(\vbf,z)= \Co Y_{\mc W'}^{g^{-1}}(\Theta \vbf,\ovl z)\Co$. Therefore, since $g$-twisted vertex operations are determined by there values on $\vbf=u\otimes\id\otimes\cdots\otimes\id$ for all $u\in\Vbb$ (as discussed in Rem. \ref{lb63}), it suffices to prove \eqref{eq89} (equivalently, to prove \eqref{eq90}) whenever $\vbf=u\otimes\id\otimes\cdots\otimes\id$.

Let $\vbf=u\otimes\id\otimes\cdots\otimes\id$. Recall (cf. Exp. \ref{lb3}) that $\vartheta_z(t)=(z+t)^{-1}-z^{-1}$ and $\mc U(\vartheta_z)=e^{zL_1}(-z^{-2})^{L_0}$. Then
\begin{align*}
\delta_{k,z^{-1}}\circ\vartheta_z(t)=(z+t)^{-\frac 1k}-z^{-\frac 1k}=\vartheta_{z^{\frac 1k}}\circ\delta_{k,z}(t).
\end{align*}
Therefore, by Thm. \ref{lb1} we have $\mc U(\delta_{k,z^{-1}})\mc U(\vartheta_z)=\mc U(\vartheta_{z^{1/k}})\mc U(\delta_{k,z})$. So
\begin{align*}
Y_{\mc W}^g(e^{zL_1}(-z)^{L_0} \vbf, z^{-1})=Y_\Wbb(\mc U(\delta_{k,z^{-1}})\mc U(\vartheta_z)u,z^{-1/k})=Y_\Wbb(\mc U(\vartheta_{z^{1/k}})\mc U(\delta_{k,z})u,z^{-1/k}).
\end{align*}
Replace $z$ with $\ovl z$ and $u$ with $\Theta u$, we get (recall \eqref{eq21} and \eqref{eq10})
\begin{align*}
&\bk{w_1|Y_{\mc W}^g(e^{\ovl zL_1}(-\ovl z^{-2})^{L_0}\Theta \vbf,\ovl z^{-1})w_2}=\bk{w_1|Y_\Wbb(\mc U(\vartheta_{\ovl z^{1/k}})\mc U(\delta_{k,\ovl z})\Theta u,\ovl z^{-1/k})w_2}\\
=&\bk{w_1|Y_\Wbb(\mc U(\vartheta_{\ovl z^{1/k}})\Theta\mc U(\delta_{k,z}) u,\ovl z^{-1/k})w_2}=\bk{Y_\Wbb(\mc U(\delta_{k,z})u,z^{1/k})w_1|w_2}
\end{align*}
which equals the LHS of \eqref{eq90}. This proves \eqref{eq90}.
\end{proof}

Now consider an arbitrary permutation $g\in\Aut\{1,2,\dots,k\}$. Then $g$ is a product of disjoint cycles $g=g_1\cdots g_l$. Let $k_j$ be the order of $g_j$. So $k_1+\cdots+k_l=k$. Let $s_0=0$ and $s_j=k_1+\cdots+k_j$ if $1\leq j\leq l$. Without loss of generality, we may assume
\begin{align*}
g_j=(s_{j-1}+1,s_{j-1}+2,\dots,s_j)
\end{align*}
For each $1\leq j\leq l$, choose a $\Vbb$-module $\Wbb_j$. By \cite[Thm. 7.10]{BDM02}, we have a $g$-twisted $\Vbb^{\otimes k}$-module on $\mc W=\Wbb_1\otimes\cdots\otimes\Wbb_l$, where the vertex operation $Y_{\mc W}^g$ is determined by
\begin{align} \label{eq91}
\begin{aligned}
&Y_{\mc W}^g(v_1\otimes\cdots\otimes v_k,z)(w_1\otimes\cdots\otimes w_l)\\
=&\bigotimes_{j=1}^l Y_{\Wbb_j}^{g_j}(v_{s_{j-1}+1}\otimes v_{s_{j-1}+2}\otimes\cdots\otimes v_{s_j},z)w_j
\end{aligned}
\end{align}
By Lem. \ref{lb64}, one easily obtains:

\begin{pp}\label{lb65}
Let $g\in\Aut\{1,\dots,k\}$, and use the above notations. Assume that each $\Wbb_j$ is a unitary $\Vbb$-module, and define the inner product on $\mc W$ to be the natural one on $\Wbb_1\otimes\cdots\otimes\Wbb_l$. Then the $g$-twisted $\Vbb$-module $(\mc W,Y_{\mc W}^g)$ defined by \eqref{eq91} is unitary.
\end{pp}

\begin{thm}
Let $\Vbb$ be a unitary holomorphic VOA, let $k\in\Zbb_+$, and let $G$ be a cyclic abelian subgroup of $\Aut\{1,2,\dots,k\}$. Then the fixed point subalgebra $(V^{\otimes k})^G$ is completely unitary.
\end{thm}

\begin{proof}
By \cite[Thm. 7.10]{BDM02}, every irreducible $G$-twisted $\Vbb^{\otimes k}$-module is of the form $(\mc W,Y_{\mc W}^g)$ (as defined in \eqref{eq91}, where $\Wbb_1=\dots=\Wbb_l=\Vbb$ because $\Vbb$ is holomorphic), which is unitarizable by Prop. \ref{lb65}. Thus, by Prop. \ref{lb66}, every irreducible $(\Vbb^{\otimes k})^G$-module is unitarizable. This finishes the proof, thanks to Thm. \ref{lb61}.
\end{proof}

\section{Application to simple current extensions}

\begin{df}
If $\Vbb$ is a VOA, then a \textbf{VOA extension} of $\Vbb$ is a VOA $\Ubb$ containing $\Vbb$ as a subspace such that the vertex operation of $\Ubb$ restricts to that of $\Vbb$, and that $\Ubb$ and $\Vbb$ share the same vacuum vector and conformal vector. 

If $\Vbb$ is a unitary VOA, a VOA extension $\Ubb$ of $\Vbb$ is called \textbf{preunitarizable} if $\Ubb$ is unitarizable as a $\Vbb$-module.  \hfill\qedsymbol
\end{df}

Given a preunitarizable VOA extension $\Ubb$ of a unitary VOA $\Vbb$, the following proposition clarifies the relationship between CFT-type, simpleness, and Haploidness (in case $\Mod(\Vbb)$ is good enough such that $\Ubb$ is an algebra object in $\Mod(\Vbb)$ in the sense of \cite{KO02,CKM17}) of $\Ubb$.

\begin{pp}\label{lb79}
Let $\Vbb$ be a unitary VOA. Let $\Ubb$ be a preunitarizable VOA extension of $\Vbb$. The following are true.
\begin{enumerate}[label=(\arabic*)]
\item $\Ubb$ is of CFT-type if and only if $\dim\Hom_\Vbb(\Vbb,\Ubb)=1$.
\item If $\Ubb$ is a simple VOA, then $\Ubb$ is of CFT-type. 
\item Assume that $\Ubb$ is a completely reducible $\Ubb$-module (e.g. when $\Ubb$ is unitary, or when $\Ubb$ is rational). If $\Ubb$ is of CFT-type, then $\Ubb$ is a simple VOA.
\end{enumerate}
\end{pp}

Note that we are not assuming $\Vbb$ to be of CFT-type. When $\Ubb=\Vbb$, (2) and (3) are \cite[Prop. 5.3]{CKLW18}.

\begin{proof}
Thm. \ref{lb75} implies
\begin{align}
\dim\End_\Ubb(\Ubb)=\dim\Hom_\Vbb(\Vbb,\Ubb)=\dim\Ubb(0)
\end{align}
So $\Ubb(0)=\Cbb\id$ iff $\dim\Hom_\Vbb(\Vbb,\Ubb)=1$. This proves (1). If $\Ubb$ is a simple VOA, then $\dim\End_\Ubb(\Ubb)=1$. This proves (2). Assume that $\Ubb$ is a completely reducible $\Ubb$-module. If $\Ubb$ is of CFT-type, then $\dim\End_\Ubb(\Ubb)=1$. So $\Ubb$ is simple (cf. Rem. \ref{lb83}). This proves (3).
\end{proof}

Now we assume that $\Vbb$ is CFT-type, self-dual, $C_2$-cofinite, and rational (i.e. $\Vbb$ is \textbf{strongly-rational}). In particular, $\Vbb$ is simple \cite[Prop. 4.6-(iv)]{CKLW18}. Then $\Mod(\Vbb)$ is rigid modular, and so every $\Vbb$-module $\Wbb$ has a categorical dual, which is isomorphic to the contragredient $\Wbb'$, cf. \cite{Hua08}. Recall that a \textbf{simple current} of $\Vbb$ is a $\Vbb$-module $\Wbb$ which is invertible, i.e. there is a $\Vbb$-module $\Mbb$ such that $\Wbb\boxtimes\Mbb\simeq\Vbb$. In that case $\Mbb$ must be the categorical dual of $\Wbb$, and hence $\Mbb\simeq\Wbb'$. Simple currents are clearly irreducible modules. The class of simple currents is clearly closed under taking $\boxtimes$ and taking contragredient.

A (finite) \textbf{simple current extension} of $\Vbb$ is a VOA extension $\Ubb$ which is a simple VOA and which, as a $\Vbb$-module, is a finite direct sum of simple currents of $\Vbb$. 

\begin{rem}\label{lb82}
Let $\Ubb$ be a simple current extension of $\Vbb$. Let $\mc G$ be the set of  equivalence classes of all these currents. Then it is well-known that $\mc G$ is closed under taking $\boxtimes$ and taking contragredient. (Quick proof: The vertex operation $Y^\Ubb$ of $\Ubb$ is a type $\Ubb\choose \Ubb\Ubb$ intertwining operator of $\Vbb$. Choose simple currents $W_1,W_2$ appearing in $\Vbb$. Then $Y^\Ubb$ restricts to a type $\Ubb\choose \Wbb_1\Wbb_2$ intertwining operator which is non-zero by Prop. \ref{lb20} and that $\Ubb$ is simple. So it restricts to a type $\Wbb_3\choose \Wbb_1\Wbb_2$ one for some simple current $\Wbb_3$ inside $\Ubb$. Hence $\Wbb_1\boxtimes\Wbb_2\simeq\Wbb_3$ is in $\Ubb$. Similarly, if $\Wbb$ is a simple current in $\Ubb$, then $Y^\Ubb$ restricts to a non-zero type $\Vbb\choose \Wbb\Ubb$ intertwining operator, and hence restricts to a non-zero type $\Vbb\choose \Wbb\Mbb$ one where $\Mbb$ is a simple current in $\Ubb$. Then $\Wbb\boxtimes\Mbb\simeq\Vbb$ and hence $\Wbb'\simeq\Mbb$ is in $\Ubb$.)
\end{rem}

Note that if $\Vbb$ is also unitary, then by Prop. \ref{lb79}, a preunitarizable simple current extension of $\Vbb$ is of CFT-type.

We need an auxiliary result:

\begin{pp}\label{lb81}
Let $\Vbb$ be a CFT-type, unitary, $C_2$-cofinite, and rational VOA. Let $\mc C$ be a full abelian $C^*$-subcategory of the $C^*$-category of unitary $\Vbb$-modules. For each $\Wbb_1,\Wbb_2\in\mc C$, assume that the $\Vbb$-modules $\Wbb_1\boxtimes\Wbb_2$ and $\Wbb_1'$ are  isomorphic to some objects in $\mc C$, assume that the fusion product $\Wbb_1\boxtimes\Wbb_2$ is positive, and choose the unitary $\Vbb$-module structure on $\Wbb_1\boxtimes\Wbb_2$ under which the operator $A$ in \eqref{eq69} is $\id$. Then $\mc C$, together with $\boxtimes$ and the associators and the the braiding $\ss$ and the unitors of $\Mod(\Vbb)$, is a unitary ribbon fusion category.
\end{pp}

Roughly speaking, by saying $\mc C$ is a full abelian $C^*$-subcategory, we mean that $\mc C$ is closed under taking direct sums and unitary submodules, and that the morphims between objects are homomorphisms of $\Vbb$-modules.

\begin{proof}
If all $\Vbb$-modules are unitarizable, and if $\mc C$ is the category of all unitary $\Vbb$-modules, then this proposition is \cite[Thm. 7.9]{Gui19b}. In the general case, the proof is the same as that the $\Mod^{\mathrm u}_{\mc F^\boxtimes}(V)$ in \cite[Thm. 7.8]{Gui19b} is a unitary ribbon fusion category.
\end{proof}

\begin{thm}\label{lb80}
Let $\Vbb$ be a CFT-type, unitary, $C_2$-cofinite, and rational VOA. Let $\Ubb$ be a preunitarizable simple current extension of $\Vbb$. Then the inner product on $\Vbb$ can be extended uniquely to an inner product $\bk{\cdot|\cdot}$ on $\Ubb$ such that $\Ubb$ is a unitary VOA.
\end{thm}

We call such $(\Ubb,\bk{\cdot|\cdot})$ (satisfying the last sentence of Thm. \ref{lb80}) a \textbf{unitary VOA extension} of $\Vbb$. In the special case that $\Ubb$ is a direct sum of two simple currents, this theorem was proved by \cite[Thm. 3.3]{DL14} under some small additional assumption; see also \cite[Thm. 3.11]{CGH23} for a related result in the case of vertex operator superalgebras.

\begin{proof}
Let $\mc G$ be as in Rem. \ref{lb82}. Let $\mc C$ be the $C^*$-category of unitary $\Vbb$-modules that are (equivalent to) finite direct sums of unitary irreducible $\Vbb$-modules in $\mc G$. If the $*$-structure is forgotten, $\mc C$ is a full abelian subcategory of $\Mod(\Vbb)$. Moreover, Rem. \ref{lb82} shows that $\mc C$ is closed under $\boxtimes$ and taking contragredient modules. Therefore, by Thm. \ref{lb60}, $\mc C$ satisfies the assumptions in Prop. \ref{lb81} and hence is a unitary ribbon fusion category. Thus, we can use the same argument as in the proof of \cite[Thm. 4.7]{CGGH23} to prove that $\Ubb$ is uniquely a unitary VOA extension by invoking Thm. 3.2 and Thm. 3.9 of \cite{CGGH23}. (Note that one needs the fact that $\Ubb$, as a Haploid algebra in $\mc C$, is rigid. This follows from \cite[Lem. 1.20]{KO02} because $\Ubb$ as an object in $\mc C$ has strictly positive quantum dimension as $\mc C$ is a unitary fusion category.)
\end{proof}

\newpage

\printindex

\subsection*{Declarations}
Conflict of interest: The author has no competing interests to declare that are relevant to the content of this article.

\noindent {\small \sc Yau Mathematical Sciences Center, Tsinghua University, Beijing, China.}

\noindent {\textit{E-mail}}: binguimath@gmail.com\qquad bingui@tsinghua.edu.cn
\end{document}